\documentclass[reqno]{amsart}

\usepackage{a4wide}
\usepackage[usenames,dvipsnames]{xcolor}
\usepackage{mathrsfs}
\usepackage{mathtools}
\usepackage{tikz-cd}
\usepackage{amsmath}
\usepackage{amssymb}
\usepackage{bbm}
\usepackage{aligned-overset}
\numberwithin{equation}{section}
\usepackage[colorlinks,citecolor=green,linkcolor=red]{hyperref}
\usepackage{hyphenat}
\usepackage[latin1]{inputenc}

\newcommand{\B}{\mathbb{B}}
\newcommand{\V}{\mathbb{V}}
\newcommand{\N}{\mathbb{N}}
\newcommand{\R}{\mathbb{R}}

\newcommand{\mm}{{\mathfrak m}}

\newcommand{\sfd}{{\sf d}}
\renewcommand{\d}{{\mathrm d}}

\newcommand{\restr}[1]{\lower3pt\hbox{\(|_{#1}\)}}

\newcommand{\nchi}{{\raise.3ex\hbox{\(\chi\)}}}
\newcommand{\Der}{{\rm Der}}
\newcommand{\1}{\mathbbm 1}
\newcommand{\fr}{\penalty-20\null\hfill\(\blacksquare\)}

\newcommand{\ppi}{\boldsymbol{\pi}}
\newcommand{\sppi}{{\mbox{\scriptsize\boldmath\(\pi\)}}}

\newcommand{\Lip}{{\rm Lip}}
\newcommand{\lip}{{\rm lip}}
\renewcommand{\div}{{\rm div}}

\newtheorem{theorem}{Theorem}[section]
\newtheorem{corollary}[theorem]{Corollary}
\newtheorem{lemma}[theorem]{Lemma}
\newtheorem{proposition}[theorem]{Proposition}
\newtheorem{definition}[theorem]{Definition}
\newtheorem{example}[theorem]{Example}
\newtheorem{remark}[theorem]{Remark}

\title[Derivations and Sobolev functions on extended metric-measure spaces]{Derivations and Sobolev functions \\ on extended metric-measure spaces}

\author{Enrico Pasqualetto}
\address{Department of Mathematics and Statistics,
P.O.\ Box 35 (MaD), FI-40014 University of Jyvaskyla}
\email{enrico.e.pasqualetto@jyu.fi}

\author{Janne Taipalus}
\address{Department of Mathematics and Statistics,
P.O.\ Box 35 (MaD), FI-40014 University of Jyvaskyla}
\email{janne.m.m.taipalus@jyu.fi}

\begin{document}
\date{\today} 
\begin{abstract}
We investigate the first-order differential calculus over extended metric-topological measure spaces.
The latter are quartets \(\mathbb X=(X,\tau,{\sf d},\mathfrak m)\), given by an extended metric space $(X,{\sf d})$
together with a weaker topology $\tau$ (satisfying suitable compatibility conditions) and a finite
Radon measure $\mathfrak m$ on $(X,\tau)$. The class of extended metric-topological measure spaces
encompasses all metric measure spaces and many infinite-dimensional metric-measure structures, such as
abstract Wiener spaces. In this framework, we study the following classes of objects:
\begin{itemize}
\item The Banach algebra ${\rm Lip}_b(X,\tau,{\sf d})$ of bounded $\tau$-continuous ${\sf d}$-Lipschitz functions on $X$.
\item Several notions of Lipschitz derivations on $X$, defined in duality with ${\rm Lip}_b(X,\tau,{\sf d})$.
\item The metric Sobolev space $W^{1,p}(\mathbb X)$, defined in duality with Lipschitz derivations on $X$.
\end{itemize}
Inter alia, we generalise both Weaver's and Di Marino's theories of Lipschitz derivations to the extended setting,
and we discuss their connections. We also introduce a Sobolev space $W^{1,p}(\mathbb X)$ via an integration-by-parts
formula, along the lines of Di Marino's notion of Sobolev space, and we prove its equivalence with
other approaches, studied in the extended setting by Ambrosio, Erbar and Savar\'{e}.
En route, we obtain some results of independent interest, among which are:
\begin{itemize}
\item A Lipschitz-constant-preserving extension result for $\tau$-continuous ${\sf d}$-Lipschitz functions.
\item A novel and rather robust strategy for proving the equivalence of Sobolev-type spaces defined via an
integration-by-parts formula and those obtained with a relaxation procedure.
\item A new description of an isometric predual of the metric Sobolev space $W^{1,p}(\mathbb X)$.
\end{itemize}
\end{abstract}
\maketitle
\tableofcontents
\section{Introduction}
\subsection{General overview}
In the last three decades, the analysis in nonsmooth spaces has undergone impressive developments.
After the first nonlocal notion of \emph{metric Sobolev space} over a metric measure space \((X,\sfd,\mm)\)
had been introduced by Haj\l asz in \cite{Haj:96}, several (essentially equivalent) local notions were studied
in the literature:
\begin{itemize}
\item[\(\rm A)\)] The space \(H^{1,p}(X)\) obtained by approximation, via a \emph{relaxation} procedure.
This approach was pioneered by Cheeger \cite{Cheeger00} and later revisited by Ambrosio, Gigli and Savar\'{e}
\cite{AmbrosioGigliSavare11,AmbrosioGigliSavare11-3}.
\item[\(\rm B)\)] The space \(W^{1,p}(X)\) proposed by Di Marino in \cite{DiMar:14,DiMarPhD}, based on an
integration-by-parts formula involving a suitable class of \emph{Lipschitz derivations} with divergence.
\item[\(\rm C)\)] The \emph{Newtonian space} \(N^{1,p}(X)\) introduced by Shanmugalingam \cite{Shanmugalingam00},
based on the concept of \emph{upper gradient} by Heinonen and Koskela \cite{Hei:Kos:98}, and on the metric
version of Fuglede's notion of \emph{\(p\)-modulus} \cite{Fug:57}.
\item[\(\rm D)\)] The `Beppo Levi space' \(B^{1,p}(X)\), where the exceptional curve families for the validity
of the upper gradient inequality are selected via \emph{test plans} of curves. The first definition of this type
is due to Ambrosio, Gigli and Savar\'{e} \cite{AmbrosioGigliSavare11,AmbrosioGigliSavare11-3}. The variant of
plan of curves we consider in this paper, involving the concept of \emph{barycenter}, was introduced by Ambrosio,
Di Marino and Savar\'{e} in \cite{Amb:Mar:Sav:15}. 
\end{itemize}
We point out that our choices of notation for the various metric Sobolev spaces may depart from the original ones,
but they are consistent with the presentation in \cite{AILP24}. Other definitions of metric Sobolev spaces were
introduced and studied in the literature, but we do not mention them here as they are not needed for the purposes
of this paper. Remarkably, all the above four theories -- the two `Eulerian approaches' A), B) and the two `Lagrangian
approaches' C), D) -- were proven to be fully equivalent on arbitrary \emph{complete} metric measure spaces
\cite{Cheeger00,Shanmugalingam00,AmbrosioGigliSavare11-3}. Other related equivalence results were then achieved
in \cite{Amb:Mar:Sav:15,EB:20,LP24,AILP24}.
\medskip

Nevertheless, there are many infinite-dimensional metric-measure structures of interest -- where a refined differential
calculus is available or feasible -- that are not covered by the theory of metric measure spaces. Due to this reason, Ambrosio,
Erbar and Savar\'{e} introduced in \cite{AmbrosioErbarSavare16} the language of \emph{extended metric-topological measure
spaces}, which we abbreviate to \emph{e.m.t.m.\ spaces}. The class of e.m.t.m.\ spaces includes, besides `standard'
metric measure spaces, \emph{abstract Wiener spaces} \cite{Bogachev15} and \emph{configuration spaces} \cite{AlbeverioKondratievRockner},
among others. The main goal of \cite{AmbrosioErbarSavare16} was to understand the connection between gradient contractivity,
transport distances and lower Ricci bounds, as well as the interplay between metric and differentiable structures,
in the setting of e.m.t.m.\ spaces. One of the numerous contributions of \cite{AmbrosioErbarSavare16} is the introduction
of the notion of Sobolev space \(H^{1,p}(X)\) on e.m.t.m.\ spaces, later investigated further by Savar\'{e} in the
lecture notes \cite{Savare22}. Therein, the e.m.t.m.\ versions of the Sobolev spaces \(N^{1,p}(X)\) and \(B^{1,p}(X)\)
were introduced and studied in detail, ultimately obtaining the identification \(H^{1,p}(X)=N^{1,p}(X)=B^{1,p}(X)\)
on all complete e.m.t.m.\ spaces. The duality properties of these metric Sobolev spaces were then investigated by Ambrosio
and Savar\'{e} in \cite{Amb:Sav:21}.
\medskip

The primary objectives of this paper are to introduce the Sobolev space \(W^{1,p}(X)\) via integration-by-parts for
e.m.t.m.\ spaces, to show its equivalence with the other approaches and to explore the benefits it brings to
the theory of metric Sobolev spaces. To achieve these goals, we first develop the machinery of Lipschitz derivations
for e.m.t.m.\ spaces, which in turn requires an in-depth understanding of the algebra of real-valued bounded
\(\tau\)-continuous \(\sfd\)-Lipschitz functions on \(X\).
\medskip

Before delving into a more detailed description of the contents of this paper, let us expound the advantages
of working in the extended setting. Besides its intrinsic interest, the study of e.m.t.m.\ spaces
has significant implications at the level of metric measure spaces. On e.m.t.m.\ spaces the roles of the topology
and of the distance are `decoupled', and it turned out that for this reason the category of e.m.t.m.\ spaces is
closed under several useful operations under which the category of metric measure spaces is not closed.
Key examples are the \emph{compactification} \cite[Section 2.1.7]{Savare22} and the passage to the
\emph{length-conformal distance} \cite[Section 2.3.2]{Savare22}. Therefore, once an effective calculus on e.m.t.m.\ spaces
is developed, it is possible to reduce some problems on metric measure spaces to problems on \(\tau\)-compact length e.m.t.m.\ spaces
(as done, for example, in \cite[Section 5.2]{Savare22}). We believe that the full potential of this technique
has not been fully explored yet. On the other hand, dealing with arbitrary e.m.t.m.\ spaces poses new challenges,
which require new ideas and solutions. In the remaining sections of the Introduction, we shall comment on some of them.
\subsection{The algebra of \texorpdfstring{\(\tau\)}{tau}-continuous \texorpdfstring{\(\sfd\)}{d}-Lipschitz functions}
Let \((X,\tau,\sfd)\) be an extended metric-topological space (see Definition \ref{def:emtm}). We consider the algebra
of bounded \(\tau\)-continuous \(\sfd\)-Lipschitz functions on \(X\), denoted by \(\Lip_b(X,\tau,\sfd)\). The latter
is a Banach algebra with respect to the norm
\[
\|f\|_{\Lip_b(X,\tau,\sfd)}\coloneqq\|f\|_{C_b(X,\tau)}+\Lip(f,\sfd).
\]
While the Banach algebra \(\Lip_b(Y,\sfd_Y)\) on a metric space \((Y,\sfd_Y)\) is (isometrically isomorphic to) the dual of a Banach
space, i.e.\ of the \emph{Arens--Eells space} \(\text{\AE}(Y)\) of \(Y\) \cite{Weaver2018}, the space \(\Lip_b(X,\tau,\sfd)\)
may not have a predual (as we show in Proposition \ref{prop:ex_no_predual}), thus it is not endowed with a weak\(^*\) topology.
This fact is relevant when discussing the continuity of derivations, see Section \ref{s:intro_der}.
\medskip

Another issue we need to address in the paper is whether it is possible to extend \(\tau\)-continuous \(\sfd\)-Lipschitz
functions preserving the Lipschitz constant. These kinds of extension results are very important e.g.\ in some localisation
arguments (such as in Proposition \ref{prop:glob_to_loc}). On metric spaces the McShane--Whitney extension theorem serves
the purpose, but on e.m.t.\ spaces the problem becomes much more delicate, because one has to preserve both
\(\tau\)-continuity and \(\sfd\)-Lipschitzianity when extending a function. In Section \ref{s:extLip} we deal with this
matter. Leveraging strong extension techniques by Matou\v{s}kov\'{a} \cite{Matouskova00}, we obtain
the sought-after Lipschitz-constant-preserving extension result for bounded \(\tau\)-continuous \(\sfd\)-Lipschitz functions
(Theorem \ref{thm:McShane_emts}), which is sharp (Remark \ref{rmk:ext_on_cpt_emt}).
\subsection{Metric derivations}\label{s:intro_der}
In Section \ref{s:Lipschitz_derivations}, we analyse various spaces of derivations on e.m.t.m.\ spaces.
In Definition \ref{def:der} we introduce a rather general (and purely algebraic) notion of derivation,
which comprises the different variants we will consider. By a \emph{Lipschitz derivation} on an e.m.t.m.\ space
\(\mathbb X=(X,\tau,\sfd,\mm)\) we mean a linear map \(b\colon\Lip_b(X,\tau,\sfd)\to L^0(\mm)\)
satisfying the \emph{Leibniz rule}:
\[
b(fg)=f\,b(g)+g\,b(f)\quad\text{ for every }f,g\in\Lip_b(X,\tau,\sfd).
\]
Here, \(L^0(\mm)\) denotes the algebra of all real-valued \(\tau\)-Borel functions on \(X\), up to \(\mm\)-a.e.\ equality.
Distinguished subclasses of derivations are those having \emph{divergence} (Definition \ref{def:div_Lip_der}),
that are \emph{local} (Definition \ref{def:loc_der}), or that satisfy \emph{`weak\(^*\)-type' (sequential) continuity}
properties (Definition \ref{def:w-type_cont}). In addition to these, we develop the basic theory of two crucial subfamilies
of Lipschitz derivations:
\begin{itemize}
\item \textsc{Weaver derivations.} In Definition \ref{def:Weaver_der} we propose a generalisation of Weaver's concept
of `bounded measurable vector field' \cite[Definition 10.30 a)]{Weaver2018} to the extended setting. Consistently e.g.\ with
\cite{Schioppa16}, we adopt the term \emph{Weaver derivation}. An important technical point here is that we ask
for the weak\(^*\)-type \emph{sequential} continuity, not for the weak\(^*\)-type continuity. The reason
is that weakly\(^*\)-type continuous derivations are trivial on the `purely non-\(\sfd\)-separable component'
\(X\setminus{\rm S}_{\mathbb X}\) of \(X\) (as in Lemma \ref{lem:d-sep_comp}), see
Proposition \ref{prop:w*_cont_vs_w*_seq_cont}.
\item \textsc{Di Marino derivations.} In Definition \ref{def:DiMar_der} we introduce the natural generalisation of Di Marino's
notion of derivation \cite{DiMar:14,DiMarPhD} to e.m.t.m.\ spaces. More specifically, we consider the space \(\Der^q(\mathbb X)\)
of \(q\)-integrable derivations, and its subspace \(\Der^q_q(\mathbb X)\) consisting of all those \(q\)-integrable derivations
having \(q\)-integrable divergence, for some given exponent \(q\in(1,\infty)\). This axiomatisation is tailored to the notion of metric Sobolev
space \(W^{1,p}(\mathbb X)\) (where \(p\in(1,\infty)\) is the conjugate exponent of \(q\)) that one aims at defining by means
of an integration-by-parts formula where \(\Der^q_q(\mathbb X)\) is used as the family of `test vector fields'.
\end{itemize}
Since in this paper we are primarily interested in the Sobolev calculus, we shall focus our attention mostly on Di Marino derivations.
Nevertheless, we set up also the basic theory of Weaver derivations and we debate their relation with the Di Marino ones
(see Proposition \ref{prop:glob_to_loc} or Theorem \ref{thm:W_vs_DM}, where we borrow some ideas from \cite{AILP24}).
We believe that Weaver derivations may find interesting applications even in the analysis on e.m.t.m.\ spaces,
for instance for studying suitable generalisations of metric currents or Alberti representations (cf.\ with \cite{Schioppa16,Schioppa16-2,Schioppa14}),
but addressing these kinds of issues is outside the scope of the present paper.
\subsection{Metric Sobolev spaces}
In Section \ref{s:Sob_space}, we introduce the metric Sobolev space \(W^{1,p}(\mathbb X)\), and we compare it
with \(H^{1,p}(\mathbb X)\), \(B^{1,p}(\mathbb X)\) and \(N^{1,p}(\mathbb X)\). Mimicking \cite[Definition 1.5]{DiMar:14},
we declare that some \(f\in L^p(\mm)\) belongs to \(W^{1,p}(\mathbb X)\) if there is a linear operator
\(L_f\colon\Der^q_q(\mathbb X)\to L^1(\mm)\) satisfying some algebraic and topological conditions, as well as the
following integration-by-parts formula:
\[
\int L_f(b)\,\d\mm=-\int f\,\div(b)\,\d\mm\quad\text{ for every }b\in\Der^q_q(\mathbb X);
\]
see Definition \ref{def:Sobolev_space_via_der}. Each \(f\in W^{1,p}(\mathbb X)\) is associated with a distinguished
function \(|Df|\in L^p(\mm)^+\), which has the role of the `modulus of the weak differential of \(f\)'.
\medskip

In Section \ref{s:H=W}, we show that on \emph{any} e.m.t.m.\ space it holds that
\[
H^{1,p}(\mathbb X)=W^{1,p}(\mathbb X),\;\text{ with }|Df|=|Df|_H\text{ for every }f\in W^{1,p}(\mathbb X);
\]
see Theorem \ref{thm:H=W}. The proof strategy for the inclusion \(H^{1,p}(\mathbb X)\subseteq W^{1,p}(\mathbb X)\)
is taken from \cite{DiMar:14} up to some technical discrepancies, whereas the verification of the converse inclusion
relies on a new argument, which was partially inspired by \cite{LP24}. In a nutshell, we first observe that \(H^{1,p}(\mathbb X)\)
induces a \emph{differential} \(\d\colon L^p(\mm)\to L^p(T^*\mathbb X)\), where \(L^p(T^*\mathbb X)\) is the e.m.t.m.\ version
of Gigli's notion of \emph{cotangent module} from \cite{Gigli14} (Theorem \ref{thm:cotg_mod}) and \(\d\) is an unbounded operator
with domain \(D(\d)=H^{1,p}(\mathbb X)\), then we prove that \(W^{1,p}(\mathbb X)\subseteq H^{1,p}(\mathbb X)\)
via a convex duality argument involving the adjoint \(\d^*\) of \(\d\). The latter proof strategy is rather robust and suitable
for being adapted to obtain analogous equivalence results for other functional spaces. We also point out that the identification
\(H^{1,p}(\mathbb X)=W^{1,p}(\mathbb X)\) for possibly non-complete spaces is new and interesting even in the particular case
where \((X,\sfd)\) is a metric space and \(\tau\) is the topology induced by \(\sfd\), and it covers e.g.\ those situations
in which \(X\) is an open domain in a larger (typically complete) ambient space.
\medskip

By combining Theorem \ref{thm:H=W} with \cite{Savare22}, we obtain that on \emph{complete} e.m.t.m.\ spaces it holds that
\[
W^{1,p}(\mathbb X)=B^{1,p}(\mathbb X),\;\text{ with }|Df|=|Df|_B\text{ for every }f\in W^{1,p}(\mathbb X);
\]
see Corollary \ref{cor:W=B}. If in addition \((X,\tau)\) is Souslin, then the space \(W^{1,p}(\mathbb X)\) can be identified
also with the Newtonian space \(N^{1,p}(\mathbb X)\); see Remark \ref{rmk:Newt}. On the other hand, these identities
are not always in force without the completeness assumption, cf.\ with the last paragraph of Section \ref{s:def_B}.
However, we show that -- on arbitrary e.m.t.m.\ spaces -- each \emph{\(\mathcal T_q\)-test plan} \(\ppi\)
(as in Definition \ref{def:test_plan}) induces a derivation \(b_\sppi\in\Der^q_q(\mathbb X)\)
(see Proposition \ref{prop:der_induced_by_tp}), and as a consequence we obtain that the inclusion
\(W^{1,p}(\mathbb X)\subseteq B^{1,p}(\mathbb X)\) holds and that \(|Df|_B\leq|Df|\) for all \(f\in W^{1,p}(\mathbb X)\)
(Theorem \ref{thm:W_in_B}).
\medskip

Finally, in Section \ref{s:predual_W1p} we present a quite elementary construction of some \emph{isometric predual}
of the metric Sobolev space \(W^{1,p}(\mathbb X)\), see Theorem \ref{thm:predual_W1p}. The formulation of the Sobolev
space in terms of derivations is particularly appropriate for this kind of construction. The existence of an isometric
predual of \(H^{1,p}(\mathbb X)\) was already known from \cite{Amb:Sav:21}.
\subsection*{Acknowledgements}
The first named author was supported by the Research Council of Finland grant 362898.
The second named author was supported by the Research Council of Finland grant 354241 and the Emil Aaltonen Foundation
through the research group ``Quasiworld network''. We thank Sylvester Eriksson-Bique and Timo Schultz for the several helpful discussions.
\section{Preliminaries}
Let us fix some general terminology and notation, which we will use throughout the whole paper. For any \(a,b\in\R\), we write \(a\vee b\coloneqq\max\{a,b\}\) and \(a\wedge b\coloneqq\min\{a,b\}\).
Given a set \(X\) and a function \(f\colon X\to\R\), we denote by \({\rm Osc}_S(f)\in[0,+\infty]\) the \textbf{oscillation} of \(f\) on a set \(S\subseteq X\), i.e.
\[
{\rm Osc}_S(f)\coloneqq\sup_S f-\inf_S f.
\]
For any Banach space \(\B\), we denote by \(\B'\) its dual Banach space. A map \(T\colon\B_1\to\B_2\) between two Banach spaces \(\B_1\) and \(\B_2\) is called an
\textbf{isomorphism} (resp.\ an \textbf{isometric isomorphism}) provided it is a linear homeomorphism (resp.\ a norm-preserving linear homeomorphism). Accordingly, we say that
\(\B_1\) and \(\B_2\) are \textbf{isomorphic} (resp.\ \textbf{isometrically isomorphic}) provided there exists an isomorphism (resp.\ an isometric isomorphism) \(T\colon\B_1\to\B_2\).
Finally, we say that \(\mathbb B_1\) \textbf{embeds} (resp.\ \textbf{isometrically embeds}) into \(\mathbb B_2\) provided \(\mathbb B_1\) is isomorphic (resp.\ isometrically isomorphic)
to some subspace of \(\B_2\).
\subsection{Topological and metric notions}
Let us recall some notions in topology, referring e.g.\ to the book \cite{Kelley75} for a detailed discussion on the topic.
Let \((X,\tau)\) be a topological space. Then:
\begin{itemize}
\item \((X,\tau)\) is said to be \textbf{completely regular} if for any \(x\in X\) and any neighbourhood
\(U\in\tau\) of \(x\) there exists a continuous function \(f\colon X\to[0,1]\) such that \(f(x)=1\)
and \(f|_{X\setminus U}=0\). Equivalently, \((X,\tau)\) is completely regular if \(\tau\) is induced
by a family of semidistances.
\item \((X,\tau)\) is said to be \textbf{normal} if for any pair of disjoint closed sets \(A,B\subseteq X\)
there exist disjoint open sets \(U_A,U_B\in\tau\) such that \(A\subseteq U_A\) and \(B\subseteq U_B\).
\item \((X,\tau)\) is said to be a \textbf{Tychonoff space} if it is completely regular and Hausdorff.
Every locally compact Hausdorff topological space is a Tychonoff space.
\end{itemize}
Given two topological spaces \((X,\tau_X)\) and \((Y,\tau_Y)\), we denote by \(C((X,\tau_X);(Y,\tau_Y))\)
the space of continuous maps from \((X,\tau_X)\) to \((Y,\tau_Y)\); we drop \(\tau_X\) or \(\tau_Y\)
from our notation when the chosen topologies are clear from the context. We use the shorthand notation
\(C(X,\tau)\coloneqq C((X,\tau);\R)\) for any topological space \((X,\tau)\), where the target \(\R\)
is equipped with the Euclidean topology. Then
\[
C_b(X,\tau)\coloneqq\big\{f\in C(X,\tau)\;\big|\;f\text{ is bounded}\big\}
\]
is a Banach space if endowed with the supremum norm \(\|f\|_{C_b(X,\tau)}\coloneqq\sup_{x\in X}|f(x)|\).
\medskip

Next, let us recall some metric concepts. By an \textbf{extended distance} on a set \(X\) we mean a symmetric function
\(\sfd\colon X\times X\to[0,+\infty]\) that satisfies the triangle inequality and vanishes exactly on the diagonal
\(\{(x,x):x\in X\}\). The pair \((X,\sfd)\) is called an \textbf{extended metric space}.
As usual, if \(\sfd(x,y)<+\infty\) for every \(x,y\in X\), then \(\sfd\) is called a distance and \((X,\sfd)\) is called a metric space.
Given an extended metric space \((X,\sfd)\), a center \(x\in X\) and a radius \(r\in(0,+\infty)\), we denote
\[
B_r^\sfd(x)\coloneqq\big\{y\in X\;\big|\;\sfd(x,y)<r\big\},\qquad\bar B_r^\sfd(x)\coloneqq\big\{y\in X\;\big|\;\sfd(x,y)\leq r\big\}.
\]
A map \(\varphi\colon X\to Y\) between two extended metric spaces \((X,\sfd_X)\) and \((Y,\sfd_Y)\)
is said to be Lipschitz (or \(L\)-Lipschitz) if for some constant \(L\geq 0\) we have that
\(\sfd_Y(\varphi(x),\varphi(y))\leq L\,\sfd_X(x,y)\) holds for every \(x,y\in X\). We denote by
\(\Lip_b(X,\sfd)\) the space of all bounded Lipschitz functions from an extended metric space \((X,\sfd)\)
to the real line \(\R\) (equipped with the Euclidean distance). Denote
\[
\Lip(f,A,\sfd)\coloneqq\sup\bigg\{\frac{|f(x)-f(y)|}{\sfd(x,y)}\;\bigg|\;x,y\in A,\,x\neq y\bigg\}
\quad\text{ for every }f\in\Lip_b(X,\sfd)\text{ and }A\subseteq X.
\]
For brevity, we write \(\Lip(f,\sfd)\coloneqq\Lip(f,X,\sfd)\). It is well known that \(\Lip_b(X,\sfd)\) is
a Banach space with respect to the norm \(\|f\|_{\Lip_b(X,\sfd)}\coloneqq\Lip(f,\sfd)+\sup_{x\in X}|f(x)|\).
\medskip

Now, consider an extended metric space \((X,\sfd)\) together with a topology \(\tau\) on \(X\). We define
\[
\Lip_b(X,\tau,\sfd)\coloneqq\Lip_b( X,\sfd)\cap C(X,\tau).
\]
We endow the vector space \(\Lip_b(X,\tau,\sfd)\) with the norm
\[
\|f\|_{\Lip_b(X,\tau,\sfd)}\coloneqq\Lip(f,\sfd)+\|f\|_{C_b(X,\tau)}\quad\text{ for every }f\in\Lip_b(X,\tau,\sfd).
\]
\begin{remark}{\rm
We claim that
\[
\big(\Lip_b(X,\tau,\sfd),\|\cdot\|_{\Lip_b(X,\tau,\sfd)}\big)\;\text{ is a Banach algebra.}
\]
Indeed, \(\|f\|_{\Lip_b(X,\tau,\sfd)}=\|f\|_{\Lip_b(X,\sfd)}\) holds for every \(f\in\Lip_b(X,\tau,\sfd)\),
and every uniform limit of \(\tau\)-continuous functions is \(\tau\)-continuous, thus \(\Lip_b(X,\tau,\sfd)\)
is a closed vector subspace of \(\Lip_b(X,\sfd)\). In particular, \(\Lip_b(X,\tau,\sfd)\) is a Banach space.
Moreover, it can be readily checked that for any given \(f,g\in\Lip_b(X,\tau,\sfd)\) we have that \(fg\in\Lip_b(X,\tau,\sfd)\),
\(\|fg\|_{C_b(X,\tau)}\leq\|f\|_{C_b(X,\tau)}\|g\|_{C_b(X,\tau)}\) and \(\Lip(fg,\sfd)\leq\|f\|_{C_b(X,\tau)}\Lip(g,\sfd)+\|g\|_{C_b(X,\tau)}\Lip(f,\sfd)\),
whence it follows that
\[\begin{split}
\|fg\|_{\Lip_b(X,\tau,\sfd)}&=\Lip(fg,\sfd)+\|fg\|_{C_b(X,\tau)}\\
&\leq\|f\|_{C_b(X,\tau)}\Lip(g,\sfd)+\|g\|_{C_b(X,\tau)}\Lip(f,\sfd)+\|f\|_{C_b(X,\tau)}\|g\|_{C_b(X,\tau)}\\
&\leq\big(\Lip(f,\sfd)+\|f\|_{C_b(X,\tau)}\big)\big(\Lip(g,\sfd)+\|g\|_{C_b(X,\tau)}\big)=\|f\|_{\Lip_b(X,\tau,\sfd)}\|g\|_{\Lip_b(X,\tau,\sfd)}.
\end{split}\]
All in all, we have shown that \(\Lip_b(X,\tau,\sfd)\) is a Banach algebra, as we claimed.
\fr}\end{remark}

At times, it is convenient to use the following shorthand notation:
\begin{equation}\label{eq:def_Lip_b_1}
\Lip_{b,1}(X,\tau,\sfd)\coloneqq\big\{f\in\Lip_b(X,\tau,\sfd)\;\big|\;\Lip(f,\sfd)\leq 1\big\}.
\end{equation}
Any given \(f\in\Lip_b(X,\tau,\sfd)\) is associated with a function \(\lip_\sfd(f)\) that accounts for the `infinitesimal Lipschitz constants'
of \(f\) at the different points of \(X\):
\begin{definition}[Asymptotic slope]
Let \((X,\sfd)\) be an extended metric space and let \(\tau\) be a topology on \(X\). Let \(f\in\Lip_b(X,\tau,\sfd)\) be given. Then we define
the function \(\lip_\sfd(f)\colon X\to[0,\Lip(f,\sfd)]\) as
\[
\lip_\sfd(f)(x)\coloneqq\inf\big\{\Lip(f,U,\sfd)\;\big|\;x\in U\in\tau\big\}\quad\text{ for every }x\in X.
\]
We say that \(\lip_\sfd(f)\) is the \textbf{asymptotic slope} of \(f\).
\end{definition}

The function \(\lip_\sfd(f)\) is \(\tau\)-upper semicontinuous, as it follows from the ensuing remark:
\begin{remark}\label{rmk:suff_cond_sc}{\rm
Let \((X,\tau)\) be a topological space and \(S\neq\varnothing\) a subset of \(\tau\). Let \(\mathcal F\colon S\to[0,+\infty]\) be any given functional. Define
\[
F(x)\coloneqq\inf\big\{\mathcal F(U)\;\big|\,x\in U\in S\big\}\quad\text{ for every }x\in X.
\]
Then \(F\colon X\to[0,+\infty]\) is a \(\tau\)-upper semicontinuous function. Indeed, for any \(U\in S\) we have that
\[
F_U(x)\coloneqq\left\{\begin{array}{ll}
\mathcal F(U)\\
+\infty
\end{array}\quad\begin{array}{ll}
\text{ for every }x\in U,\\
\text{ for every }x\in X\setminus U
\end{array}\right.
\]
defines a \(\tau\)-upper semicontinuous function \(F_U\colon X\to[0,+\infty]\), thus \(F=\inf_{U\in S}F_U\) is \(\tau\)-upper semicontinuous as well.
Similarly, we have that \(G(x)\coloneqq\sup\{\mathcal F(U):x\in U\in S\}\) (with the convention that \(\sup(\varnothing)=0\)) defines
a \(\tau\)-lower semicontinuous function \(G\colon X\to[0,+\infty]\).
\fr}\end{remark}
\subsection{Measure theory}
Let \((X,\Sigma,\mm)\) be a measure space. We denote by \(L^0(\mm)\) the algebra of all equivalence classes (up to \(\mm\)-a.e.\ equality)
of measurable functions \(f\colon X\to\R\). For any \(p\in[1,\infty]\), we denote by \((L^p(\mm),\|\cdot\|_{L^p(\mm)})\)
the \textbf{Lebesgue space} of exponent \(p\) on \((X,\Sigma,\mm)\). Then \(L^p(\mm)\) is a Banach space (and \(L^\infty(\mm)\) is also a Banach algebra).
Moreover, \(L^p(\mm)\) is a Riesz space with respect to the partial order given by the \(\mm\)-a.e.\ inequality: given any \(f,g\in L^p(\mm)\),
we declare that \(f\leq g\) if and only if \(f(x)\leq g(x)\) holds for \(\mm\)-a.e.\ \(x\in X\). Assuming that the measure \(\mm\) is \(\sigma\)-finite,
we also have that \(L^p(\mm)\) is \textbf{Dedekind complete}, which means that any family of functions \(\{f_i\}_{i\in I}\subseteq L^p(\mm)\) with an upper
bound (i.e.\ there exists \(g\in L^p(\mm)\) such that \(f_i\leq g\) for all \(i\in I\)) has a supremum \(\bigvee_{i\in I}f_i\in L^p(\mm)\).
The latter is the unique element of \(L^p(\mm)\) such that
\begin{itemize}
\item \(f_j\leq\bigvee_{i\in I}f_i\) for every \(j\in I\),
\item if \(\tilde f\in L^p(\mm)\) satisfies \(f_j\leq\tilde f\) for every \(j\in I\), then \(\bigvee_{i\in I}f_i\leq\tilde f\).
\end{itemize}
In addition, one can find an at most countable subset \(C\subseteq I\) such that \(\bigvee_{i\in I}f_i=\bigvee_{i\in C}f_i\) (i.e.\ \(L^p(\mm)\) has the so-called
\textbf{countable sup property}). Similarly, every set \(\{f_i\}_{i\in I}\subseteq L^p(\mm)\) with a lower bound has an infimum \(\bigwedge_{i\in I}f_i\in L^p(\mm)\)
and there exists \(\tilde C\subseteq I\) at most countable such that \(\bigwedge_{i\in I}f_i=\bigwedge_{i\in\tilde C}f_i\) (i.e.\ the countable inf property holds).
In particular, essential unions (and essential intersections) exist: given any family \(\{E_i\}_{i\in I}\subseteq\Sigma\), we can find a set \(E\in\Sigma\) such that
\begin{itemize}
\item \(\mm(E_i\setminus E)=0\) for every \(i\in I\),
\item if \(F\in\Sigma\) satisfies \(\mm(E_i\setminus F)=0\) for every \(i\in I\), then \(\mm(E\setminus F)=0\).
\end{itemize}
The set \(E\) is \(\mm\)-a.e.\ unique, in the sense that \(\mm(E\Delta\tilde E)=0\) for any other set \(\tilde E\in\Sigma\) having the same properties.
We say that \(E\) is the \textbf{\(\mm\)-essential union} of \(\{E_i\}_{i\in I}\). It also holds that \(E\) can be chosen of the form \(\bigcup_{i\in C}E_i\),
for some at most countable subset \(C\subseteq I\).
\medskip

Let \((X,\Sigma,\mm)\) be a finite measure space. Following \cite[\S 1.12(iii)]{Bogachev07}, we say that \(\mm\) is a \textbf{separable measure} if there exists a countable
family \(\mathcal C\subseteq\Sigma\) such that for every \(E\in\Sigma\) and \(\varepsilon>0\) we can find \(F\in\mathcal C\) such that \(\mm(E\Delta F)<\varepsilon\).
The following conditions are equivalent:
\begin{itemize}
\item \(\mm\) is a separable measure,
\item \(L^p(\mm)\) is separable for some \(p\in[1,\infty)\),
\item \(L^p(\mm)\) is separable for every \(p\in[1,\infty)\).
\end{itemize}
See for instance \cite[\S 7.14(iv) and Exercise 4.7.63]{Bogachev07}. In the class of spaces of our interest in this paper,
we can encounter examples of spaces whose reference measure is non-separable (cf.\ with Example \ref{ex:non-sep_m}).
An advantage of \(\mm\) being separable is that it is equivalent to the fact that the weak\(^*\) topology of
\(L^\infty(\mm)\) restricted to its closed unit ball is metrisable (see e.g.\ Lemma \ref{lem:suff_cond_wstar_seq_cont}).
\medskip

Let \((X,\tau)\) be a Hausdorff topological space. We denote by \(\mathscr B(X,\tau)\) its Borel \(\sigma\)-algebra. A finite Borel measure \(\mu\colon\mathscr B(X,\tau)\to[0,+\infty)\)
is called a \textbf{Radon measure} if it is \textbf{inner regular}, i.e.
\[
\mu(B)=\sup\big\{\mu(K)\;\big|\;K\subseteq B,\,K\text{ is \(\tau\)-compact}\big\}\quad\text{ for every }B\in\mathscr B(X,\tau).
\]
It follows that \(\mu\) is also \textbf{outer regular}, which means that
\[
\mu(B)=\inf\big\{\mu(U)\;\big|\;U\in\tau,\,B\subseteq U\big\}\quad\text{ for every }B\in\mathscr B(X,\tau).
\]
We denote by \(\mathcal M_+(X)\) or \(\mathcal M_+(X,\tau)\) the collection of all finite Radon measures on \((X,\tau)\). We refer to the monograph \cite{Schwartz73}
for a thorough account of the theory of Radon measures. Below we collect some more definitions and results that we shall need later in the paper.
\begin{remark}\label{rmk:monot_conv_for_nets}{\rm
Radon measures verify the following version of the monotone convergence theorem: if \(\mu\) is a finite Radon measure on a Hausdorff topological space \((X,\tau)\) and
\((f_i)_{i\in I}\) is a non-decreasing net of \(\tau\)-lower semicontinuous functions \(f_i\colon X\to[0,+\infty)\) satisfying \(\sup_{i\in I,x\in X}f_i(x)<+\infty\), then
\[
\lim_{i\in I}\int f_i\,\d\mu=\int\lim_{i\in I}f_i\,\d\mu.
\]
Note that \(\lim_{i\in I}f_i=\sup_{i\in I}f_i\) is \(\tau\)-lower semicontinuous, in particular it is Borel measurable and thus the right-hand
side of the identity above is meaningful. See e.g.\ \cite[Lemma 7.2.6]{Bogachev07}.
\fr}\end{remark}
Let \((X,\tau_X)\) and \((Y,\tau_Y)\) be Tychonoff spaces. Given a finite Radon measure \(\mu\) on \(X\), a map \(\varphi\colon X\to Y\) is said
to be \textbf{Lusin \(\mu\)-measurable} if for any \(\varepsilon>0\) there exists a compact set \(K_\varepsilon\subseteq X\) such that
\(\mu(X\setminus K_\varepsilon)\leq\varepsilon\) and \(\varphi|_{K_\varepsilon}\) is continuous. Each Lusin \(\mu\)-measurable map is in particular Borel \(\mu\)-measurable
(i.e.\ \(\varphi^{-1}(B)\) is a \(\mu\)-measurable subset of \(X\) for every Borel set \(B\subseteq Y\)). Moreover, if \(\mu\in\mathcal M_+(X)\) is given
and \(\varphi\colon X\to Y\) is Lusin \(\mu\)-measurable, then we have that
\[
(\varphi_\#\mu)(B)\coloneqq\mu(\varphi^{-1}(B))\quad\text{ for every Borel set }B\subseteq Y
\]
defines a Radon measure \(\varphi_\#\mu\in\mathcal M_+(Y)\), called the \textbf{pushforward} of \(\mu\) under \(\varphi\).
A map \(\varphi\colon X\to Y\) is said to be \textbf{universally Lusin measurable} if it is Lusin \(\mu\)-measurable for every \(\mu\in\mathcal M_+(X)\).
\begin{remark}\label{rmk:about_univ_Lusin_meas}{\rm
We point out that the \(\mu\)-a.e.\ pointwise limit of a sequence of Lusin \(\mu\)-measurable functions is Lusin \(\mu\)-measurable, thus in particular
the pointwise limit of a sequence of universally Lusin measurable functions is universally
Lusin measurable. Indeed, fix a Tychonoff space \((X,\tau)\) and a Radon measure \(\mu\in\mathcal M_+(X)\).
Assume that a sequence \((f_n)_n\) of Lusin \(\mu\)-measurable functions \(f_n\colon X\to\R\) and a limit function
\(f\colon X\to\R\) satisfy \(f(x)=\lim_n f_n(x)\in\R\) for \(\mu\)-a.e.\ \(x\in X\). Given any \(\varepsilon>0\) and \(n\in\N\),
we can find a compact set \(K_\varepsilon^n\subseteq X\) such that \(\mu(X\setminus K_\varepsilon^n)\leq\varepsilon/2^n\)
and \(f_n|_{K_\varepsilon^n}\) is continuous. Then \(K_\varepsilon\coloneqq\bigcap_{n\in\N}K_\varepsilon^n\) is a compact
set with \(\mu(X\setminus K_\varepsilon)\leq\varepsilon\) such that each \(f_n|_{K_\varepsilon}\) is continuous. Thanks to
Egorov's theorem, we can find a compact set \(\tilde K_\varepsilon\subseteq K_\varepsilon\) with
\(\mu(X\setminus\tilde K_\varepsilon)\leq 2\varepsilon\) such that \(f_n|_{\tilde K_\varepsilon}\to f|_{\tilde K_\varepsilon}\)
uniformly, so that \(f|_{\tilde K_\varepsilon}\) is continuous. Hence, \(f\) is Lusin \(\mu\)-measurable.

Furthermore, we point out that any bounded \(\tau\)-lower semicontinuous function \(f\colon X\to[0,+\infty)\) defined on a Tychonoff
space \((X,\tau)\) is universally Lusin measurable. To prove it, fix any Radon measure \(\mu\in\mathcal M_+(X)\). It follows
e.g.\ from \cite[Lemma 7.2.6]{Bogachev07} that
\[
\int f\,\d\mu=\sup\bigg\{\int g\,\d\mu\;\bigg|\;g\in C(X,\tau),\,g\leq f\bigg\}.
\]
Hence, we can find a non-decreasing sequence of functions \((g_n)_n\subseteq C(X,\tau)\) such that \(g_n\leq f\) for every
\(n\in\N\) and \(\lim_n\int g_n\,\d\mu=\int f\,\d\mu\). By applying the monotone convergence theorem, we deduce that
\(\lim_n g_n(x)=f(x)\) for \(\mu\)-a.e.\ \(x\in X\). Since each continuous function is clearly Lusin \(\mu\)-measurable,
by the first claim of this remark we conclude that \(f\) is Lusin \(\mu\)-measurable.
\fr}\end{remark}
\subsubsection{\(L^p(\mm)\)-Banach \(L^\infty(\mm)\)-modules}
In this section, we recall some key concepts in the theory of \(L^p\)-Banach \(L^\infty\)-modules, which are Banach spaces equipped
with additional structures (roughly speaking, with a `pointwise norm' and a multiplication by \(L^\infty\)-functions). This language
has been developed by Gigli in \cite{Gigli14}, with the aim of providing a functional-analytic framework for a vector calculus in
metric measure spaces. Strictly related notions were previously studied in the literature for different purposes, see e.g.\ the notion
of \emph{random normed module} introduced by Guo \cite{Guo-1989,Guo-1992} and investigated in a long series of works (see \cite{Guo-2011,Guo-2024}
and the references therein), or the notion of \emph{random Banach space} introduced by Haydon, Levy and Raynaud \cite{HLR91}.
The definitions and results presented below are taken from \cite{Gigli14,Gigli17}.
\medskip

For any measure space \((X,\Sigma,\mm)\), the space \(L^\infty(\mm)\) is a commutative ring (with unity)
with respect to the usual pointwise operations. Since the field of real numbers \(\R\) can be identified with a subring of
\(L^\infty(\mm)\) (via the map sending \(\lambda\in\R\) to the function that is \(\mm\)-a.e.\ equal to \(\lambda\)),
every module over \(L^\infty(\mm)\) is in particular a vector space. Recall also that a homomorphism
\(T\colon M\to N\) of \(L^\infty(\mm)\)-modules is an \(L^\infty(\mm)\)-linear operator, i.e.\ a map satisfying
\[
T(f\cdot v+g\cdot w)=f\cdot T(v)+g\cdot T(w)\quad\text{ for every }f,g\in L^\infty(\mm)\text{ and }v,w\in M.
\]
In particular, each homomorphism of \(L^\infty(\mm)\)-modules is a homomorphism of vector spaces, i.e.\ a linear operator.
Observe that \(L^p(\mm)\) is an \(L^\infty(\mm)\)-module for every \(p\in[1,\infty]\).
\begin{definition}[\(L^p(\mm)\)-Banach \(L^\infty(\mm)\)-module]
Let \((X,\Sigma,\mm)\) be a \(\sigma\)-finite measure space and let \(p\in(1,\infty)\). Then a module \(\mathscr M\)
over \(L^\infty(\mm)\) is said to be an \textbf{\(L^p(\mm)\)-Banach \(L^\infty(\mm)\)-module} if it is endowed with a
functional \(|\cdot|\colon\mathscr M\to L^p(\mm)^+\), called a \textbf{pointwise norm} on \(\mathscr M\), such that:
\begin{itemize}
\item[\(\rm i)\)] For any \(v\in\mathscr M\), it holds that \(|v|=0\) if and only if \(v=0\).
\item[\(\rm ii)\)] \(|v+w|\leq|v|+|w|\) for every \(v,w\in\mathscr M\).
\item[\(\rm iii)\)] \(|f\cdot v|=|f||v|\) for every \(f\in L^\infty(\mm)\) and \(v\in\mathscr M\).
\item[\(\rm iv)\)] The norm \(\|v\|_{\mathscr M}\coloneqq\||v|\|_{L^p(\mm)}\) on \(\mathscr M\) is complete.
\end{itemize}
\end{definition}

Every \(L^p(\mm)\)-Banach \(L^\infty(\mm)\)-module is in particular a Banach space. A map \(\Phi\colon\mathscr M\to\mathscr N\) between
\(L^p(\mm)\)-Banach \(L^\infty(\mm)\)-modules \(\mathscr M\), \(\mathscr N\) is said to be an \textbf{isomorphism of \(L^p(\mm)\)-Banach \(L^\infty(\mm)\)-modules}
if it is an isomorphism of \(L^\infty(\mm)\)-modules satisfying \(|\Phi(v)|=|v|\) for all \(v\in\mathscr M\).
\begin{definition}[Dual of an \(L^p(\mm)\)-Banach \(L^\infty(\mm)\)-module]
Let \((X,\Sigma,\mm)\) be a \(\sigma\)-finite measure space. Let \(p,q\in(1,\infty)\) be conjugate exponents
and let \(\mathscr M\) be an \(L^p(\mm)\)-Banach \(L^\infty(\mm)\)-module. Then we define \(\mathscr M^*\)
as the set of all those homomorphisms \(\omega\colon\mathscr M\to L^1(\mm)\) of \(L^\infty(\mm)\)-modules
for which there exists a function \(g\in L^q(\mm)^+\) such that
\begin{equation}\label{eq:dual_mod}
|\omega(v)|\leq g|v|\quad\text{ for every }v\in\mathscr M.
\end{equation}
The space \(\mathscr M^*\) is called the \textbf{continuous module dual} of \(\mathscr M\).
\end{definition}

The space \(\mathscr M^*\) is a module over \(L^\infty(\mm)\) if endowed with the following pointwise operations:
\[\begin{split}
(\omega+\eta)(v)&\coloneqq\omega(v)+\eta(v)\quad\text{ for every }\omega,\eta\in\mathscr M^*\text{ and }v\in\mathscr M,\\
(f\cdot\omega)(v)&\coloneqq f\,\omega(v)\quad\text{ for every }f\in L^\infty(\mm),\,\omega\in\mathscr M^*\text{ and }v\in\mathscr M.
\end{split}\]
Moreover, to any element \(\omega\in\mathscr M^*\) we associate the function \(|\omega|\in L^q(\mm)^+\), which we define as
\[
|\omega|\coloneqq\bigvee\big\{\omega(v)\;\big|\;v\in\mathscr M,\,|v|\leq 1\big\}=\bigwedge\big\{g\in L^q(\mm)^+\;\big|\;g\text{ satisfies \eqref{eq:dual_mod}}\big\}.
\]
It holds that \((\mathscr M^*,|\cdot|)\) is an \(L^q(\mm)\)-Banach \(L^\infty(\mm)\)-module.
\medskip

The continuous module dual \(\mathscr M^*\) of \(\mathscr M\) is in particular a Banach space, which can be identified with the dual Banach space
\(\mathscr M'\) through the operator \(\textsc{Int}_{\mathscr M}\colon\mathscr M^*\to\mathscr M'\), defined as
\begin{equation}\label{eq:def_Int}
\textsc{Int}_{\mathscr M}(\omega)(v)\coloneqq\int\omega(v)\,\d\mm\quad\text{ for every }\omega\in\mathscr M^*\text{ and }v\in\mathscr M.
\end{equation}
Indeed, the map \(\textsc{Int}_{\mathscr M}\) is an isometric isomorphism of Banach spaces (see \cite[Proposition 1.2.13]{Gigli14}).
\subsection{Extended metric-topological measure spaces}
In this section, we discuss the notion of \emph{extended metric-topological (measure) space} that was introduced by Ambrosio, Erbar and Savar\'{e}
in \cite[Definitions 4.1 and 4.7]{AmbrosioErbarSavare16} (see also \cite[Definition 2.1.3]{Savare22}).
\begin{definition}[Extended metric-topological measure space]\label{def:emtm}
Let \((X,\sfd)\) be an extended metric space and let \(\tau\) be a Hausdorff topology on \(X\). Then we say that \((X,\tau,\sfd)\) is an \textbf{extended metric-topological space}
(or an \textbf{e.m.t.\ space} for short) if the following conditions hold:
\begin{itemize}
\item[\({\rm i)}\)] The topology \(\tau\) coincides with the initial topology of \(\Lip_b(X,\tau,\sfd)\).
\item[\({\rm ii)}\)] The extended distance \(\sfd\) can be recovered through the formula
\begin{equation}\label{dist:recovery}
\sfd(x,y)=\sup\big\{|f(x)-f(y)|\;\big|\,f\in\Lip_{b,1}(X,\tau,\sfd)\big\}\quad\text{ for every }x,y\in X,
\end{equation}
where \(\Lip_{b,1}(X,\tau,\sfd)\) is defined as in \eqref{eq:def_Lip_b_1}.
\end{itemize}
When \((X,\tau,\sfd)\) is equipped with a finite Radon measure \(\mm\in\mathcal M_+(X)\), we say that \(\mathbb X\coloneqq(X,\tau,\sfd,\mm)\)
is an \textbf{extended metric-topological measure space} (or an \textbf{e.m.t.m.\ space} for short).
\end{definition}

In particular, if \((X,\tau,\sfd)\) is an e.m.t.\ space, then \((X,\tau)\) is a Tychonoff space.
Given an e.m.t.m.\ space \(\mathbb X=(X,\tau,\sfd,\mm)\), we know from \cite[Lemma 2.1.27]{Savare22} that
\begin{equation}\label{eq:Lip_dense_Lp}
\Lip_b(X,\tau,\sfd)\text{ is strongly dense in }L^p(\mm),\text{ for every }p\in[1,\infty).
\end{equation}
Moreover, given any set \(E\in\mathscr B(X,\tau)\), it can be readily checked that
\begin{equation}\label{eq:restr_emtm}
\mathbb X\llcorner E\coloneqq(E,\tau_E,\sfd_E,\mm\llcorner E)
\end{equation}
is an e.m.t.m.\ space, where \(\tau_E\) is the subspace topology on \(E\) induced by \(\tau\),
while \(\sfd_E\coloneqq\sfd|_{E\times E}\) and \(\mm\llcorner E\) denotes the Radon measure on \(E\) that is
obtained from \(\mm\) by restriction.
\medskip

Let us now prove some technical results, which will be needed later. First, we show that each e.m.t.m.\ space can be decomposed
(in an \(\mm\)-a.e.\ unique manner) into a \(\sfd\)-separable component and a `purely non-\(\sfd\)-separable' one:
\begin{lemma}[Maximal \(\sfd\)-separable component \({\rm S}_{\mathbb X}\)]\label{lem:d-sep_comp}
Let \(\mathbb X=(X,\tau,\sfd,\mm)\) be a given e.m.t.m.\ space. Then there exists a \(\sfd\)-separable set
\({\rm S}_{\mathbb X}\in\mathscr B(X,\tau)\) such that \(\mm(E)=0\) holds for any \(\sfd\)-separable set \(E\in\mathscr B(X,\tau)\)
satisfying \(E\subseteq X\setminus{\rm S}_{\mathbb X}\). Moreover, the set \({\rm S}_{\mathbb X}\) is unique in the \(\mm\)-a.e.\ sense,
meaning that \(\mm({\rm S}_{\mathbb X}\Delta\tilde{\rm S}_{\mathbb X})=0\) for any other set \(\tilde{\rm S}_{\mathbb X}\in\mathscr B(X,\tau)\)
having the same properties as \({\rm S}_{\mathbb X}\).
\end{lemma}
\begin{proof}
Fix any \(\mm\)-a.e.\ representative \({\rm S}_{\mathbb X}\in\mathscr B(X,\tau)\) of the \(\mm\)-essential union of
the family of sets
\[
\big\{S\in\mathscr B(X,\tau)\;\big|\;S\text{ is }\sfd\text{-separable and }\mm(S)>0\big\}.
\]
Recall that \({\rm S}_{\mathbb X}\) can be chosen to be of the form \(\bigcup_{n\in\N}S_n\), for some
sequence \((S_n)_n\subseteq\mathscr B(X,\tau)\) such that \(S_n\) is \(\sfd\)-separable and \(\mm(S_n)>0\)
for every \(n\in\N\). In particular, the set \({\rm S}_{\mathbb X}\) is \(\sfd\)-separable. If
\(E\subseteq X\setminus{\rm S}_{\mathbb X}\) is \(\tau\)-Borel and \(\sfd\)-separable, then
\(\mm(E)=0\) thanks to the definition of \(\mm\)-essential union. Finally, if \(\tilde{\rm S}_{\mathbb X}\)
is another set having the same properties as \({\rm S}_{\mathbb X}\), then the inclusion
\(\tilde{\rm S}_{\mathbb X}\setminus{\rm S}_{\mathbb X}\subseteq X\setminus{\rm S}_{\mathbb X}\)
(resp.\ \({\rm S}_{\mathbb X}\setminus\tilde{\rm S}_{\mathbb X}\subseteq X\setminus\tilde{\rm S}_{\mathbb X}\))
implies that \(\mm(\tilde{\rm S}_{\mathbb X}\setminus{\rm S}_{\mathbb X})=0\)
(resp.\ \(\mm({\rm S}_{\mathbb X}\setminus\tilde{\rm S}_{\mathbb X})=0\)), thus
\(\mm({\rm S}_{\mathbb X}\Delta\tilde{\rm S}_{\mathbb X})=0\).
\end{proof}

Next, we give sufficient conditions for the separability of the measure \(\mm\) of an e.m.t.m.\ space.
The proof of the ensuing result is rather standard, but we provide it for the reader's convenience.
\begin{lemma}\label{lem:separability_Lp}
Let \(\mathbb X=(X,\tau,\sfd,\mm)\) be an e.m.t.m.\ space. Assume either that \(\tau\) is metrisable on every \(\tau\)-compact
set or that \(\mm(X\setminus{\rm S}_{\mathbb X})=0\). Then it holds that the measure \(\mm\) is separable.
\end{lemma}
\begin{proof}
Let us distinguish the two cases. First, assume that \(\tau\) is metrisable on \(\tau\)-compact sets.
Take an increasing sequence \((K_n)_{n\in\N}\) of \(\tau\)-compact subsets of \(X\) with \(\mm\big(X\setminus\bigcup_n K_n\big)=0\).
For any \(n\in\N\), fix a distance \(\sfd_n\) on \(K_n\) metrising \(\tau\), and a \(\sfd_n\)-dense sequence \((x^n_j)_{j\in\N}\) in \(K_n\). Define
\[
\mathcal C\coloneqq\bigcup_{n\in\N}\bigg\{\bigcup_{j\in F}\bar B_{q_j}^{\sfd_n}(x^n_j)\;\bigg|\;F\subseteq\N\text{ finite},\,(q_j)_{j\in F}\subseteq\mathbb Q\cap(0,+\infty)\bigg\}.
\]
Note that \(\mathcal C\) is a countable family of \(\tau\)-closed subsets of \(X\), thus \(\mathcal C\subseteq\mathscr B(X,\tau)\). We claim that
\begin{equation}\label{eq:separability_Lp_claim}
\inf_{C\in\mathcal C}\mm(E\Delta C)=0\quad\text{ for every }E\in\mathscr B(X,\tau),
\end{equation}
whence the separability of \(\mm\) follows. To prove the claim, fix \(E\subseteq X\) \(\tau\)-Borel and \(\varepsilon>0\). We can choose
\(n\in\N\) so that \(\mm(E\setminus K_n)\leq\varepsilon\). By the inner regularity of \(\mm\), we can find a \(\tau\)-compact set \(K\subseteq E\cap K_n\)
such that \(\mm((E\cap K_n)\setminus K)\leq\varepsilon\). By the outer regularity of \(\mm\), we can find \(U\in\tau\) such that
\(K\subseteq U\) and \(\mm(U\setminus K)\leq\varepsilon\). Due to the compactness of \(K\), there exist \(y_1,\ldots,y_k\in K\) and \(r_1,\ldots,r_k>0\)
such that \(K\subseteq\bigcup_{i=1}^k\bar B_{r_i}^{\sfd_n}(y_i)\subseteq U\cap K_n\). Moreover, for any \(i=1,\ldots,k\) we can find
\(j_i\in\N\) and \(q_i\in\mathbb Q\cap(r_i,+\infty)\) such that \(\bar B_{r_i}^{\sfd_n}(y_i)\subseteq\bar B_{q_i}^{\sfd_n}(x^n_{j_i})\subseteq U\cap K_n\).
Therefore, we have that \(C\coloneqq\bigcup_{i=1}^k\bar B_{q_i}^{\sfd_n}(x^n_{j_i})\in\mathcal C\) satisfies \(K\subseteq C\subseteq U\),
whence it follows that \(\mm(E\Delta C)\leq 3\varepsilon\). This proves \eqref{eq:separability_Lp_claim}, which gives the statement in the case where \(\tau\)
is metrisable on \(\tau\)-compact sets.

Let us pass to the second case: assume \(\mm(X\setminus{\rm S}_{\mathbb X})=0\). Fix a \(\sfd\)-dense sequence \((y_k)_{k\in\N}\) in \({\rm S}_{\mathbb X}\).
In this case, we define the countable collection \(\mathcal C\) of \(\tau\)-closed subsets of \(X\) as
\[
\mathcal C\coloneqq\bigg\{\bigcup_{j\in F}\bar B_{q_k}^\sfd(y_k)\;\bigg|\;F\subseteq\N\text{ finite},\,(q_k)_{k\in F}\subseteq\mathbb Q\cap(0,+\infty)\bigg\}.
\]
We claim that \eqref{eq:separability_Lp_claim} holds. To prove it, fix any \(E\in\mathscr B(X,\tau)\) and \(\varepsilon>0\). By the outer regularity of \(\mm\),
we can find a \(\tau\)-open set \(U\subseteq X\) such that \(E\subseteq U\) and \(\mm(U\setminus E)\leq\varepsilon\). Since \(\tau\) is coarser than the topology
induced by \(\sfd\), we have that \(U\) is \(\sfd\)-open, thus there exist a subsequence \((y_{k_j})_{j\in\N}\) of \((y_k)_{k\in\N}\) and a sequence of radii
\((q_j)_{j\in\N}\subseteq\mathbb Q\cap(0,+\infty)\) such that \(E\cap{\rm S}_{\mathbb X}\subseteq\bigcup_{j\in\N}\bar B_{q_j}^\sfd(y_{k_j})\subseteq U\).
Thanks to the continuity from below of \(\mm\), we can thus find \(N\in\N\) such that the set \(C\coloneqq\bigcup_{j=1}^N\bar B_{q_j}^\sfd(y_{k_j})\in\mathcal C\)
satisfies \(\mm(E\Delta C)\leq 2\varepsilon\). This proves \eqref{eq:separability_Lp_claim}, thus the statement holds when \(\mm(X\setminus{\rm S}_{\mathbb X})=0\).
\end{proof}

Observe that the second assumption in Lemma \ref{lem:separability_Lp} is verified, for instance, when \((X,\sfd)\) is separable. We also point out that the first
assumption can be relaxed to: \emph{for some \(D\in\mathscr B(X,\tau)\) such that \(\mm\) is concentrated on \(D\), the topology \(\tau\) is metrisable
on every \(\tau\)-compact subset of \(D\).}
A significant example of a non-metrisable topology \(\tau\) that is metrisable on all \(\tau\)-compact sets is the weak\(^*\) topology of the dual \(\B'\)
of a separable infinite-dimensional Banach space \(\B\).
\subsubsection{Compactification of an extended metric-topological space}
A very important feature of the category of extended metric-topological spaces is that it is closed under a notion
of \emph{compactification}, devised in this framework by Savar\'{e} \cite[Section 2.1.7]{Savare22} via the Gelfand
theory of Banach algebras. By virtue of the existence of compactifications, one can reduce many proofs to the compact case.
\medskip

Let us briefly recall the construction of the Gelfand compactification of an e.m.t.\ space \((X,\tau,\sfd)\).
By a \textbf{character} of \(\Lip_b(X,\tau,\sfd)\) we mean a non-zero element \(\varphi\) of the dual Banach space
of the normed space \((\Lip_b(X,\tau,\sfd),\|\cdot\|_{C_b(X,\tau)})\) that satisfies
\begin{equation}\label{eq:charact_mult}
\varphi(fg)=\varphi(f)\varphi(g)\quad\text{ for every }f,g\in\Lip_b(X,\tau,\sfd).
\end{equation}
We denote by \(\hat X\) the set of all characters of \(\Lip_b(X,\tau,\sfd)\). We equip \(\hat X\) with the topology
\(\hat\tau\) obtained by restricting the weak\(^*\) topology of the dual of \((\Lip_b(X,\tau,\sfd),\|\cdot\|_{C_b(X,\tau)})\)
to \(\hat X\). The canonical embedding map \(\iota\colon X\hookrightarrow\hat X\) is given by
\[
\iota(x)(f)\coloneqq f(x)\quad\text{ for every }x\in X\text{ and }f\in\Lip_b(X,\tau,\sfd).
\]
Moreover, the \textbf{Gelfand transform} \(\Gamma\colon\Lip_b(X,\tau,\sfd)\to C_b(\hat X,\hat\tau)\) is defined as
\[
\Gamma(f)(\varphi)\coloneqq\varphi(f)\quad\text{ for every }f\in\Lip_b(X,\tau,\sfd)\text{ and }\varphi\in\hat X.
\]
Note that \(\Gamma(f)\circ\iota=f\) for every \(f\in\Lip_b(X,\tau,\sfd)\). Finally, we define the extended distance \(\hat\sfd\) as
\[
\hat\sfd(\varphi,\psi)\coloneqq\sup\big\{|\varphi(f)-\psi(f)|\;\big|\;
f\in\Lip_{b,1}(X,\tau,\sfd)\big\}\quad\text{ for every }\varphi,\psi\in\hat X.
\]
\begin{remark}\label{rmk:character_on_const}{\rm
We claim that
\[
\varphi(\lambda\1_X)=\lambda\quad\text{ for every }\varphi\in\hat X\text{ and }\lambda\in\R.
\]
Indeed, \eqref{eq:charact_mult} and the linearity of \(\varphi\) guarantee that
\(\varphi(\lambda\1_X)\varphi(\1_X)=\varphi(\lambda\1_X)=\lambda\,\varphi(\1_X)\), and
\eqref{eq:charact_mult} implies also that \(\varphi(\1_X)\neq 0\) (otherwise, we would have that \(\varphi(f)=
\varphi(f\1_X)=\varphi(f)\varphi(\1_X)=0\) for every \(f\in\Lip_b(X,\tau,\sfd)\), contradicting the fact
that \(\varphi\neq 0\)). It follows that \(\varphi(\lambda\1_X)=\lambda\).
\fr}\end{remark}

The objects \(\hat X\), \(\hat\tau\), \(\iota\), \(\Gamma\) and \(\hat\sfd\) defined above
have the following properties \cite[Theorem 2.1.34]{Savare22}:
\begin{theorem}[Gelfand compactification of an e.m.t.\ space]
Let \((X,\tau,\sfd)\) be an e.m.t.\ space. Then \((\hat X,\hat\tau,\hat\sfd)\) is an e.m.t.\ space
and \((\hat X,\hat\tau)\) is compact. Moreover, the following conditions hold:
\begin{itemize}
\item[\({\rm i)}\)] The map \(\iota\) is a homeomorphism between \((X,\tau)\) and its image \(\iota(X)\) in \((\hat X,\hat\tau)\).
\item[\({\rm ii)}\)] The set \(\iota(X)\) is a dense subset of \((\hat X,\hat\tau)\).
\item[\({\rm iii)}\)] We have that \(\hat\sfd(\iota(x),\iota(y))=\sfd(x,y)\) for every \(x,y\in X\).
\end{itemize}
We say that \((\hat X,\hat\tau,\hat\sfd)\) is the \textbf{compactification} of \((X,\tau,\sfd)\),
with embedding \(\iota\colon X\hookrightarrow\hat X\).
\end{theorem}

If \(\mathbb X=(X,\tau,\sfd,\mm)\) is an e.m.t.m.\ space and \((\hat X,\hat\tau,\hat\sfd)\) denotes the compactification of
\((X,\tau,\sfd)\), with embedding \(\iota\colon X\hookrightarrow\hat X\), then we define the measure \(\hat\mm\) on \(\hat X\) as
\[
\hat\mm\coloneqq\iota_\#\mm\in\mathcal M_+(\hat X,\hat\tau).
\]
The fact that \(\hat\mm\) is a Radon measure follows from the continuity of \(\iota\) (as all continuous maps are
universally Lusin measurable). Given any exponent \(p\in[1,\infty]\), we have that \(\iota\colon X\hookrightarrow\hat X\)
induces via pre-composition a map \(\iota^*\colon L^p(\hat\mm)\to L^p(\mm)\) (sending the \(\hat\mm\)-a.e.\ equivalence class of
a \(p\)-integrable Borel function \(\hat f\colon\hat X\to\R\) to the \(\mm\)-a.e.\ equivalence class of \(\hat f\circ\iota\)),
which is an isomorphism of Banach spaces and of Riesz spaces (and also of Banach algebras when \(p=\infty\)). 
\medskip

Albeit implicitly contained in \cite{Savare22}, we isolate the following result for the reader's convenience:
\begin{lemma}\label{lem:Gelfand_isom}
Let \((X,\tau,\sfd)\) be an e.m.t.\ space. Let \((\hat X,\hat\tau,\hat\sfd)\) be its compactification, with embedding
\(\iota\colon X\hookrightarrow\hat X\). Then the Gelfand transform \(\Gamma\) maps \(\Lip_b(X,\tau,\sfd)\) to
\(\Lip_b(\hat X,\hat\tau,\hat\sfd)\). Moreover, it holds that \(\Gamma\colon\Lip_b(X,\tau,\sfd)\to\Lip_b(\hat X,\hat\tau,\hat\sfd)\)
is an isomorphism of Banach algebras, with inverse given by
\begin{equation}\label{eq:inverse_Gamma}
\Lip_b(\hat X,\hat\tau,\hat\sfd)\ni\hat f\longmapsto\hat f\circ\iota\in\Lip_b(X,\tau,\sfd).
\end{equation}
\end{lemma}
\begin{proof}
Fix \(f\in\Lip_b(X,\tau,\sfd)\). If \(\Lip(f,\sfd)=0\), then \(f\) is constant, thus \(|\Gamma(f)(\varphi)-\Gamma(f)(\psi)|=0\)
for every \(\varphi,\psi\in\hat X\) by Remark \ref{rmk:character_on_const}. If \(\Lip(f,\sfd)>0\), then \(\tilde f\coloneqq\Lip(f,\sfd)^{-1}f\in\Lip_{b,1}(X,\tau,\sfd)\), thus
\[
|\Gamma(f)(\varphi)-\Gamma(f)(\psi)|=|\varphi(f)-\psi(f)|=\Lip(f,\sfd)|\varphi(\tilde f)-\psi(\tilde f)|\leq\Lip(f,\sfd)\hat\sfd(\varphi,\psi)
\]
for all \(\varphi,\psi\in\hat X\). All in all, we have that \(\Gamma(f)\in\Lip_b(\hat X,\hat\tau,\hat\sfd)\) and \(\Lip(\Gamma(f),\hat\sfd)\leq\Lip(f,\sfd)\). Also,
\[\begin{split}
\Lip(\Gamma(f),\hat\sfd)&\geq\Lip(\Gamma(f),\iota(X),\hat\sfd)
=\sup\bigg\{\frac{|\Gamma(f)(\iota(x))-\Gamma(f)(\iota(y))|}{\hat\sfd(\iota(x),\iota(y))}\;\bigg|\;x,y\in X,\,x\neq y\bigg\}\\
&=\sup\bigg\{\frac{|f(x)-f(y)|}{\sfd(x,y)}\;\bigg|\;x,y\in X,\,x\neq y\bigg\}=\Lip(f,\sfd),
\end{split}\]
so that \(\Lip(\Gamma(f),\hat\sfd)=\Lip(f,\sfd)\). Moreover, since \(\Gamma(f)\) is \(\hat\tau\)-continuous and \(\iota(X)\) is \(\hat\tau\)-dense in \(\hat X\),
we have that \(\|\Gamma(f)\|_{C_b(\hat X,\hat\tau)}=\sup_{x\in X}|\Gamma(f)(\iota(x))|=\sup_{x\in X}|f(x)|=\|f\|_{C_b(X,\tau)}\). Hence, it holds that
\(\Gamma(\Lip_b(X,\tau,\sfd))\subseteq\Lip_b(\hat X,\hat\tau,\hat\sfd)\) and \(\|\Gamma(f)\|_{\Lip_b(\hat X,\hat\tau,\hat\sfd)}=\|f\|_{\Lip_b(X,\tau,\sfd)}\)
for every \(f\in\Lip_b(X,\tau,\sfd)\).

Now, denote by \(I\) the map in \eqref{eq:inverse_Gamma}. Clearly, \(\Gamma\) and \(I\) are homomorphisms of Banach algebras. As we
already pointed out, we have that \((I\circ\Gamma)(f)=\Gamma(f)\circ\iota=f\) for every \(f\in\Lip_b(X,\tau,\sfd)\), which means that
\(I\circ\Gamma={\rm id}_{\Lip_b(X,\tau,\sfd)}\). Conversely, for any \(\hat f\in\Lip_b(\hat X,\hat\tau,\hat\sfd)\) we have that
\[
(\Gamma\circ I)(\hat f)(\iota(x))=\Gamma(\hat f\circ\iota)(\iota(x))=\iota(x)(\hat f\circ\iota)=\hat f(\iota(x))\quad\text{ for every }x\in X,
\]
which gives that \((\Gamma\circ I)(\hat f)|_{\iota(X)}=\hat f|_{\iota(X)}\). Since \((\Gamma\circ I)(\hat f)\), \(\hat f\) are \(\hat\tau\)-continuous and \(\iota(X)\)
is \(\hat\tau\)-dense in \(\hat X\), we conclude that \((\Gamma\circ I)(\hat f)=\hat f\), thus \(\Gamma\circ I={\rm id}_{\Lip_b(\hat X,\hat\tau,\hat\sfd)}\). The proof is complete.
\end{proof}

Let us also point out that for any given function \(f\in\Lip_b(X,\tau,\sfd)\) it holds that
\begin{equation}\label{eq:ineq_lip_a_cpt}
\lip_\sfd(f)(x)\leq\lip_{\hat\sfd}(\Gamma(f))(\iota(x))\quad\text{ for every }x\in X,
\end{equation}
but it might happen that the inequality in \eqref{eq:ineq_lip_a_cpt} is not an equality. Hence, we have that
\begin{equation}\label{eq:ineq_lip_a_cpt_a.e.}
\lip_\sfd(f)\leq\iota^*\big(\lip_{\hat\sfd}(\Gamma(f))\big)\quad\text{ holds }\mm\text{-a.e.\ on }X,
\text{ for every }f\in\Lip_b(X,\tau,\sfd),
\end{equation}
but it might happen that the \(\mm\)-a.e.\ inequality in \eqref{eq:ineq_lip_a_cpt_a.e.} is not an \(\mm\)-a.e.\ equality.
\subsubsection{Examples of extended metric-topological spaces}\label{s:examples_emtms}
We collect here many examples of e.m.t.(m.) spaces. As observed in \cite[Section 13]{AmbrosioErbarSavare16} and \cite[Section 2.1.3]{Savare22}, the following are e.m.t.m.\ spaces:
\begin{itemize}
\item A metric space \((X,\sfd)\) together with the topology \(\tau_\sfd\) induced by \(\sfd\) and a finite Radon measure \(\mm\geq 0\) on \(X\).
In particular, a complete and separable metric space \((X,\sfd)\) together with the topology $\tau_\sfd$ and a finite Borel measure \(\mm\geq 0\) on \(X\)
(as all finite Borel measures on a complete and separable metric space are Radon). The latter are often referred to as \textbf{metric measure spaces} in the literature.
\item A Banach space \(\B\) together with the distance induced by its norm, the weak topology \(\tau_w\) and a finite Radon measure on \((\B,\tau_w)\).
\item The dual \(\B'\) of a Banach space \(\B\) together with the distance induced by the dual  norm, the weak\(^*\) topology \(\tau_{w^*}\) and a finite Radon measure
on \((\B',\tau_{w^*})\). We point out that if \(\B\) is separable, then \((\B',\tau_{w^*})\) is a Lusin space \cite[Corollary 1 at p.\ 115]{Schwartz73}, so that every
finite Borel measure on \((\B',\tau_{w^*})\) is Radon.
\item An \textbf{abstract Wiener space}, i.e.\ a separable Banach space \(X\) together with a (centered, non-degenerate) Gaussian measure \(\gamma\) and the extended distance
that is induced by the \textbf{Cameron--Martin space} of \((X,\gamma)\); see e.g.\ \cite{Bogachev15}.
\item Other important examples of e.m.t.m.\ spaces are given by some `extended sub-Finsler-type structures' \cite[Example 2.1.3]{Savare22} or the so-called
\emph{configuration spaces} \cite[Section 13.3]{AmbrosioErbarSavare16}.
\end{itemize}

The class of e.m.t.\ spaces in the first bullet point above (i.e.\ metric spaces equipped with the topology induced by the distance) shows that, in a sense,
the theory of e.m.t.\ spaces is an extension of that of metric spaces. On the other hand, as it is evident from Example \ref{ex:Schultz's_space} below (which
was pointed out to us by Timo Schultz), the category of e.m.t.\ spaces encompasses also the one of Tychonoff spaces, but in this paper we will not investigate
further in this direction.
\begin{example}[`Purely-topological' e.m.t.\ space]\label{ex:Schultz's_space}{\rm
Let \((X,\tau)\) be a given Tychonoff space. Let us denote by \(\sfd_{\rm discr}\) the \textbf{discrete distance} on \(X\), i.e.\ we define
\[
\sfd_{\rm discr}(x,y)\coloneqq\left\{\begin{array}{ll}
1\\
0
\end{array}\quad\begin{array}{ll}
\text{ for every }x,y\in X\text{ with }x\neq y,\\
\text{ for every }x,y\in X\text{ with }x=y.
\end{array}\right.
\]
Then \((X,\tau,\sfd_{\rm discr})\) is an e.m.t.\ space. Indeed, it can be readily checked that the \(\sfd_{\rm discr}\)-Lipschitz functions
\(f\colon X\to\R\) are exactly the bounded functions and \(\Lip(f,\sfd_{\rm discr})={\rm Osc}_X(f)\), in particular
\[
\Lip_b(X,\tau,\sfd_{\rm discr})=C_b(X,\tau),\qquad\|\cdot\|_{\Lip_b(X,\tau,\sfd_{\rm discr})}={\rm Osc}_X(\cdot)+\|\cdot\|_{C_b(X,\tau)}.
\]
Therefore, the complete regularity of \((X,\tau)\) ensures that the initial topology of \(\Lip_b(X,\tau,\sfd_{\rm discr})\) coincides with \(\tau\)
(so that Definition \ref{def:emtm} i) holds), and for any two distinct points \(x,y\in X\) we can find (as \((X,\tau)\) is completely Hausdorff)
a \(\tau\)-continuous function \(f\colon X\to[0,1]\) such that \(f(x)=1\) and \(f(y)=0\), so that \(\sfd_{\rm discr}(x,y)=1=|f(x)-f(y)|\) (whence
Definition \ref{def:emtm} ii) follows).
\fr}\end{example}
\begin{example}\label{ex:mixed_Schultz_space}{\rm
We endow \(X\coloneqq[0,1]^2\subseteq\R^2\) with the Euclidean topology \(\tau\) and the distance
\[
\sfd((x,t),(y,s))\coloneqq\max\{\sfd_{\rm discr}(x,y),\sfd_{\rm Eucl}(t,s)\}\quad\text{ for every }(x,t),(y,s)\in X,
\]
where \(\sfd_{\rm discr}\) denotes the discrete distance, while \(\sfd_{\rm Eucl}(t,s)\coloneqq|t-s|\) is the Euclidean distance.
One can easily check that \((X,\tau,\sfd)\) is an e.m.t.\ space, and that a given function \(f\colon X\to\R\) belongs to the space
\(\Lip_b(X,\tau,\sfd)\) if and only if it is \(\tau\)-continuous, \(f(x,\cdot)\in\Lip_b([0,1],\sfd_{\rm Eucl})\) for every \(x\in[0,1]\)
and \(\sup_{x\in[0,1]}\Lip(f(x,\cdot),\sfd_{\rm Eucl})<+\infty\). Moreover, straightforward arguments show that
\begin{equation}\label{eq:mixed_Schultz_space}
\Lip(f,\sfd)={\rm Osc}_X(f)\vee\sup_{x\in[0,1]}\Lip(f(x,\cdot),\sfd_{\rm Eucl})
\end{equation}
for every \(f\in\Lip_b(X,\tau,\sfd)\).
\fr}\end{example}

Whereas the Banach algebra \(\Lip_b(X,\sfd)\) associated to a metric space \((X,\sfd)\) is (isometrically isomorphic to)
a dual Banach space (see \cite[Corollary 3.4]{Weaver2018}), in the more general setting of e.m.t.\ spaces we can provide
examples where \(\Lip_b(X,\tau,\sfd)\) is not isometrically isomorphic (and not even just isomorphic) to a dual Banach space,
see Proposition \ref{prop:ex_no_predual} below. The possible non-existence of a predual of \(\Lip_b(X,\tau,\sfd)\) will have
an important role in Definition \ref{def:w-type_cont}.
\begin{proposition}\label{prop:ex_no_predual}
Let \((K,\tau)\) be an infinite compact metrisable topological space. Let \(\sfd_{\rm discr}\) denote the discrete distance on \(K\).
Then \(\Lip_b(K,\tau,\sfd_{\rm discr})\) is not isomorphic to a dual Banach space.
\end{proposition}
\begin{proof}
We recall from Example \ref{ex:Schultz's_space} that \((K,\tau,\sfd_{\rm discr})\) is an extended
metric-topological space that satisfies \({\rm L}\coloneqq\Lip_b(K,\tau,\sfd_{\rm discr})=C(K,\tau)\) and
\(\|f\|_{\rm L}\coloneqq\|f\|_{\Lip_b(K,\tau,\sfd_{\rm discr})}={\rm Osc}_K(f)+\|f\|_{C(K,\tau)}\) for every
\(f\in{\rm L}\). Note that \(\|f\|_{C(K,\tau)}\leq\|f\|_{\rm L}\leq 3\|f\|_{C(K,\tau)}\) for every \(f\in{\rm L}\).
Since \((K,\tau)\) is a compact metrisable topological space, it holds that \(C(K,\tau)\) is separable \cite[Theorem 4.1.3]{Albiac-Kalton}
and thus \({\rm L}\) is separable.
Since \(\tau\) is a Hausdorff topology, by virtue of Remark \ref{rmk:inf_open_Hausd} below we can find a sequence \((U_n)_{n\in\N}\subseteq\tau\) of pairwise disjoint sets such that each set \(U_n\) contains at least two distinct points
\(x_n\) and \(y_n\). Since \((K,\tau)\) is completely regular, for any \(n\in\N\) we can find a \(\tau\)-continuous function \(f_n\colon K\to[-1,1]\)
such that \(\{f_n\neq 0\}\subseteq U_n\), \(f_n(x_n)=1\) and \(f_n(y_n)=-1\). Letting \(c_{00}\) be the vector space of real-valued
sequences \(a=(a_n)_n\) satisfying \(a_n=0\) for all but finitely many indices \(n\in\N\), we define the linear operator
\(\phi\colon c_{00}\to{\rm L}\) as
\[
\phi(a)\coloneqq\frac{1}{3}\sum_{\substack{n\in\N:\\a_n\neq 0}}a_n f_n\in{\rm L}\quad\text{ for every }a=(a_n)_n\in c_{00}.
\]
Recall that \(c_{00}\) is a dense subspace of the Banach space \((c_0,\|\cdot\|_{c_0})\), where \(c_0\) is the
space of real-valued sequences \(a=(a_n)_n\) with \(\lim_n a_n=0\), and \(\|\cdot\|_{c_0}\) is the supremum norm
\(\|a\|_{c_0}\coloneqq\sup_n|a_n|\). Given that \(\|\phi(a)\|_{\rm L}=\|a\|_{c_0}\) for every \(a\in c_{00}\) by construction,
we have that \(\phi\) can be uniquely extended to a linear isometry \(\bar\phi\colon c_0\to{\rm L}\).
Since \(c_0\) cannot be embedded in a separable dual Banach space (see \cite[Theorem 6.3.7]{Albiac-Kalton}
or \cite[Theorem 4]{Bessaga-Pelczynksi58}), we can finally conclude that \({\rm L}\) is not isomorphic to a dual Banach space.
\end{proof}
\begin{remark}\label{rmk:inf_open_Hausd}{\rm
If \((X,\tau)\) is an infinite Hausdorff space, then there exists a sequence \((U_n)_{n\in\N}\) of pairwise disjoint
non-empty open subsets of \(X\). To prove this claim, we distinguish two cases. If \(X\) has infinitely many isolated points,
take a sequence \((x_n)_{n\in\N}\) of pairwise distinct isolated points of \(X\), and note that letting \(U_n\coloneqq\{x_n\}\)
for every \(n\in\N\) does the job. If \(X\) has only finitely many isolated points, then the set \(\tilde X\) of all accumulation
points is an open subset of \(X\) (by the Hausdorff assumption); since each neighbourhood of an accumulation point is infinite
(again, by the Hausdorff assumption), we can construct recursively a sequence \((U_n)_{n\in\N}\) of pairwise disjoint infinite
open subsets of \(\tilde X\), which are -- a fortiori -- open subsets of \(X\). The claim is proved.
\fr}\end{remark}
\begin{example}[An e.m.t.m.\ space whose reference measure is non-separable]\label{ex:non-sep_m}{\rm
Let \((X,\tau,\sfd_{\rm discr})\) be the product \(X\coloneqq[0,1]^{\mathfrak c}\) of the continuum of intervals
together with the product topology \(\tau\) and the discrete distance \(\sfd_{\rm discr}\). Since \((X,\tau)\)
is compact and Hausdorff, we know from Example \ref{ex:Schultz's_space} that \((X,\tau,\sfd_{\rm discr})\) is an
e.m.t.\ space. Moreover, we equip \((X,\tau)\) with the probability Radon measure \(\mm\) obtained as the product
of the one-dimensional Lebesgue measures; to be precise, the product measure of the Lebesgue measures is defined
on the product \(\sigma\)-algebra \(\bigotimes_{t\in\mathfrak c}\mathscr B([0,1])\), but it extends to a Radon measure
\(\mm\) on \(\mathscr B(X,\tau)\) thanks to \cite[Theorem 7.14.3]{Bogachev07}. However, the measure \(\mm\) of the
e.m.t.m.\ space \((X,\tau,\sfd_{\rm discr},\mm)\) is not separable, see \cite[Section 7.14(iv)]{Bogachev07}.
\fr}\end{example}
\subsubsection{Rectifiable arcs and path integrals}
Let \((X,\tau,\sfd)\) be an e.m.t.\ space. As in \cite[Section 2.2.1]{Savare22}, we endow the space \(C([0,1];(X,\tau))\) of all \(\tau\)-continuous curves
\(\gamma\colon[0,1]\to X\) with the compact-open topology \(\tau_C\) and with the extended distance \(\sfd_C\colon C([0,1];(X,\tau))\times C([0,1];(X,\tau))\to[0,+\infty]\),
which we define as
\[
\sfd_C(\gamma,\sigma)\coloneqq\sup_{t\in[0,1]}\sfd(\gamma_t,\sigma_t)\quad\text{ for every }\gamma,\sigma\in C([0,1];(X,\tau)).
\]
Then \((C([0,1];(X,\tau)),\tau_C,\sfd_C)\) is an extended metric-topological space \cite[Proposition 2.2.2]{Savare22}. We recall that a subbasis for the compact-open
topology \(\tau_C\) is given by the family of sets
\[
\big\{S(K,V)\;\big|\;K\subseteq[0,1]\text{ compact},\,V\in\tau\big\},
\]
where we denote \(S(K,V)\coloneqq\big\{\gamma\in C([0,1];(X,\tau))\,:\,\gamma(K)\subseteq V\big\}\).
\medskip

Following \cite[Section 2.2.2]{Savare22}, we denote by \(\Sigma\) the set of all continuous, non-decreasing, surjective maps \(\phi\colon[0,1]\to[0,1]\). Let us consider
the following equivalence relation on \(C([0,1];(X,\tau))\): given any \(\gamma,\sigma\in C([0,1];(X,\tau))\), we declare that \(\gamma\sim\sigma\) if and only if
there exist \(\phi_\gamma,\phi_\sigma\in\Sigma\) such that
\[
\gamma\circ\phi_\gamma=\sigma\circ\phi_\sigma.
\]
We endow the associated quotient space \({\rm A}(X,\tau)\coloneqq C([0,1];(X,\tau))/\sim\) with the quotient topology \(\tau_{\rm A}\) induced by \(\tau_C\). The elements of
\({\rm A}(X,\tau)\) are called \textbf{arcs}. We denote by \([\gamma]\in {\rm A}(X,\tau)\) the equivalence class of a curve \(\gamma\in C([0,1];(X,\tau))\).
We define the subspace \({\rm A}(X,\sfd)\subseteq{\rm A}(X,\tau)\) as
\[
{\rm A}(X,\sfd)\coloneqq\big\{[\gamma]\;\big|\;\gamma\in C([0,1];(X,\sfd))\big\}.
\]
Letting \(\sfd_{\rm A}\colon{\rm A}(X,\sfd)\times{\rm A}(X,\sfd)\to[0,+\infty]\) be the extended distance on \({\rm A}(X,\sfd)\) given by
\[
\sfd_{\rm A}(\gamma,\sigma)\coloneqq\inf\big\{\sfd_C(\tilde\gamma,\tilde\sigma)\;\big|\;\tilde\gamma,\tilde\sigma\in C([0,1];(X,\tau)),\,[\tilde\gamma]=\gamma,\,[\tilde\sigma]=\sigma\big\}
\quad\text{ for every }\gamma,\sigma\in{\rm A}(X,\sfd),
\]
we have that \(({\rm A}(X,\sfd),\tau_{\rm A},\sfd_{\rm A})\) is an extended metric-topological space \cite[Proposition 2.2.6]{Savare22}.
\medskip

Given a curve \(\gamma\in C([0,1];(X,\sfd))\) and any \(t\in[0,1]\), the \textbf{\(\sfd\)-variation} of \(\gamma\)
on \([0,t]\) is defined as
\[
V_\gamma(t)\coloneqq\sup\bigg\{\sum_{i=1}^n\sfd(\gamma_{t_i},\gamma_{t_{i-1}})\;\bigg|\;n\in\N,\,\{t_i\}_{i=0}^n\subseteq[0,1],\,
t_0<t_1<\ldots<t_n\bigg\}\in[0,+\infty].
\]
The \textbf{\(\sfd\)-length} of \(\gamma\) is defined as \(\ell(\gamma)\coloneqq V_\gamma(1)\in[0,+\infty]\). As in \cite[Lemma 2.2.8]{Savare22}, we set
\[
{\rm BVC}([0,1];(X,\sfd))\coloneqq\big\{\gamma\in C([0,1];(X,\sfd))\;\big|\;\ell(\gamma)<+\infty\big\}.
\]
Since \(\ell\) is \(\tau_C\)-lower semicontinuous, the space \({\rm BVC}([0,1];(X,\sfd))\) is an \(F_\sigma\) subset of \(C([0,1];(X,\tau))\). We say that
a curve \(\gamma\in{\rm BVC}([0,1];(X,\sfd))\) has \textbf{constant \(\sfd\)-speed} if \(V_\gamma(t)=\ell(\gamma)t\) holds for every \(t\in[0,1]\).
For any given \(\gamma\in{\rm BVC}([0,1];(X,\sfd))\), there exists a unique \(\ell(\gamma)\)-Lipschitz curve \(R_\gamma\in{\rm BVC}([0,1];(X,\sfd))\) having
constant \(\sfd\)-speed such that
\[
\gamma(t)=R_\gamma(\ell(\gamma)^{-1}V_\gamma(t))\quad\text{ for every }t\in[0,1],
\]
with the convention that \(\ell(\gamma)^{-1}V_\gamma(t)=0\) if \(\ell(\gamma)=0\). Then it holds that \([\gamma]=[R_\gamma]\) and we say that
\(R_\gamma\) is the \textbf{arc-length parameterisation} of \(\gamma\). The space of \textbf{rectifiable arcs} is given by
\[
{\rm RA}(X,\sfd)\coloneqq\big\{[\gamma]\;\big|\;\gamma\in{\rm BVC}([0,1];(X,\sfd))\big\}\subseteq{\rm A}(X,\sfd).
\]
Then \(({\rm RA}(X,\sfd),\tau_{\rm A},\sfd_{\rm A})\) is an extended metric-topological space. Given \(\gamma,\sigma\in{\rm BVC}([0,1];(X,\sfd))\), we have that
\([\gamma]=[\sigma]\) if and only if \(R_\gamma=R_\sigma\) \cite[Lemma 2.2.11(b)]{Savare22}, thus we can unambiguously write \(R_\gamma\) for \(\gamma\in{\rm RA}(X,\sfd)\).
Similarly, we can write \(\gamma_0\), \(\gamma_1\) and \(\ell(\gamma)\) for \(\gamma\in{\rm RA}(X,\sfd)\), and
\begin{equation}\label{eq:ell_lsc}
{\rm RA}(X,\sfd)\ni\gamma\mapsto\ell(\gamma)\quad\text{ is }\tau_{\rm A}\text{-lower semicontinuous,}
\end{equation}
see \cite[Lemma 2.2.11(d)]{Savare22}. Given any \(\gamma\in{\rm RA}(X,\sfd)\) and a Borel function \(f\colon(X,\tau)\to\R\) such that \(f\circ R_\gamma\in L^1(0,1)\) (or a Borel function \(f\colon X\to[0,+\infty]\)), the \textbf{path integral} of \(f\) over \(\gamma\) is given by
\[
\int_\gamma f\coloneqq\ell(\gamma)\int_0^1 f(R_\gamma(t))\,\d t.
\]
When \(f\) is bounded, \(({\rm RA}(X,\sfd),\tau_{\rm A})\ni\gamma\mapsto\int_\gamma f\in\R\) is Borel measurable \cite[Theorem 2.2.13(e)]{Savare22}.
\medskip

For any \(t\in[0,1]\), the \textbf{arc-length evaluation map} \(\hat{\sf e}_t\colon{\rm RA}(X,\sfd)\to X\) \textbf{at time \(t\)} is defined as
\[
\hat{\sf e}_t(\gamma)\coloneqq R_\gamma(t)\quad\text{ for every }\gamma\in{\rm RA}(X,\sfd).
\]
We also introduce the \textbf{arc-length evaluation map} \(\hat{\sf e}\colon{\rm RA}(X,\sfd)\times[0,1]\to X\), given by
\begin{equation}\label{eq:def_e}
\hat{\sf e}(\gamma,t)\coloneqq\hat{\sf e}_t(\gamma)=R_\gamma(t)\quad\text{ for every }\gamma\in{\rm RA}(X,\sfd)\text{ and }t\in[0,1].
\end{equation}
Let us now prove some technical results, concerning the measurability properties of \(\hat{\sf e}\) and of a map that
describes the derivative of a continuous Lipschitz function along rectifiable arcs, which we will use in Section \ref{s:W=B}.
\begin{lemma}\label{lem:hat_e_univ_Lusin_meas}
Let \((X,\tau,\sfd)\) be an e.m.t.\ space. Then it holds that \(\hat{\sf e}\colon{\rm RA}(X,\sfd)\times[0,1]\to X\) is universally Lusin measurable
(when \({\rm RA}(X,\sfd)\times[0,1]\) is equipped with the product topology).
\end{lemma}
\begin{proof}
First of all, we claim that if \(((\gamma^i,t^i))_{i\in I}\subseteq{\rm RA}(X,\sfd)\times[0,1]\) is a given net converging to \((\gamma,t)\in{\rm RA}(X,\sfd)\times[0,1]\)
such that \(\lim_{i\in I}\ell(\gamma^i)=\ell(\gamma)\), then
\begin{equation}\label{eq:meas_hat_e_aux}
\lim_{i\in I}\hat{\sf e}(\gamma^i,t^i)=\hat{\sf e}(\gamma,t).
\end{equation}
To prove it, fix a neighbourhood \(V\in\tau\) of \(\hat{\sf e}(\gamma,t)\). By the complete regularity of \(\tau\), we can find a neighbourhood \(U\in\tau\) of \(R_\gamma(t)=\hat{\sf e}(\gamma,t)\)
whose \(\tau\)-closure \(\bar U\) is contained in \(V\). Since the curve \(R_\gamma\colon[0,1]\to X\) is \(\tau\)-continuous and \(\lim_{i\in I}t^i=t\), there exists \(i_0\in I\) such that
\(R_\gamma(t^i)\in U\) for every \(i\in I\) with \(i_0\preceq i\). Letting \(K\) denote the closure of \(\{t^i\,:\,i\in I,\,i_0\preceq i\}\), which is a compact subset of \([0,1]\), we have
that \(t\in K\) and \(R_\gamma(s)\in\bar U\subseteq V\) for every \(s\in K\), thus \(S(K,V)\in\tau_C\) is a neighbourhood of \(R_\gamma\). Since \(\lim_{i\in I}R_{\gamma^i}=R_\gamma\)
in \(\big(C([0,1];(X,\tau)),\tau_C\big)\) by \cite[Theorem 2.2.13(a)]{Savare22}, we deduce that there exists \(i_1\in I\) with \(i_0\preceq i_1\) and \(R_{\gamma^i}\in S(K,V)\) for every
\(i\in I\) with \(i_1\preceq i\). It follows that \(\hat{\sf e}(\gamma^i,t^i)=R_{\gamma^i}(t^i)\in V\) for every \(i\in I\) with \(i_1\preceq i\), which shows that \eqref{eq:meas_hat_e_aux} holds.

Now let \(\mu\in\mathcal M_+({\rm RA}(X,\sfd)\times[0,1])\) be fixed. By \eqref{eq:ell_lsc}, the map
\({\rm RA}(X,\sfd)\times[0,1]\ni(\gamma,t)\mapsto\ell(\gamma)\) is lower semicontinuous, thus it is Lusin \(\mu\)-measurable by
Remark \ref{rmk:about_univ_Lusin_meas}. Hence, for any \(\varepsilon>0\) we can find a compact set
\(\mathcal K_\varepsilon\subseteq{\rm RA}(X,\sfd)\times[0,1]\) such that
\(\mathcal K_\varepsilon\ni(\gamma,t)\mapsto\ell(\gamma)\) is continuous. The first part of the proof then gives that
\(\hat{\sf e}|_{\mathcal K_\varepsilon}\) is continuous, so that \(\hat{\sf e}\) is universally Lusin measurable.
\end{proof}
\begin{corollary}\label{cor:D_f_univ_Lusin_meas}
Let \((X,\tau,\sfd)\) be an e.m.t.\ space. Let \(f\in\Lip_b(X,\tau,\sfd)\) be given. Let us define the function
\({\rm D}_f\colon{\rm RA}(X,\sfd)\times[0,1]\to\R\) as
\[
{\rm D}_f(\gamma,t)\coloneqq\limsup_{h\to 0}\frac{f(R_\gamma(t+h))-f(R_\gamma(t))}{h}
\quad\text{ for every }\gamma\in{\rm RA}(X,\sfd)\text{ and }t\in[0,1].
\]
Then \({\rm D}_f\) is universally Lusin measurable.
\end{corollary}
\begin{proof}
Note that \({\rm D}_f(\gamma,t)=\lim_{\N\ni n\to\infty}{\rm D}_f^n(\gamma,t)\) for every
\((\gamma,t)\in{\rm RA}(X,\sfd)\times[0,1]\), where we set
\[
{\rm D}_f^n(\gamma,t)\coloneqq\sup\bigg\{\frac{f(R_\gamma(t+h))-f(R_\gamma(t))}{h}\;\bigg|\;h\in(\mathbb Q\setminus\{0\})\cap(-1/n,1/n)\bigg\}
\]
for brevity. Fix \(n\in\N\). Let us enumerate the elements of \((\mathbb Q\setminus\{0\})\cap(-1/n,1/n)\) as \((q_i)_{i\in\N}\). Then
\[
{\rm D}_f^n(\gamma,t)=\lim_{k\to\infty}\max\bigg\{\frac{f(R_\gamma(t+q_i))-f(R_\gamma(t))}{q_i}\;\bigg|\;i=1,\ldots,k\bigg\}
\quad\text{ for all }(\gamma,t)\in{\rm RA}(X,\sfd)\times[0,1].
\]
Since the map \(\hat{\sf e}\) is universally Lusin measurable by Lemma \ref{lem:hat_e_univ_Lusin_meas}, one can easily deduce that
each function \((\gamma,t)\mapsto\max_{i\leq k}(f(R_\gamma(t+q_i))-f(R_\gamma(t)))/q_i\) is universally Lusin measurable.
By taking Remark \ref{rmk:about_univ_Lusin_meas} into account, we can finally conclude that \({\rm D}_f\) is universally Lusin measurable.
\end{proof}

Given \(\gamma\in{\rm RA}(X,\sfd)\) and \(f\in\Lip_b(X,\tau,\sfd)\), we have that \(f\circ R_\gamma\colon[0,1]\to\R\)
is a Lipschitz function, thus in particular it is \(\mathscr L^1\)-a.e.\ differentiable. Therefore, it holds that
\begin{equation}\label{eq:prop_D_f}
{\rm D}_f(\gamma,t)=(f\circ R_\gamma)'(t)\quad\text{ for }\mathscr L^1\text{-a.e.\ }t\in[0,1].
\end{equation}
In particular, it holds that
\begin{equation}\label{eq:ineq_D_f}
|{\rm D}_f(\gamma,t)|\leq\ell(\gamma)(\lip_\sfd(f)\circ R_\gamma)(t)\quad\text{ for }\mathscr L^1\text{-a.e.\ }t\in[0,1].
\end{equation}
\subsubsection{Uniform structure of an extended metric-topological space}
We assume the reader is familiar with the basics of the theory of uniform spaces, for which we refer e.g.\ to \cite{Bourbaki,Bourbaki2}.
It is well known that every completely regular topology is induced by a uniform structure (in fact, completely regular topological spaces are
exactly the uniformisable topological spaces). In the setting of e.m.t.\ spaces, we make a canonical choice of such a uniform structure:
\begin{definition}[Canonical uniform structure of an e.m.t.\ space]
Let \((X,\tau,\sfd)\) be an e.m.t.\ space. Then we define the \textbf{canonical uniformity} of \((X,\tau,\sfd)\) as the uniform structure
\(\mathfrak U_{\tau,\sfd}\) on \(X\) that is induced by the family of semidistances \(\{\delta_f:f\in\Lip_{b,1}(X,\tau,\sfd)\}\), which are defined as
\[
\delta_f(x,y)\coloneqq|f(x)-f(y)|\quad\text{ for every }f\in\Lip_{b,1}(X,\tau,\sfd)\text{ and }x,y\in X.
\]
\end{definition}

It can be readily checked that the following properties are verified:
\begin{itemize}
\item The topology induced by \(\mathfrak U_{\tau,\sfd}\) coincides with \(\tau\).
\item The topology \(\tau\) is metrisable if and only if \(\mathfrak U_{\tau,\sfd}\) has a countable basis of entourages.
\end{itemize}
Moreover, we denote by \(\mathfrak B_{\tau,\sfd}\subseteq\mathfrak U_{\tau,\sfd}\) the family of all \emph{open symmetric entourages} of \(\mathfrak U_{\tau,\sfd}\), i.e.
\[
\mathfrak B_{\tau,\sfd}\coloneqq\big\{\mathcal U\in\mathfrak U_{\tau,\sfd}\cap(\tau\times\tau)\;\big|\;
(y,x)\in\mathcal U\text{ for every }(x,y)\in\mathcal U\big\}.
\]
It holds that \(\mathfrak B_{\tau,\sfd}\) is a basis of entourages for \(\mathfrak U_{\tau,\sfd}\). In the case where \(\tau\) is metrisable,
it is possible to find a countable basis of entourages for \(\mathfrak U_{\tau,\sfd}\) consisting of elements of \(\mathfrak B_{\tau,\sfd}\).
\begin{remark}{\rm
Let \(f\in\Lip_b(X,\tau,\sfd)\) and \(\mathcal U\in\mathfrak B_{\tau,\sfd}\) be given. Then we claim that
\[
\Lip(f,\mathcal U[\cdot],\sfd)\colon X\to[0,\Lip(f,\sfd)]\quad\text{ is }\tau\text{-lower semicontinuous,}
\]
where \(\mathcal U[x]\coloneqq\{y\in X:(x,y)\in\mathcal U\}\) for all \(x\in X\). Indeed, \(\mathcal U[y]\cap\mathcal U[z]\in\tau\)
for every \(y,z\in X\) and
\[
\Lip(f,\mathcal U[x],\sfd)=\sup\bigg\{\frac{|f(y)-f(z)|}{\sfd(y,z)}\;\bigg|\;
y,z\in X,\,y\neq z,\,x\in\mathcal U[y]\cap\mathcal U[z]\bigg\}\quad\text{ for every }x\in X,
\]
so that the function \(\Lip(f,\mathcal U[\cdot],\sfd)\) is \(\tau\)-lower semicontinuous thanks to Remark \ref{rmk:suff_cond_sc}.
\fr}\end{remark}

Let us now discuss how the canonical uniform structure behaves under restriction of the e.m.t.\ space. Let \((X,\tau,\sfd)\) be a given e.m.t.\ space and fix
\(E\in\mathscr B(X,\tau)\). Consider the restricted e.m.t.\ space \((E,\tau_E,\sfd_E)\) (as in \eqref{eq:restr_emtm}). Then it holds that
\begin{equation}\label{eq:unif_struct_restr}
\mathfrak U_{\tau_E,\sfd_E}=\{\mathcal U|_{E\times E}\;|\;\mathcal U\in\mathfrak U_{\tau,\sfd}\},\qquad
\mathfrak B_{\tau_E,\sfd_E}=\{\mathcal U|_{E\times E}\;|\;\mathcal U\in\mathfrak B_{\tau,\sfd}\}.
\end{equation}
The first identity follows easily from the definition of canonical uniformity. The second identity follows from \(\tau_{E\times E}=\tau_E\times\tau_E\)
and from the fact that \(\mathcal U\cap\mathcal U^{-1}\in\mathfrak B_{\tau,\sfd}\) for every \(\mathcal U\in\mathfrak U_{\tau,\sfd}\cap(\tau\times\tau)\),
where we set \(\mathcal U^{-1}\coloneqq\{(y,x):(x,y)\in\mathcal U\}\).
\subsection{Sobolev spaces \texorpdfstring{\(H^{1,p}\)}{H1p} via relaxation}\label{s:H1p}
The first notion of Sobolev space over an e.m.t.m.\ space we consider is the one obtained by \emph{relaxation}, which was introduced
in \cite[Section 3.1]{Savare22} as a generalisation of \cite{Cheeger00,AmbrosioGigliSavare11,AmbrosioGigliSavare11-3}. A function
\(f\in L^p(\mm)\) is declared to be in the Sobolev space \(H^{1,p}(\mathbb X)\) if it is the \(L^p(\mm)\)-limit of a sequence
\((f_n)_n\) of functions in \(\Lip_b(X,\tau,\sfd)\) whose asymptotic slopes \((\lip_\sfd(f_n))_n\) form a bounded sequence in \(L^p(\mm)\).
Namely, following \cite[Definitions 3.1.1 and 3.1.3]{Savare22}:
\begin{definition}[The Sobolev space \(H^{1,p}(\mathbb X)\)]
Let \(\mathbb X=(X,\tau,\sfd,\mm)\) be an e.m.t.m.\ space and \(p\in(1,\infty)\). Then we define the \textbf{Cheeger \(p\)-energy functional}
\(\mathcal E_p\colon L^p(\mm)\to[0,+\infty]\) of \(\mathbb X\) as
\[
\mathcal E_p(f)\coloneqq\inf\bigg\{\liminf_{n\to\infty}\frac{1}{p}\int\lip_\sfd(f_n)^p\,\d\mm\;\bigg|\;(f_n)_n\subseteq\Lip_b(X,\tau,\sfd),\,f_n\to f\text{ in }L^p(\mm)\bigg\}
\]
for all \(f\in L^p(\mm)\). Then we define the \textbf{Sobolev space} \(H^{1,p}(\mathbb X)\) as the finiteness domain of \(\mathcal E_p\), i.e.
\[
H^{1,p}(\mathbb X)\coloneqq\big\{f\in L^p(\mm)\;\big|\;\mathcal E_p(f)<+\infty\big\}.
\]
\end{definition}

The Cheeger \(p\)-energy functional is convex, \(p\)-homogeneous and \(L^p(\mm)\)-lower semicontinuous. The vector subspace \(H^{1,p}(\mathbb X)\) of \(L^p(\mm)\)
is a Banach space with respect to the Sobolev norm
\[
\|f\|_{H^{1,p}(\mathbb X)}\coloneqq\big(\|f\|_{L^p(\mm)}^p+p\,\mathcal E_p(f)\big)^{1/p}\quad\text{ for every }f\in H^{1,p}(\mathbb X).
\]
Also, \(\mathcal E_p\) admits an integral representation, in terms of \emph{relaxed slopes} \cite[Definition 3.1.5]{Savare22}:
\begin{definition}[Relaxed slope]
Let \(\mathbb X=(X,\tau,\sfd,\mm)\) be an e.m.t.m.\ space and \(p\in(1,\infty)\). Let \(f\in L^p(\mm)\) be given. Then we say that a function \(G\in L^p(\mm)^+\)
is a \textbf{\(p\)-relaxed slope} of \(f\) if there exist a sequence \((f_n)_n\subseteq\Lip_b(X,\tau,\sfd)\) and a function \(\tilde G\in L^p(\mm)^+\) such that
the following hold:
\begin{itemize}
\item[\(\rm i)\)] \(f_n\to f\) strongly in \(L^p(\mm)\),
\item[\(\rm ii)\)] \(\lip_\sfd(f_n)\rightharpoonup\tilde G\) weakly in \(L^p(\mm)\),
\item[\(\rm iii)\)] \(\tilde G\leq G\) in the \(\mm\)-a.e.\ sense.
\end{itemize}
\end{definition}

Below, we collect many properties and calculus rules for \(p\)-relaxed slopes (see \cite[Section 3.1.1]{Savare22}).
\begin{itemize}
\item The set of all \(p\)-relaxed slopes of a given \(f\in H^{1,p}(\mathbb X)\) is a closed sublattice of \(L^p(\mm)\).
Its (unique) \(\mm\)-a.e.\ minimal element is denoted by \(|Df|_H\in L^p(\mm)^+\) and is called the \textbf{minimal \(p\)-relaxed slope} of \(f\).
\item The Cheeger \(p\)-energy functional can be represented as
\[
\mathcal E_p(f)=\frac{1}{p}\int|Df|_H^p\,\d\mm\quad\text{ for every }f\in H^{1,p}(\mathbb X).
\]
\item Given any \(f\in H^{1,p}(\mathbb X)\), there exists a sequence \((f_n)_n\subseteq\Lip_b(X,\tau,\sfd)\) such that \(f_n\to f\) and \(\lip_\sfd(f_n)\to|Df|_H\)
strongly in \(L^p(\mm)\).
\item \(\Lip_b(X,\tau,\sfd)\subseteq H^{1,p}(\mathbb X)\), and \(|Df|_H\leq\lip_\sfd(f)\) holds \(\mm\)-a.e.\ for every \(f\in\Lip_b(X,\tau,\sfd)\).
\item We have that \(|D(f+g)|_H\leq|Df|_H+|Dg|_H\) and \(|D(\lambda f)|_H=|\lambda||Df|_H\) hold \(\mm\)-a.e.\ for every \(f,g\in H^{1,p}(\mathbb X)\) and \(\lambda\in\R\).
\item \textsc{Locality property.} If \(f\in H^{1,p}(\mathbb X)\) and \(N\subseteq\R\) is a Borel set with \(\mathscr L^1(N)=0\), then
\[
|Df|_H=0\quad\text{ holds }\mm\text{-a.e.\ on }f^{-1}(N).
\]
In particular, \(|Df|_H=|Dg|_H\) holds \(\mm\)-a.e.\ on \(\{f=g\}\) for every \(f,g\in H^{1,p}(\mathbb X)\).
\item \textsc{Chain rule.} If \(f\in H^{1,p}(\mathbb X)\) and \(\phi\in\Lip_b(\R)\), then \(\phi\circ f\in H^{1,p}(\mathbb X)\) and
\[
|D(\phi\circ f)|_H\leq|\phi'|\circ f\,|Df|_H\quad\text{ holds }\mm\text{-a.e.\ on }X.
\]
\item \textsc{Leibniz rule.} If \(f,g\in H^{1,p}(\mathbb X)\cap L^\infty(\mm)\) are given, then \(fg\in H^{1,p}(\mathbb X)\) and
\[
|D(fg)|_H\leq|f||Dg|_H+|g||Df|_H\quad\text{ holds }\mm\text{-a.e.\ on }X.
\]
\end{itemize}

Minimal \(p\)-relaxed slopes are induced by a linear \emph{differential} operator \(\d\colon H^{1,p}(\mathbb X)\to L^p(T^*\mathbb X)\), where \(L^p(T^*\mathbb X)\)
is a distinguished \(L^p(\mm)\)-Banach \(L^\infty(\mm)\)-module, called the \emph{\(p\)-cotangent module}:
\begin{theorem}[Cotangent module]\label{thm:cotg_mod}
Let \(\mathbb X=(X,\tau,\sfd,\mm)\) be an e.m.t.m.\ space and \(p\in(1,\infty)\). Then there exist an \(L^p(\mm)\)-Banach \(L^\infty(\mm)\)-module
\(L^p(T^*\mathbb X)\) (called the \textbf{\(p\)-cotangent module}) and a linear operator \(\d\colon H^{1,p}(\mathbb X)\to L^p(T^*\mathbb X)\) (called the \textbf{differential}) such that:
\begin{itemize}
\item[\(\rm i)\)] \(|\d f|=|Df|_H\) for every \(f\in H^{1,p}(\mathbb X)\).
\item[\(\rm ii)\)] The \(L^\infty(\mm)\)-linear span of \(\{\d f:f\in H^{1,p}(\mathbb X)\}\) is dense in \(L^p(T^*\mathbb X)\).
\end{itemize}
The pair \((L^p(T^*\mathbb X),\d)\) is unique up to a unique isomorphism: for any \((\mathscr M,\tilde\d)\) having the same properties,
there exists a unique isomorphism of \(L^p(\mm)\)-Banach \(L^\infty(\mm)\)-modules \(\Phi\colon L^p(T^*\mathbb X)\to\mathscr M\) such that
\[\begin{tikzcd}
H^{1,p}(\mathbb X) \arrow[r,"\d"] \arrow[rd,swap,"\tilde\d"] & L^p(T^*\mathbb X) \arrow[d,"\Phi"] \\
& \mathscr M
\end{tikzcd}\]
is a commutative diagram. Moreover, the differential \(\d\) satisfies the following \textbf{Leibniz rule}:
\begin{equation}\label{eq:Leibniz_rule_diff}
\d(fg)=f\cdot\d g+g\cdot\d f\quad\text{ for every }f,g\in H^{1,p}(\mathbb X)\cap L^\infty(\mm).
\end{equation}
\end{theorem}
\begin{proof}
This construction is due to Gigli \cite{Gigli14}. The existence and uniqueness of \((L^p(T^*\mathbb X),\d)\) can be proved by repeating verbatim
the proof of \cite[Section 2.2.1]{Gigli14} or \cite[Theorem/Definition 2.8]{Gigli17} (see also \cite[Theorem 4.1.1]{GP20}, or \cite[Theorem 3.2]{Gig:Pas:19}
for the case \(p\neq 2\)). Alternatively, one can apply \cite[Theorem 3.19]{Luc:Pas:23}. The Leibniz rule \eqref{eq:Leibniz_rule_diff} can be proved by arguing
as in \cite[Corollary 2.2.8]{Gigli14} (or as in \cite[Proposition 2.12]{Gigli17}, or as in \cite[Theorem 4.1.4]{GP20}, or as in \cite[Proposition 3.5]{Gig:Pas:19}).
\end{proof}

Following \cite[Definition 2.3.1]{Gigli14}, we then introduce the \emph{\(q\)-tangent module} of \(\mathbb X\) by duality:
\begin{definition}[Tangent module]\label{def:tg_mod}
Let \(\mathbb X=(X,\tau,\sfd,\mm)\) be an e.m.t.m.\ space. Let \(p,q\in(1,\infty)\) be conjugate exponents. Then we define the
\textbf{\(q\)-tangent module} \(L^q(T\mathbb X)\) of \(\mathbb X\) as
\[
L^q(T\mathbb X)\coloneqq L^p(T^*\mathbb X)^*.
\]
\end{definition}
Recall that \(L^q(T\mathbb X)\), when regarded as a Banach space, can be identified with the dual Banach space \(L^p(T^*\mathbb X)'\) through the isomorphism
\begin{equation}\label{eq:def_I_pX}
\textsc{I}_{p,\mathbb X}\coloneqq\textsc{Int}_{L^p(T^*\mathbb X)}\colon L^q(T\mathbb X)\to L^p(T^*\mathbb X)'
\end{equation}
defined in \eqref{eq:def_Int}. The following result can be proved by suitably adapting
\cite[Proposition 1.4.8]{Gigli14} (or by applying \cite[Proposition 3.20]{Luc:Pas:23}):
\begin{proposition}\label{prop:univ_prop_cotg_mod}
Let \(\mathbb X=(X,\tau,\sfd,\mm)\) be an e.m.t.m.\ space. Let \(p,q\in(1,\infty)\) be conjugate exponents.
Assume that \(\varphi\colon H^{1,p}(\mathbb X)\to L^1(\mm)\) is a linear map with the following property: there exists
a function \(G\in L^q(\mm)^+\) such that \(|\varphi(f)|\leq G|Df|_H\) holds for every \(f\in H^{1,p}(\mathbb X)\).
Then there exists a unique vector field \(v_\varphi\in L^q(T\mathbb X)\) such that
\[\begin{tikzcd}
H^{1,p}(\mathbb X) \arrow[d,swap,"\d"] \arrow[r,"\varphi"] & L^1(\mm) \\
L^p(T^*\mathbb X) \arrow[ur,swap,"v_\varphi"] &
\end{tikzcd}\]
is a commutative diagram. Moreover, it holds that \(|v_\varphi|\leq G\).
\end{proposition}

Exactly as in \cite[Section 2.3.1]{Gigli14}, the tangent module \(L^q(T\mathbb X)\) can be equivalently characterised
in terms of a suitable notion of derivation, which we call `Sobolev derivation' (in order to make a distinction with
the notion of `Lipschitz derivation', which we will introduce in Section \ref{s:Lipschitz_derivations}). Namely:
\begin{definition}[Sobolev derivation]
Let \(\mathbb X=(X,\tau,\sfd,\mm)\) be an e.m.t.m.\ space and \(q\in(1,\infty)\). Then by a \textbf{Sobolev derivation}
(of exponent \(q\)) on \(\mathbb X\) we mean a linear map \(\delta\colon H^{1,p}(\mathbb X)\to L^1(\mm)\) such that
the following conditions hold:
\begin{itemize}
\item[\(\rm i)\)] \(\delta(fg)=f\,\delta(g)+g\,\delta(f)\) for every \(f,g\in H^{1,p}(\mathbb X)\cap L^\infty(\mm)\).
\item[\(\rm ii)\)] There exists a function \(G\in L^q(\mm)^+\) such that \(|\delta(f)|\leq G|Df|_H\) for every \(f\in H^{1,p}(\mathbb X)\).
\end{itemize}
We denote by \(L^q_{\rm Sob}(T\mathbb X)\) the set of all Sobolev derivations of exponent \(q\) on \(\mathbb X\).
\end{definition}

The above definition is adapted from \cite[Definition 2.3.2]{Gigli14}. To any derivation
\(\delta\in L^q_{\rm Sob}(T\mathbb X)\), we associate the function \(|\delta|\in L^q(\mm)^+\) given by
\[
|\delta|\coloneqq\bigwedge\Big\{G\in L^q(\mm)^+\;\Big|\;|\delta(f)|\leq G|Df|_H\text{ for every }f\in H^{1,p}(\mathbb X)\Big\}.
\]
Note that \(|\delta(f)|\leq|\delta||Df|_H\) for all \(f\in H^{1,p}(\mathbb X)\). It is straightforward to check that
\((L^q_{\rm Sob}(T\mathbb X),|\cdot|)\) is an \(L^q(\mm)\)-Banach \(L^\infty(\mm)\)-module. The latter can be identified with the
tangent module \(L^q(T\mathbb X)\), as the next result (which is essentially taken from \cite[Theorem 2.3.3]{Gigli14}) shows:
\begin{proposition}[Identification between \(L^q(T\mathbb X)\) and \(L^q_{\rm Sob}(T\mathbb X)\)]
Let \(\mathbb X=(X,\tau,\sfd,\mm)\) be an e.m.t.m.\ space and \(q\in(1,\infty)\). Then for any \(v\in L^q(T\mathbb X)\)
we have that \(v\circ\d\colon H^{1,p}(\mathbb X)\to L^1(\mm)\) is an element of \(L^q_{\rm Sob}(T\mathbb X)\). Moreover,
the resulting map \(\Phi\colon L^q(T\mathbb X)\to L^q_{\rm Sob}(T\mathbb X)\) is an isomorphism of \(L^q(\mm)\)-Banach
\(L^\infty(\mm)\)-modules.
\end{proposition}
\begin{proof}
Let \(v\in L^q(T\mathbb X)\) be a given vector field. Then \(v\circ\d\colon H^{1,p}(\mathbb X)\to L^1(\mm)\) is linear and
\[
(v\circ\d)(fg)=\d(fg)(v)=f\,\d g(v)+g\,\d f(v)=f\,(v\circ\d)(g)+g\,(v\circ\d)(f)
\]
for every \(f,g\in H^{1,p}(\mathbb X)\cap L^\infty(\mm)\) by \eqref{eq:Leibniz_rule_diff}. Moreover,
\(|(v\circ\d)(f)|=|\d f(v)|\leq|Df|_H|v|\) for every \(f\in H^{1,p}(\mathbb X)\). This gives
\(v\circ\d\in L^q_{\rm Sob}(T\mathbb X)\) and \(|v\circ\d|\leq|v|\). It follows that
\(\Phi\colon L^q(T\mathbb X)\to L^q_{\rm Sob}(T\mathbb X)\) is a linear map such that \(|\Phi(v)|\leq|v|\)
for every \(v\in L^q(T\mathbb X)\). Since we have that
\[
\Phi(h\cdot v)(f)=((h\cdot v)\circ\d)(f)=\d f(h\cdot v)=h\,\d f(v)=h\,\Phi(v)(f)=(h\cdot\Phi(v))(f)
\]
for every \(h\in L^\infty(\mm)\) and \(f\in H^{1,p}(\mathbb X)\), we deduce that \(\Phi\) is \(L^\infty(\mm)\)-linear.
To conclude, it remains to check that for any \(\delta\in L^q_{\rm Sob}(T\mathbb X)\) there exists \(v_\delta\in L^q(T\mathbb X)\)
such that \(\Phi(v_\delta)=\delta\) and \(|v_\delta|\leq|\delta|\). Since \(\delta\colon H^{1,p}(\mathbb X)\to L^1(\mm)\)
is linear and \(|\delta(f)|\leq|\delta||Df|_H\) for every \(f\in H^{1,p}(\mathbb X)\), we deduce from Proposition
\ref{prop:univ_prop_cotg_mod} that there exists (a unique) \(v_\delta\in L^q(T\mathbb X)\) such that
\(\delta=v_\delta\circ\d=\Phi(v_\delta)\), and it holds that \(|v_\delta|\leq|\delta|\). All in all, the statement is achieved.
\end{proof}
\subsection{Sobolev spaces \texorpdfstring{\(B^{1,p}\)}{B1p} via test plans}\label{s:def_B}
The second notion of Sobolev space over an e.m.t.m.\ space we consider is the one obtained by investigating the behaviour of functions along
suitably chosen curves. The relevant object here is that of a \emph{\(\mathcal T_q\)-test plan} (see Definition \ref{def:test_plan} below),
which was introduced in \cite[Section 4.2]{Savare22} after \cite{AmbrosioGigliSavare11,AmbrosioGigliSavare11-3,Amb:Mar:Sav:15}. A function
\(f\in L^p(\mm)\) is declared to be in the Sobolev space \(B^{1,p}(\mathbb X)\) if it has a \(p\)-integrable \emph{\(\mathcal T_q\)-weak upper gradient}
(where \(p\), \(q\) are conjugate exponents), i.e.\ a function satisfying the \emph{upper gradient} inequality \cite{Koskela-MacManus98,Hei:Kos:98,Cheeger00}
along \(\ppi\)-a.e.\ curve, for every \(\mathcal T_q\)-test plan \(\ppi\). Our notation `\(B^{1,p}\)' is different from the one of \cite{Savare22},
where `\(W^{1,p}\)' is used instead. The reason is that in this paper we prefer to denote by \(W^{1,p}(\mathbb X)\) the Sobolev space that we will define
through an integration-by-parts formula in Section \ref{s:def_W1p}, which comes with a notion of `weak derivative'.
In analogy with \cite{AILP24}, the notation \(B^{1,p}(\mathbb X)\) is chosen to remind the resemblance to Beppo Levi's approach
to weakly differentiable functions.
\medskip

Let \(\mathbb X=(X,\tau,\sfd,\mm)\) be an e.m.t.m.\ space. According to \cite[Definition 4.2.1]{Savare22}, a \textbf{dynamic plan}
on \(\mathbb X\) is a Radon measure \(\ppi\in\mathcal M_+({\rm RA}(X,\sfd),\tau_{\rm A})\) satisfying
\[
\int\ell(\gamma)\,\d\ppi(\gamma)<+\infty.
\]
The \textbf{barycenter} of \(\ppi\) is defined as the unique Radon measure \(\mu_\sppi\in\mathcal M_+(X,\tau)\) such that
\[
\int f\,\d\mu_\sppi=\int\bigg(\int_\gamma f\bigg)\,\d\ppi(\gamma)\quad\text{ for every bounded Borel function }f\colon(X,\tau)\to\R.
\]
Moreover, we say that \(\ppi\) has \textbf{\(q\)-barycenter}, for some \(q\in(1,\infty)\), if it holds that \(\mu_\sppi\ll\mm\) and
\[
h_\sppi\coloneqq\frac{\d\mu_\sppi}{\d\mm}\in L^q(\mm)^+.
\]
The following definition is taken from \cite[Definition 5.1.1]{Savare22}:
\begin{definition}[\(\mathcal T_q\)-test plan]\label{def:test_plan}
Let \(\mathbb X=(X,\tau,\sfd,\mm)\) be an e.m.t.m.\ space and \(q\in(1,\infty)\). Then a dynamic plan \(\ppi\) on \(\mathbb X\) is said
to be a \textbf{\(\mathcal T_q\)-test plan} provided it has \(q\)-barycenter and it holds that
\[
(\hat{\sf e}_0)_\#\ppi,(\hat{\sf e}_1)_\#\ppi\ll\mm,\qquad\frac{\d(\hat{\sf e}_0)_\#\ppi}{\d\mm},\frac{\d(\hat{\sf e}_1)_\#\ppi}{\d\mm}\in L^q(\mm)^+.
\]
We denote by \(\mathcal T_q(\mathbb X)\) the set of all \(\mathcal T_q\)-test plans on \(\mathbb X\).
\end{definition}

The corresponding notion of weak upper gradient is the following (from \cite[Definition 5.1.4]{Savare22}):
\begin{definition}[\(\mathcal T_q\)-weak upper gradient]
Let \(\mathbb X=(X,\tau,\sfd,\mm)\) be an e.m.t.m.\ space and \(q\in(1,\infty)\). Let \(f\colon X\to\R\) and
\(G\colon X\to[0,+\infty)\) be given \(\tau\)-Borel functions. Then we say that \(G\) is a \textbf{\(\mathcal T_q\)-weak upper
gradient} of \(f\) provided for any \(\ppi\in\mathcal T_q(\mathbb X)\) it holds that
\begin{equation}\label{eq:def_Tq_wug}
|f(\gamma_1)-f(\gamma_0)|\leq\int_\gamma G<+\infty\quad\text{ for }\ppi\text{-a.e.\ }\gamma\in{\rm RA}(X,\sfd).
\end{equation}
\end{definition}

If \(f,\tilde f\colon X\to\R\) are \(\tau\)-Borel functions satisfying \(f=\tilde f\) in the \(\mm\)-a.e.\ sense,
then \(f\) and \(\tilde f\) have the same \(\mathcal T_q\)-weak upper gradients. Hence, we can unambiguously
say that a function \(f\in L^1(\mm)\) has a \(\mathcal T_q\)-weak upper gradient. 
\begin{lemma}\label{lem:integral_def_B1p}
Let \(\mathbb X=(X,\tau,\sfd,\mm)\) be an e.m.t.m.\ space. Let \(p,q\in(1,\infty)\) be conjugate exponents. Let \(f\colon X\to\R\) and \(G\colon X\to[0,+\infty)\)
be given \(\tau\)-Borel functions with \(\int G^p\,\d\mm<+\infty\). Then the function \(G\) is a \(\mathcal T_q\)-weak upper gradient of \(f\) if and only if
\begin{equation}\label{eq:int_Tq_wug}
\int f(\gamma_1)-f(\gamma_0)\,\d\ppi(\gamma)\leq\int G\,h_\sppi\,\d\mm\quad\text{ for every }\ppi\in\mathcal T_q(\mathbb X).
\end{equation}
\end{lemma}
\begin{proof}
Necessity can be shown by integrating \eqref{eq:def_Tq_wug}. For sufficiency, we argue by contradiction: suppose that \eqref{eq:int_Tq_wug} holds,
but \(G\) is not a \(\mathcal T_q\)-weak upper gradient of \(f\). Then there exist a \(\mathcal T_q\)-test plan \(\ppi\in\mathcal T_q(\mathbb X)\),
a Borel set \(\Gamma\subseteq{\rm RA}(X,\sfd)\) with \(\ppi(\Gamma)>0\) and some \(\varepsilon>0\) such that
\begin{equation}\label{eq:int_Tq_wug_aux}
|f(\gamma_1)-f(\gamma_0)|\geq\varepsilon+\int_\gamma G\quad\text{ for every }\gamma\in\Gamma.
\end{equation}
Denote \(\Gamma_+\coloneqq\{\gamma\in\Gamma\,:\,f(\gamma_1)\geq f(\gamma_0)\}\) and \(\Gamma_-\coloneqq\Gamma\setminus\Gamma_+\).
Now let us consider \(\ppi_+\coloneqq\ppi|_{\Gamma_+}\in\mathcal T_q(\mathbb X)\) and \(\ppi_-\coloneqq{\rm Rev}_\#(\ppi|_{\Gamma_-})\in\mathcal T_q(\mathbb X)\),
where \({\rm Rev}\colon{\rm RA}(X,\sfd)\to{\rm RA}(X,\sfd)\) denotes the map sending a rectifiable arc \([\gamma]\) to the
\(\sim\)-equivalence class of the curve \([0,1]\ni t\mapsto\gamma_{1-t}\in X\). We deduce that
\[\begin{split}
\varepsilon\ppi(\Gamma_\pm)+\int G\,h_{\sppi_\pm}\,\d\mm&=\int\bigg(\varepsilon+\int_\gamma G\bigg)\,\d\ppi_\pm(\gamma)
\overset{\eqref{eq:int_Tq_wug_aux}}\leq\int f(\gamma_1)-f(\gamma_0)\,\d\ppi_\pm(\gamma)
\overset{\eqref{eq:int_Tq_wug}}\leq\int G\,h_{\sppi_\pm}\,\d\mm.
\end{split}\]
Either \(\ppi(\Gamma_+)>0\) or \(\ppi(\Gamma_-)>0\), thus the above estimates lead to a contradiction.
\end{proof}

If \(f\in L^1(\mm)\) has a \(\mathcal T_q\)-weak upper gradient in \(L^p(\mm)\) (where \(p\in(1,\infty)\) denotes the conjugate
exponent of \(q\)), then there exists a unique function \(|Df|_B\in L^p(\mm)^+\), which we call the
\textbf{minimal \(\mathcal T_q\)-weak upper gradient} of \(f\), such that the following hold:
\begin{itemize}
\item[\(\rm i)\)] \(|Df|_B\) has a representative \(G_f\colon X\to[0,+\infty)\) that is
a \(\mathcal T_q\)-weak upper gradient of \(f\).
\item[\(\rm ii)\)] If \(G\) is a \(\mathcal T_q\)-weak upper gradient of \(f\), then \(|Df|_B\leq G\)
holds \(\mm\)-a.e.\ in \(X\).
\end{itemize}
See \cite[paragraph after Definition 5.1.23]{Savare22}. Consequently, the following definition (which is taken
from \cite[Definition 5.1.24]{Savare22}) is well posed:
\begin{definition}[The Sobolev space \(B^{1,p}(\mathbb X)\)]
Let \(\mathbb X=(X,\tau,\sfd,\mm)\) be an e.m.t.m.\ space. Let \(p,q\in(1,\infty)\) be conjugate exponents.
Then we define the \textbf{Sobolev space} \(B^{1,p}(\mathbb X)\) as the set of all functions \(f\in L^p(\mm)\)
having a \(\mathcal T_q\)-weak upper gradient in \(L^p(\mm)\). Moreover, we define
\[
\|f\|_{B^{1,p}(\mathbb X)}\coloneqq\big(\|f\|_{L^p(\mm)}^p+\||Df|_B\|_{L^p(\mm)}^p\big)^{1/p}\quad\text{ for every }f\in B^{1,p}(\mathbb X).
\]
\end{definition}

It holds that \((B^{1,p}(\mathbb X),\|\cdot\|_{B^{1,p}(\mathbb X)})\) is a Banach space.
In the setting of \(\sfd\)-complete e.m.t.m.\ spaces, the full equivalence of \(H^{1,p}\) and \(W^{1,p}\)
was obtained by Savar\'{e} in \cite[Theorem 5.2.7]{Savare22} (see Theorem \ref{thm:H=B} below for the
precise statement), thus generalising previous results for metric measure spaces
\cite{Cheeger00,Shanmugalingam00,AmbrosioGigliSavare11,AmbrosioGigliSavare11-3}.
See also \cite{EB:20,LP24,AILP24} for other related equivalence results.
\begin{theorem}[\(H^{1,p}=B^{1,p}\) on complete e.m.t.m.\ spaces]\label{thm:H=B}
Let \(\mathbb X=(X,\tau,\sfd,\mm)\) be an e.m.t.m.\ space such that \((X,\sfd)\) is a complete extended metric space.
Let \(p\in(1,\infty)\) be given. Then
\[
H^{1,p}(\mathbb X)=B^{1,p}(\mathbb X).
\]
Moreover, it holds that \(|Df|_B=|Df|_H\) for every \(f\in H^{1,p}(\mathbb X)\).
\end{theorem}

The completeness assumption in Theorem \ref{thm:H=B} cannot be dropped. For instance, let us consider
the space \((-1,1)\setminus\{0\}\) equipped with the restriction of the Euclidean distance, its induced
topology and the restriction of the one-dimensional Lebesgue measure. It can be readily checked that the
function \(\1_{(0,1)}\) is \(B^{1,p}\)-Sobolev with null minimal \(\mathcal T_q\)-weak upper gradient,
but not \(H^{1,p}\)-Sobolev.
\section{Extensions of \texorpdfstring{\(\tau\)}{tau}-continuous \texorpdfstring{\(\sfd\)}{d}-Lipschitz functions}\label{s:extLip}
A fundamental tool in metric geometry is the \emph{McShane--Whitney extension theorem}, which states that every real-valued Lipschitz function
defined on some subset of a metric space can be extended to a Lipschitz function on the whole metric space, also preserving the Lipschitz constant.
In the setting of extended metric-topological spaces, we rather need an extension theorem for \(\tau\)-continuous \(\sfd\)-Lipschitz
functions for which both the \(\tau\)-continuity and the \(\sfd\)-Lipschitz conditions are preserved. The extension results obtained by Matou\v{s}kov\'{a} in
\cite{Matouskova00} are fit for our purposes:
\begin{theorem}[Extension result]\label{thm:McShane_emts}
Let \((X,\tau,\sfd)\) be an e.m.t.\ space with \((X,\tau)\) normal. Assume
\begin{equation}\label{eq:hp_extension}
\bar B^\sfd_r(C)\;\text{ is }\tau\text{-closed, for every }\tau\text{-closed set }C\subseteq X\text{ and }r\in(0,+\infty).
\end{equation}
 Let \(C\subseteq X\) be a \(\tau\)-closed set. Let \(f\colon C\to\R\) be a bounded \(\tau\)-continuous \(\sfd\)-Lipschitz function.
 Then there exists a function \(\bar f\in\Lip_b(X,\tau,\sfd)\) such that
\[
\bar f|_C=f,\qquad\Lip(\bar f,\sfd)=\Lip(f,C,\sfd),\qquad\inf_C f\leq\bar f\leq\sup_C f.
\]
\end{theorem}
\begin{proof}
Without loss of generality, we can assume that \(\Lip(f,C,\sfd)>0\). Let us define
\[
M\coloneqq\frac{{\rm Osc}_C(f)}{\Lip(f,C,\sfd)}>0
\]
and let us consider the truncated distance \(\tilde\sfd\coloneqq\sfd\wedge M\). Then \(\tilde\sfd\) is \((\tau\times\tau)\)-lower semicontinuous,
\(f\) is \(\tilde\sfd\)-Lipschitz and \(\Lip(f,C,\tilde\sfd)=\Lip(f,C,\sfd)\). By virtue of \cite[Theorem 2.4]{Matouskova00}, we can find
a \(\tau\)-continuous \(\tilde\sfd\)-Lipschitz extension \(\bar f\colon X\to\R\) of \(f\) such that \(\Lip(\bar f,\tilde\sfd)=\Lip(f,C,\tilde\sfd)\)
and \(\inf_C f\leq\bar f\leq\sup_C f\). Given that \(\tilde\sfd\leq\sfd\), we can thus conclude that \(\bar f\in\Lip_b(X,\tau,\sfd)\) and
\(\Lip(\bar f,\sfd)=\Lip(f,C,\sfd)\).
\end{proof}
\begin{remark}\label{rmk:ext_on_cpt_emt}{\rm
Let us make some comments on Theorem \ref{thm:McShane_emts}:
\begin{itemize}
\item[\(\rm i)\)] Every \(\tau\)-compact e.m.t.\ space \((X,\tau,\sfd)\) satisfies the assumptions of Theorem \ref{thm:McShane_emts}.
Indeed, all compact Hausdorff spaces are normal, and \eqref{eq:hp_extension} holds by \cite[proof of Corollary 2.5]{Matouskova00}.
\item[\(\rm ii)\)] If \(\B\) is a Banach space, \(\sfd_{\B'}\) denotes the distance on \(\B'\) induced by its norm and \(\tau_{w^*}\) is
the weak\(^*\) topology of \(\B'\), then \((\B',\tau_{w^*},\sfd_{\B'})\) fulfils the assumptions of Theorem \ref{thm:McShane_emts}, as it
is shown in the proof of \cite[Corollary 2.6]{Matouskova00}.
\item[\(\rm iii)\)] The requirement \eqref{eq:hp_extension} cannot be dropped. Indeed, if \(\B\) is a non-reflexive Banach space,
\(\sfd_\B\) denotes its induced distance and \(\tau_w\) is its weak topology, then \((\B,\tau_w,\sfd_\B)\) neither fulfils \eqref{eq:hp_extension}
nor the conclusions of Theorem \ref{thm:McShane_emts}; see the proof of \cite[Theorem 3.1]{Matouskova00}.
Note also that if in addition \(\B'\) is separable, then \((\B,\tau_w)\) is normal (it can be readily checked that it is both
regular and Lindel\"{o}f, thus it is normal by \cite[Lemma at page 113]{Kelley75}).
\item[\(\rm iv)\)] If \((X,\tau)\) is a normal Hausdorff space and \(\sfd\coloneqq\sfd_{\rm discr}\) denotes the discrete distance on \(X\), then
Theorem \ref{thm:McShane_emts} for \((X,\tau,\sfd)\) reduces to the \emph{Tietze extension theorem} for bounded functions (note
that \eqref{eq:hp_extension} holds in this case, since \(\bar B_r^\sfd(C)=C\) if \(r<1\), \(\bar B_r^\sfd(C)=X\) otherwise).
In particular, in Theorem \ref{thm:McShane_emts} both the assumptions that \(\tau\) is normal and that the set \(C\) is \(\tau\)-closed are needed.
\item[\(\rm v)\)] If \((X,\sfd)\) is a metric space and \(\tau_\sfd\) denotes the topology induced by \(\sfd\), then Theorem
\ref{thm:McShane_emts} for \((X,\tau_\sfd,\sfd)\) implies the McShane--Whitney extension theorem for bounded functions.
\item[\(\rm vi)\)] Differently from the Tietze and the McShane--Whitney extension theorems, in Theorem \ref{thm:McShane_emts} the
boundedness assumption on \(f\) cannot be dropped; see e.g.\ \cite[Example 3.2]{Matouskova00}.
\fr
\end{itemize}
}\end{remark}

In Section \ref{s:Lipschitz_derivations}, the above extension result will be used to study the relation between
different notions of Lipschitz derivations. Rather than Theorem \ref{thm:McShane_emts}, we will apply a consequence of it:
\begin{corollary}\label{cor:Lip_ext_cpt}
Let \((X,\tau,\sfd)\) be an e.m.t.\ space. Let \(K\subseteq X\) be a \(\tau\)-compact set. Let \(f\colon K\to\R\) be a bounded
\(\tau\)-continuous \(\sfd\)-Lipschitz function. Then there exists \(\bar f\in\Lip_b(X,\tau,\sfd)\) such that
\[
\bar f|_K=f,\qquad\Lip(\bar f,\sfd)=\Lip(f,K,\sfd),\qquad\min_K f\leq\bar f\leq\max_K f.
\]
\end{corollary}
\begin{proof}
Consider the compactification \((\hat X,\hat\tau,\hat\sfd)\) of \((X,\tau,\sfd)\) and the canonical embedding \(\iota\colon X\hookrightarrow\hat X\).
Since \(\iota\) is continuous, we have that \(\iota(K)\) is \(\hat\tau\)-compact. The function \(g\colon\iota(K)\to\R\), which we
define as \(g(y)\coloneqq f(\iota^{-1}(y))\) for every \(y\in\iota(K)\), is \(\hat\tau\)-continuous and \(\hat\sfd\)-Lipschitz.
By applying Theorem \ref{thm:McShane_emts} (taking also Remark \ref{rmk:ext_on_cpt_emt} i) into account), we deduce that there exists
a function \(\bar g\colon\hat X\to\R\) such that \(\bar g|_{\iota(K)}=g\), \(\Lip(\bar g,\hat\sfd)=\Lip(g,\iota(K),\hat\sfd)\) and
\(\min_{\iota(K)}g\leq\bar g\leq\max_{\iota(K)}g\). Now define \(\bar f\colon X\to\R\) as \(\bar f(x)\coloneqq\bar g(\iota(x))\) for every \(x\in X\).
Observe that \(\bar f\in\Lip_b(X,\tau,\sfd)\), \(\bar f|_K=f\), \(\Lip(\bar f,\sfd)=\Lip(f,K,\sfd)\) and \(\min_K f\leq\bar f\leq\max_K f\).
Therefore, the statement is proved.
\end{proof}
\section{Lipschitz derivations}\label{s:Lipschitz_derivations}
Let us begin by introducing a rather general notion of \emph{Lipschitz derivation} over an arbitrary e.m.t.m.\ space. In Sections \ref{s:Weaver} and \ref{s:Di_Marino},
we will then identify and study two special classes of derivations, which extend previous notions by Weaver \cite{Weaver01,Weaver2018} and Di Marino \cite{DiMar:14,DiMarPhD},
respectively.
\begin{definition}[Lipschitz derivation]\label{def:der}
Let \(\mathbb X=(X,\tau,\sfd,\mm)\) be an e.m.t.m.\ space. Then by a \textbf{Lipschitz derivation} on \(\mathbb X\)
we mean a linear operator \(b\colon\Lip_b(X,\tau,\sfd)\to L^0(\mm)\) such that
\begin{equation}\label{eq:Leibniz_Lip_der}
b(fg)=f\,b(g)+g\,b(f)\quad\text{ for every }f,g\in\Lip_b(X,\tau,\sfd).
\end{equation}
We refer to \eqref{eq:Leibniz_Lip_der} as the \textbf{Leibniz rule}. We denote by \({\rm Der}(\mathbb X)\) the set of all derivations on \(\mathbb X\).
\end{definition}

It can be readily checked that the space \({\rm Der}(\mathbb X)\) is a module over \(L^0(\mm)\) if endowed with
\[\begin{split}
(b+\tilde b)(f)\coloneqq b(f)+\tilde b(f)&\quad\text{ for every }b,\tilde b\in{\rm Der}(\mathbb X)\text{ and }f\in\Lip_b(X,\tau,\sfd),\\
(hb)(f)\coloneqq h\,b(f)&\quad\text{ for every }b\in{\rm Der}(\mathbb X)\text{, }h\in L^0(\mm)\text{ and }f\in\Lip_b(X,\tau,\sfd).
\end{split}\]
In particular, \({\rm Der}(\mathbb X)\) is a vector space (since the field \(\R\) can be identified with a subring of
\(L^0(\mm)\), via the map that associates to every number \(\lambda\in\R\) the function that is \(\mm\)-a.e.\ equal to \(\lambda\)).
\begin{definition}[Divergence of a Lipschitz derivation]\label{def:div_Lip_der}
Let \(\mathbb X=(X,\tau,\sfd,\mm)\) be an e.m.t.m.\ space and \(b\in{\rm Der}(\mathbb X)\). Then we say that \(b\)
\textbf{has divergence} provided it holds that \(b(f)\in L^1(\mm)\) for every \(f\in\Lip_b(X,\tau,\sfd)\) and there
exists a function \(\div(b)\in L^1(\mm)\) such that
\begin{equation}\label{eq:def_div}
\int b(f)\,\d\mm=-\int f\,\div(b)\,\d\mm\quad\text{ for every }f\in\Lip_b(X,\tau,\sfd).
\end{equation}
We denote by \(D(\div;\mathbb X)\) the set of all Lipschitz derivations on \(\mathbb X\) having divergence.
\end{definition}

Let us make some comments on Definition \ref{def:div_Lip_der}:
\begin{itemize}
\item Since \(\Lip_b(X,\tau,\sfd)\) is weakly\(^*\) dense in \(L^1(\mm)\) (as it easily follows from \eqref{eq:Lip_dense_Lp}),
it holds that the divergence \(\div(b)\) is uniquely determined by \eqref{eq:def_div}.
\item \(D(\div;\mathbb X)\) is a vector subspace of \(\Der(\mathbb X)\).
\item \(\div\colon D(\div;\mathbb X)\to L^1(\mm)\) is a linear operator.
\item The divergence satisfies the \textbf{Leibniz rule}, i.e.\ for every \(b\in D(\div;\mathbb X)\) and
\(h\in\Lip_b(X,\tau,\sfd)\) it holds that \(hb\in D(\div;\mathbb X)\) and
\[
\div(hb)=h\,\div(b)+b(h).
\]
In particular, \(D(\div;\mathbb X)\) is a \(\Lip_b(X,\tau,\sfd)\)-submodule of \(\Der(\mathbb X)\).
\end{itemize}

We shall focus on classes of derivations satisfying additional locality or continuity properties:
\begin{definition}[Local derivation]\label{def:loc_der}
Let \(\mathbb X=(X,\tau,\sfd,\mm)\) be an e.m.t.m.\ space. Let \(b\in\Der(\mathbb X)\) be a given derivation.
Then we say that \(b\) is \textbf{local} if for every function \(f\in\Lip_b(X,\tau,\sfd)\) we have that
\[
b(f)=0\quad\text{ holds }\mm\text{-a.e.\ on }\{f=0\}.
\]
\end{definition}

Let \(E\in\mathscr B(X,\tau)\) be such that \(\mm(E)>0\). Then every local derivation \(b\in{\rm Der}(\mathbb X)\)
induces by restriction a local derivation \(b\llcorner E\in{\rm Der}(\mathbb X\llcorner E)\), where \(\mathbb X\llcorner E\)
is as in \eqref{eq:restr_emtm}, in the following way. Thanks to the inner regularity of \(\mm\), we can find a sequence
\((K_n)_n\) of pairwise disjoint \(\tau\)-compact subsets of \(E\) such that \(\mm\big(E\setminus\bigcup_{n\in\N}K_n\big)=0\).
For any \(f\in\Lip_b(E,\tau_E,\sfd_E)\) and \(n\in\N\), we know from Corollary \ref{cor:Lip_ext_cpt} that there
exists \(\bar f_n\in\Lip_b(X,\tau,\sfd)\) such that \(\bar f_n|_{K_n}=f|_{K_n}\). We then define
\begin{equation}\label{eq:restr_loc_der}
(b\llcorner E)(f)\coloneqq\sum_{n\in\N}\1_{K_n}b(\bar f_n)\in L^0(\mm\llcorner E).
\end{equation}
By using the locality of \(b\), one can readily check that \(b\llcorner E\) is well defined and local.
\medskip

In the following definition, we endow the closed unit ball \(\bar B_{\Lip_b(X,\tau,\sfd)}\) of \(\Lip_b(X,\tau,\sfd)\)
with the topology \(\tau_{pt}\) \emph{of pointwise convergence}, and the space \(L^\infty(\mm)\) with its weak\(^*\)
topology \(\tau_{w^*}\).
\begin{definition}[Weak\(^*\)-type continuity of derivations]\label{def:w-type_cont}
Let \(\mathbb X=(X,\tau,\sfd,\mm)\) be an e.m.t.m.\ space. Let \(b\in{\rm Der}(\mathbb X)\) be a given derivation
satisfying \(b(f)\in L^\infty(\mm)\) for every \(f\in\Lip_b(X,\tau,\sfd)\). Then:
\begin{itemize}
\item[\(\rm i)\)] We say that \(b\) is \textbf{weakly\(^*\)-type continuous} provided the map
\(b|_{\bar B_{\Lip_b(X,\tau,\sfd)}}\) is continuous between \((\bar B_{\Lip_b(X,\tau,\sfd)},\tau_{pt})\) and \((L^\infty(\mm),\tau_{w^*})\).
\item[\(\rm ii)\)] We say that \(b\) is \textbf{weakly\(^*\)-type sequentially continuous} provided the map
\(b|_{\bar B_{\Lip_b(X,\tau,\sfd)}}\) is sequentially continuous between \((\bar B_{\Lip_b(X,\tau,\sfd)},\tau_{pt})\) and \((L^\infty(\mm),\tau_{w^*})\).
\end{itemize}
\end{definition}

Some comments on the weak\(^*\)-type continuity and the weak\(^*\)-type sequential continuity:
\begin{itemize}
\item Since derivations are linear, the weak\(^*\)-type continuity can be equivalently reformulated by asking that
\emph{if a bounded net \((f_i)_{i\in I}\subseteq\Lip_b(X,\tau,\sfd)\) and a function \(f\in\Lip_b(X,\tau,\sfd)\)
satisfy \(\lim_{i\in I}f_i(x)=f(x)\) for every \(x\in X\), then \(\lim_{i\in I}b(f_i)=b(f)\) with
respect to the weak\(^*\) topology of \(L^\infty(\mm)\).} Similarly, the weak\(^*\)-type sequential continuity is
equivalent to asking that \emph{if a bounded sequence \((f_n)_{n\in\N}\subseteq\Lip_b(X,\tau,\sfd)\) and a function
\(f\in\Lip_b(X,\tau,\sfd)\) satisfy \(\lim_n f_n(x)=f(x)\) for every \(x\in X\), then \(b(f_n)\overset{*}\rightharpoonup b(f)\)
weakly\(^*\) in \(L^\infty(\mm)\) as \(n\to\infty\).}
\item The terminology `weak\(^*\)-type (sequential) continuity' is motivated by the fact that it strongly resembles the weak\(^*\) (sequential) continuity
in the Banach algebra \(\Lip_b(X,\sfd)\) of bounded Lipschitz functions on a metric space (see \cite[Corollary 3.4]{Weaver2018}), even though in
the setting of e.m.t.\ spaces one has that \(\Lip_b(X,\tau,\sfd)\) does not always have a predual (see Proposition \ref{prop:ex_no_predual})
and thus we cannot talk about an actual weak\(^*\) topology on it.
\item We point out that if a bounded sequence \((f_n)_n\subseteq\Lip_b(X,\tau,\sfd)\) and a function \(f\colon X\to\R\) satisfy \(f_n(x)\to f(x)\)
for every \(x\in X\), then \(f\) is \(\sfd\)-Lipschitz, but it can happen that it is not \(\tau\)-continuous, and thus it does not belong to
\(\Lip_b(X,\tau,\sfd)\); see Example \ref{ex:tau_pt_not_cpt} below.
\item The weak\(^*\)-type continuity is stronger than the weak\(^*\)-type sequential continuity, but they are not equivalent concepts, as we will see
in Proposition \ref{prop:w*_cont_vs_w*_seq_cont} and Remark \ref{ex:der_Wiener}.
\end{itemize}
\begin{example}\label{ex:tau_pt_not_cpt}{\rm
When \((X,\sfd)\) is a metric space, the topology \(\tau_{pt}\) on \(\bar B_{\Lip_b(X,\sfd)}\) coincides with
the restriction of the weak\(^*\) topology of \(\Lip_b(X,\tau,\sfd)\), thus in particular \((\bar B_{\Lip_b(X,\sfd)},\tau_{pt})\)
is a compact Hausdorff topological space. On the contrary, in the more general setting of e.m.t.\ spaces the
Hausdorff topological space \((\bar B_{\Lip_b(X,\tau,\sfd)},\tau_{pt})\) needs not be compact. For example,
consider the unit interval \([0,1]\) together with the Euclidean topology \(\tau\) and the discrete distance
\(\sfd_{\rm discr}\), which gives a `purely-topological' e.m.t.\ space as in Example \ref{ex:Schultz's_space}.
Letting \((f_n)_{n\in\N}\subseteq\Lip_b([0,1],\tau,\sfd_{\rm discr})\) be defined as \(f_n(t)\coloneqq(nt)\wedge 1\)
for every \(n\in\N\) and \(t\in[0,1]\), we have that \(\|f_n\|_{\Lip_b([0,1],\tau,\sfd_{\rm discr})}=2\) for every \(n\in\N\)
and \(\1_{(0,1]}(t)=\lim_n f_n(t)\) for every \(t\in[0,1]\), but \(\1_{(0,1]}\notin\Lip_b([0,1],\tau,\sfd_{\rm discr})\)
(because it is not \(\tau\)-continuous at \(0\)). In particular, \((\bar B_{\Lip_b([0,1],\tau,\sfd_{\rm discr})},\tau_{pt})\)
is not compact.
\fr}\end{example}

The weak\(^*\)-type sequential continuity condition implies both locality and strong continuity:
\begin{theorem}\label{thm:Weaver_locality_and_bdd}
Let \(\mathbb X=(X,\tau,\sfd,\mm)\) be an e.m.t.m.\ space. Let \(b\in\Der(\mathbb X)\) be weakly\(^*\)-type sequentially
continuous. Then \(b\) is a local derivation. Moreover, the map \(b\colon\Lip_b(X,\tau,\sfd)\to L^\infty(\mm)\)
is a bounded linear operator.
\end{theorem}
\begin{proof}
The proof of locality is essentially taken from \cite[Lemma 10.34]{Weaver2018}. Fix any \(f\in\Lip_b(X,\tau,\sfd)\). For any \(n\in\N\), we define the auxiliary functions
\(\phi_n,\psi_n\colon\R\to\R\) as \(\phi_n(t)\coloneqq 1-e^{-nt^2}\) and \(\psi_n(t)\coloneqq t\,\phi_n(t)\) for every \(t\in\R\). Since \(0\leq\phi_n(t)\leq 1\) and
\(\phi'_n(t)=2nte^{-nt^2}\) for all \(t\in\R\), we have that \(\phi_n\) is Lipschitz on \(f(X)\) and thus \(\phi_n\circ f\in\Lip_b(X,\tau,\sfd)\). Moreover, \(-|t|\leq\psi_n(t)\leq|t|\)
and \(0\leq\psi'_n(t)\leq 1+2e^{-3/2}\) for all \(t\in\R\), so that \(\psi_n\circ f\in\Lip_b(X,\tau,\sfd)\) with \(\|\psi_n\circ f\|_{C_b(X,\tau)}\leq\|f\|_{C_b(X,\tau)}\) and
\(\Lip(\psi_n\circ f,\sfd)\leq(1+2e^{-3/2})\Lip(f,\sfd)\). In particular, the sequence \((\psi_n\circ f)_n\) is norm bounded in \(\Lip_b(X,\tau,\sfd)\). Note also that
\(\lim_n(\psi_n\circ f)(x)=f(x)\) for every \(x\in X\), whence it follows that
\[
f\,b(\phi_n\circ f)+(\phi_n\circ f)b(f)=b((\phi_n\circ f)f)=b(\psi_n\circ f)\overset{*}{\rightharpoonup}b(f)\quad\text{ weakly\(^*\) in }L^\infty(\mm)\text{ as }n\to\infty
\]
by the weak\(^*\)-type sequential continuity of \(b\). In particular, as \(\1_{\{f=0\}}(f\,b(\phi_n\circ f)+(\phi_n\circ f)b(f))=0\)
holds \(\mm\)-a.e.\ for every \(n\in\N\), we conclude that \(\1_{\{f=0\}}b(f)=0\) in the \(\mm\)-a.e.\ sense, thus \(b\) is local.

Let us now prove that \(b\colon\Lip_b(X,\tau,\sfd)\to L^\infty(\mm)\) is a bounded linear operator. Given any function
\(h\in L^1(\mm)\), we define the linear operator \(T_h\colon\Lip_b(X,\tau,\sfd)\to\R\) as
\[
T_h(f)\coloneqq\int h\,b(f)\,\d\mm\quad\text{ for every }f\in\Lip_b(X,\tau,\sfd).
\]
If \((f_n)_{n\in\N}\subseteq\Lip_b(X,\tau,\sfd)\) and \(f\in\Lip_b(X,\tau,\sfd)\) satisfy \(\|f_n-f\|_{\Lip_b(X,\tau,\sfd)}\to 0\)
as \(n\to\infty\), then we have in particular that \(\sup_{n\in\N}\|f_n\|_{\Lip_b(X,\tau,\sfd)}<+\infty\) and
\(f(x)=\lim_n f_n(x)\) for every \(x\in X\), so that \(b(f_n)\overset{*}{\rightharpoonup}b(f)\) weakly\(^*\) in \(L^\infty(\mm)\)
by the weak\(^*\)-type sequential continuity of \(b\), and thus accordingly \(T_h(f_n)=\int h\,b(f_n)\,\d\mm\to\int h\,b(f)\,\d\mm=T_h(f)\).
This shows that \(T_h\colon\Lip_b(X,\tau,\sfd)\to\R\) is continuous, thus \(T_h\in\Lip_b(X,\tau,\sfd)'\).
Next, denote \(B\coloneqq\{h\in L^1(\mm):\|h\|_{L^1(\mm)}\leq 1\}\). Note that
\[
\sup_{h\in B}|T_h(f)|\leq\sup_{h\in B}\int|h||b(f)|\,\d\mm\leq\|b(f)\|_{L^\infty(\mm)}\quad\text{ for every }f\in\Lip_b(X,\tau,\sfd)
\]
by H\"{o}lder's inequality. Thanks to the Uniform Boundedness Principle, we then deduce that
\[
M\coloneqq\sup_{h\in B}\|T_h\|_{\Lip_b(X,\tau,\sfd)'}<+\infty.
\]
Therefore, we can conclude that for any \(f\in\Lip_b(X,\tau,\sfd)\) with \(\|f\|_{\Lip_b(X,\tau,\sfd)}\leq 1\) it holds that
\[
\|b(f)\|_{L^\infty(\mm)}=\sup_{h\in B}\int h\,b(f)\,\d\mm=\sup_{h\in B}T_h(f)\leq\sup_{h\in B}\|T_h\|_{\Lip_b(X,\tau,\sfd)'}=M,
\]
whence it follows that \(b\colon\Lip_b(X,\tau,\sfd)\to L^\infty(\mm)\) is a bounded linear operator.
\end{proof}

The next result clarifies the interplay between weak\(^*\)-type continuous derivations and the decomposition
of an e.m.t.m.\ space into its maximal \(\sfd\)-separable and purely non-\(\sfd\)-separable components.
The proof of i) was suggested to us by Sylvester Eriksson-Bique.
\begin{proposition}\label{prop:w*_cont_vs_w*_seq_cont}
Let \(\mathbb X=(X,\tau,\sfd,\mm)\) be an e.m.t.m.\ space. Let \(b\in{\rm Der}(\mathbb X)\) be given. Then:
\begin{itemize}
\item[\(\rm i)\)] If \(b\) is weakly\(^*\)-type continuous, then \(b(f)=0\) \(\mm\)-a.e.\ on
\(X\setminus{\rm S}_{\mathbb X}\) for every \(f\in\Lip_b(X,\tau,\sfd)\).
\item[\(\rm ii)\)] If \(b\) is a local derivation and \(b\llcorner{\rm S}_{\mathbb X}\) is weakly\(^*\)-type sequentially
continuous, then \(b\llcorner{\rm S}_{\mathbb X}\) is weakly\(^*\)-type continuous. In particular, if
\(b\) is weakly\(^*\)-type sequentially continuous, then \(b\llcorner{\rm S}_{\mathbb X}\) is weakly\(^*\)-type continuous.
\end{itemize}
\end{proposition}
\begin{proof}
\ \\
{\bf i)}
Assume that \(b\) is weakly\(^*\)-type continuous. We argue by contradiction: suppose that there exists a function
\(f\in\Lip_b(X,\tau,\sfd)\) such that \(\mm(\{b(f)\neq 0\}\setminus{\rm S}_{\mathbb X})>0\). Up to replacing
\(f\) with \(-f\), we can assume that
\(\mm(\{b(f)>0\}\setminus{\rm S}_{\mathbb X})>0\), so that there exists a real number \(\lambda>0\) such that \(\mm(\{b(f)\geq\lambda\}\setminus{\rm S}_{\mathbb X})>0\).
Fix any \(\tau\)-Borel \(\mm\)-a.e.\ representative \(P\) of \(\{b(f)\geq\lambda\}\setminus{\rm S}_{\mathbb X}\)
satisfying \(P\subseteq X\setminus{\rm S}_{\mathbb X}\). Next, let us define
\[
\mathcal I\coloneqq\big\{(F,G)\;\big|\;F\subseteq X\text{ finite,}\,G\subseteq\Lip_{b,1}(X,\tau,\sfd)\text{ finite}\big\}.
\]
For any \((F,G),(\tilde F,\tilde G)\in\mathcal I\), we declare that \((F,G)\preceq(\tilde F,\tilde G)\) if and only if \(F\subseteq\tilde F\) and \(G\subseteq\tilde G\).
Note that \((\mathcal I,\preceq)\) is a directed set. We then define the net \((u_{F,G})_{(F,G)\in\mathcal I}\subseteq\Lip_b(X,\tau,\sfd)\) as
\[
u_{F,G}(x)\coloneqq\min_{p\in F}\max_{g\in G}|g(x)-g(p)|\wedge 1\quad\text{ for every }(F,G)\in\mathcal I\text{ and }x\in X.
\]
Given any \(x\in X\), we have that \(u_{F,G}(x)=0\) holds for every \((F,G)\in\mathcal I\) with
\((\{x\},\varnothing)\preceq(F,G)\), thus accordingly \(\lim_{(F,G)\in\mathcal I}u_{F,G}(x)=0\)
and \(\lim_{(F,G)\in\mathcal I}(u_{F,G}f)(x)=0\). Since \(\{u_{F,G}:(F,G)\in\mathcal I\}\) and
\(\{u_{F,G}f:(F,G)\in\mathcal I\}\) are bounded subsets of \(\Lip_b(X,\tau,\sfd)\), we deduce that
\[
\lim_{(F,G)\in\mathcal I}u_{F,G}\,b(f)=\lim_{(F,G)\in\mathcal I}\big(b(u_{F,G}f)-f\,b(u_{F,G})\big)=0\quad
\text{ weakly\(^*\) in }L^\infty(\mm),
\]
by the weak\(^*\)-type continuity of \(b\) and the Leibniz rule. Hence,
\(\lim_{(F,G)\in\mathcal I}\int_P u_{F,G}\,b(f)\,\d\mm=0\). Since \(0\leq\lambda\int_P u_{F,G}\,\d\mm
\leq\int_P u_{F,G}\,b(f)\,\d\mm\) for every \((F,G)\in\mathcal I\), we get \(\lim_{(F,G)\in\mathcal I}\int_P u_{F,G}\,\d\mm=0\).
Then we can find a \(\preceq\)-increasing sequence \(((F_k,G_k))_{k\in\N}\subseteq\mathcal I\) such that
\begin{equation}\label{eq:w_star_cont_triv_aux1}
\int_P u_{F,G}\,\d\mm\leq\frac{1}{k}\quad\text{ for every }k\in\N\text{ and }(F,G)\in\mathcal I
\text{ with }(F_k,G_k)\preceq(F,G).
\end{equation}
Given any \(k\in\N\), consider the directed set
\(I_k\coloneqq\big\{G\subseteq\Lip_{b,1}(X,\tau,\sfd):G_k\subseteq G\text{ with }G\text{ finite}\big\}\)
ordered by inclusion. Being \((u_{F_k,G})_{G\in I_k}\) a non-decreasing net of \(\tau\)-continuous functions, we have
\begin{equation}\label{eq:w_star_cont_triv_aux2}
\int_P\min_{p\in F_k}\sfd(x,p)\wedge 1\,\d\mm(x)=\int_P\lim_{G\in I_k}u_{F_k,G}\,\d\mm
=\lim_{G\in I_k}\int_P u_{F_k,G}\,\d\mm\leq\frac{1}{k}\quad\text{ for all }k\in\N
\end{equation}
thanks to \eqref{dist:recovery}, Remark \ref{rmk:monot_conv_for_nets} and \eqref{eq:w_star_cont_triv_aux1}.
Now, observe that \(\min_{p\in F_k}\sfd(x,p)\wedge 1\searrow\inf_{p\in C}\sfd(x,p)\wedge 1\) as \(k\to\infty\) for every \(x\in X\),
where \(C\) denotes the countable set \(\bigcup_{k\in\N}F_k\). By the dominated convergence theorem,
we deduce from \eqref{eq:w_star_cont_triv_aux2} that \(\int_P\inf_{p\in C}\sfd(x,p)\wedge 1\,\d\mm(x)=0\),
which implies that there exists a set \(N\in\mathscr B(X,\tau)\) such that \(\mm(N)=0\) and
\(\inf_{p\in C}\sfd(x,p)\wedge 1=0\) for every \(x\in P\setminus N\). Therefore, \(C\) is \(\sfd\)-dense
in \(P\setminus N\), in contradiction with the fact that \(P\setminus N\subseteq X\setminus{\rm S}_{\mathbb X}\)
and \(\mm(P\setminus N)>0\).\\
{\bf ii)} Assume that \(b\) is local and that \(b\llcorner{\rm S}_{\mathbb X}\) is weakly\(^*\)-type sequentially
continuous. Theorem \ref{thm:Weaver_locality_and_bdd} ensures that there exists a constant \(C>0\) such that
\(|(b\llcorner{\rm S}_{\mathbb X})(f)|\leq C\|f\|_{\Lip_b({\rm S}_{\mathbb X},\tau_{{\rm S}_{\mathbb X}},\sfd_{{\rm S}_{\mathbb X}})}\)
holds \(\mm\llcorner{\rm S}_{\mathbb X}\)-a.e.\ on \({\rm S}_{\mathbb X}\) for every
\(f\in\Lip_b({\rm S}_{\mathbb X},\tau_{{\rm S}_{\mathbb X}},\sfd_{{\rm S}_{\mathbb X}})\).
For any \(R>0\), let us denote
\[\begin{split}
A_R&\coloneqq\big\{f\in\Lip_b({\rm S}_{\mathbb X},\tau_{{\rm S}_{\mathbb X}},\sfd_{{\rm S}_{\mathbb X}})
\;\big|\;\|f\|_{\Lip_b({\rm S}_{\mathbb X},\tau_{{\rm S}_{\mathbb X}},\sfd_{{\rm S}_{\mathbb X}})}\leq R\big\},\\
B_R&\coloneqq\big\{h\in L^\infty(\mm\llcorner{\rm S}_{\mathbb X})
\;\big|\;\|h\|_{L^\infty(\mm\llcorner{\rm S}_{\mathbb X})}\leq CR\big\}.
\end{split}\]
Observe that \(b(f)\in B_R\) for every \(f\in A_R\). Since \(({\rm S}_{\mathbb X},\sfd_{{\rm S}_{\mathbb X}})\) is separable,
we know from Lemma \ref{lem:separability_Lp} that \(L^1(\mm\llcorner{\rm S}_{\mathbb X})\) is separable,
so that the restriction of the weak\(^*\) topology of \(L^\infty(\mm\llcorner{\rm S}_{\mathbb X})\) to \(B_R\)
is metrised by some distance \(\delta_R\). Moreover, fixed some countable \(\sfd\)-dense subset \((x_n)_{n\in\N}\)
of \({\rm S}_{\mathbb X}\), we define the distance \(\sfd^R\) on \(A_R\) as
\[
\sfd^R(f,g)\coloneqq\sum_{n\in\N}\frac{|f(x_n)-g(x_n)|}{2^n}\quad\text{ for every }f,g\in A_R.
\]
Using the fact that the set \(A_R\) is \(\sfd\)-equi-Lipschitz, it is straightforward to check that
\(\sfd^R\) metrises the pointwise convergence of functions in \(A_R\). Therefore, for the derivation
\(b\llcorner{\rm S}_{\mathbb X}\) the weak\(^*\)-type continuity is equivalent to the weak\(^*\)-type sequential
continuity, since both conditions are equivalent to the continuity of \(b|_{A_R}\colon(A_R,\sfd^R)\to(B_R,\delta_R)\)
for every \(R>0\).
\end{proof}

We highlight the following facts, which are immediate consequences of Proposition \ref{prop:w*_cont_vs_w*_seq_cont}:
\begin{corollary}\label{cor:w*_cont_trivial}
Let \(\mathbb X=(X,\tau,\sfd,\mm)\) be an e.m.t.m.\ space. Then the following properties hold:
\begin{itemize}
\item[\(\rm i)\)] If \(\mm({\rm S}_{\mathbb X})=0\), the null derivation is the unique weakly\(^*\)-type
continuous derivation on \(\mathbb X\).
\item[\(\rm ii)\)] If \(\mm(X\setminus{\rm S}_{\mathbb X})=0\), a derivation \(b\in\Der(\mathbb X)\)
is weakly\(^*\)-type continuous if and only if it is weakly\(^*\)-type sequentially continuous.
\end{itemize}
\end{corollary}
\subsection{Weaver derivations}\label{s:Weaver}
Motivated by Corollary \ref{cor:w*_cont_trivial}, we give the following definition:
\begin{definition}[Weaver derivation]\label{def:Weaver_der}
Let \(\mathbb X=(X,\tau,\sfd,\mm)\) be an e.m.t.m.\ space and \(b\in{\rm Der}(\mathbb X)\). Then we say that
\(b\) is a \textbf{Weaver derivation} on \(\mathbb X\) if it is weakly\(^*\)-type sequentially continuous.
We denote by \(\mathscr X(\mathbb X)\) the set of all Weaver derivations on \(\mathbb X\).
\end{definition}

The goal of our axiomatisation above is to extend Weaver's notion of `bounded measurable vector field' \cite[Definition 10.30 a)]{Weaver2018}
to the setting of e.m.t.m.\ spaces. In fact, in those cases where the set \(X\setminus{\rm S}_{\mathbb X}\) is \(\mm\)-negligible (which cover
e.g.\ all metric measure spaces), we know from Corollary \ref{cor:w*_cont_trivial} ii) that our notion of Weaver derivation is consistent with
\cite[Definition 10.30 a)]{Weaver2018}. On the other hand, many e.m.t.m.\ spaces of interest (e.g.\ Example \ref{ex:Schultz's_space} or
abstract Wiener spaces) are `purely non-\(\sfd\)-separable', meaning that \(\mm({\rm S}_{\mathbb X})=0\). If this is the case, then no non-null
derivation is weakly\(^*\)-type continuous by Corollary \ref{cor:w*_cont_trivial} i). Due to this reason, in our definition of Weaver derivation
we ask for the weak\(^*\)-type sequential continuity in lieu. As we will see in Example \ref{ex:der_Wiener}, abstract Wiener spaces -- despite
lacking in weak\(^*\)-type continuous derivations -- have plenty of weak\(^*\)-type sequential ones. The axiomatisation
we have chosen is also motivated by Theorem \ref{thm:W_vs_DM}.
\medskip

The space \(\mathscr X(\mathbb X)\) is an \(L^\infty(\mm)\)-submodule (and, thus, a vector subspace) of \(\Der(\mathbb X)\).
To any Weaver derivation \(b\in\mathscr X(\mathbb X)\), we associate the function \(|b|_W\in L^\infty(\mm)^+\), which we define as
\[
|b|_W\coloneqq\bigwedge\big\{g\in L^\infty(\mm)^+\;\big|\;|b(f)|\leq g\|f\|_{\Lip_b(X,\tau,\sfd)}\;
\mm\text{-a.e.\ for every }f\in\Lip_b(X,\tau,\sfd)\big\}.
\]
Note that \(|b(f)|\leq|b|_W\|f\|_{\Lip_b(X,\tau,\sfd)}\) holds \(\mm\)-a.e.\ on \(X\) for every \(f\in\Lip_b(X,\tau,\sfd)\).
\medskip

We also point out that all Weaver derivations \(b\in\mathscr X(\mathbb X)\) are bounded linear operators (thanks to Theorem \ref{thm:Weaver_locality_and_bdd}).
For `bounded measurable vector fields', this fact was observed in \cite[paragraph after Definition 10.30]{Weaver2018}, but in that case a stronger statement
actually holds: the image of the closed unit ball of \(\Lip_b(X,\sfd)\) under \(b\) is a weakly\(^*\) compact subset of \(L^\infty(\mm)\) (since the closed
unit ball is weakly\(^*\) compact by the Banach--Alaoglu theorem, and \(b\) is weakly\(^*\) continuous). In our setting, we have seen already in Example
\ref{ex:tau_pt_not_cpt} that \((\bar B_{\Lip_b(X,\tau,\sfd)},\tau_{pt})\) is not always compact. The following example shows that for Weaver derivations
\(b\in\mathscr X(\mathbb X)\) on an e.m.t.m.\ space \(\mathbb X\) it is not necessarily true that the image \(b(\bar B_{\Lip_b(X,\tau,\sfd)})\subseteq L^\infty(\mm)\)
is a weakly\(^*\) compact set.
\begin{example}{\rm
Let \((X,\tau,\sfd)\) be the e.m.t.\ space described in Example \ref{ex:mixed_Schultz_space}. We equip it with the restriction \(\mm\) of the 2-dimensional
Lebesgue measure, so that \(\mathbb X\coloneqq(X,\tau,\sfd,\mm)\) is an e.m.t.m.\ space. Given any function
\(f\in\Lip_b(X,\tau,\sfd)\), we have that \(f(x,\cdot)\in\Lip_b([0,1],\sfd_{\rm Eucl})\) for every \(x\in[0,1]\),
thus the derivative \(f'(x,\cdot)(t)\in\R\) exists for \(\mathscr L^1\)-a.e.\ \(t\in[0,1]\) by Rademacher's theorem.
In particular, thanks to Fubini's theorem and to \eqref{eq:mixed_Schultz_space}, it makes sense to define
\(b(f)\in L^\infty(\mm)\) as
\[
b(f)(x,t)\coloneqq f'(x,\cdot)(t)\quad\text{ for }\mm\text{-a.e.\ }(x,t)\in X.
\]
It easily follows from the classical calculus rules for the a.e.\ derivatives of Lipschitz functions from \([0,1]\)
to \(\R\) that the resulting operator \(b\colon\Lip_b(X,\tau,\sfd)\to L^\infty(\mm)\) is a derivation on \(\mathbb X\).
Moreover, if \((f_n)_{n\in\N}\subseteq\Lip_b(X,\tau,\sfd)\) and \(f\in\Lip_b(X,\tau,\sfd)\) are such that
\(\sup_{n\in\N}\|f_n\|_{\Lip_b(X,\tau,\sfd)}<+\infty\) and \(f(x,t)=\lim_n f_n(x,t)\) for every \((x,t)\in X\),
then for every \(x\in[0,1]\) the sequence \((f_n(x,\cdot))_n\) is equi-Lipschitz and equibounded,
thus \(f'_n(x,\cdot)\overset{*}{\rightharpoonup}f'(x,\cdot)\) weakly\(^*\) in \(L^\infty(0,1)\) (as \(f'_n(x,\cdot)\)
is the weak derivative of \(f_n(x,\cdot)\) by Rademacher's theorem). Hence, for any \(h\in L^1(\mm)\) we have that
\[
\int h\,b(f_n)\,\d\mm=\int_0^1\!\!\int_0^1 h(x,t)f'_n(x,\cdot)(t)\,\d t\,\d x
\to\int_0^1\!\!\int_0^1 h(x,t)f'(x,\cdot)(t)\,\d t\,\d x=\int h\,b(f)\,\d\mm
\]
as \(n\to\infty\), by Fubini's theorem, the fact that \(h(x,\cdot)\in L^1(0,1)\) for a.e.\ \(x\in[0,1]\),
and the dominated convergence theorem. This proves that \(b\) is weakly\(^*\)-type sequentially continuous,
so that \(b\in\mathscr X(\mathbb X)\).

Next, we claim that \(b(\bar B_{\Lip_b(X,\tau,\sfd)})\) is not a weakly\(^*\) closed subset of \(L^\infty(\mm)\),
thus in particular it is not a weakly\(^*\) compact subset of \(L^\infty(\mm)\). To prove it, we define
\((f_n)_{n\in\N}\subseteq\Lip_b(X,\tau,\sfd)\) as
\[
f_n(x,t)\coloneqq\psi_n(x)t\quad\text{ for every }n\in\N\text{ and }(x,t)\in X,
\]
where the function \(\psi_n\colon[0,1]\to\big[0,\frac{1}{2}\big]\) is given by
\(\psi_n(x)\coloneqq\big(\frac{n}{2}\big(x-\frac{1}{2}\big)\vee 0\big)\wedge\frac{1}{2}\) for every \(x\in[0,1]\).
As \(\|f_n\|_{C_b(X,\tau)}={\rm Osc}_X(f_n)=\frac{1}{2}\) and \(\sup_{x\in[0,1]}\Lip(f_n(x,\cdot),\sfd_{\rm Eucl})=\frac{1}{2}\),
we have \(\|f_n\|_{\Lip_b(X,\tau,\sfd)}=1\) for every \(n\in\N\) thanks to \eqref{eq:mixed_Schultz_space}.
Furthermore, for every \(n\in\N\) we have that
\[
b(f_n)(x,t)=\psi_n(x)\quad\text{ for }\mm\text{-a.e.\ }(x,t)\in X,
\]
so accordingly \(b(f_n)\overset{*}{\rightharpoonup}\frac{1}{2}\1_{[\frac{1}{2},1]\times[0,1]}\eqqcolon g\) weakly\(^*\)
in \(L^\infty(\mm)\) as \(n\to\infty\). To conclude, it remains to show that \(g\notin b(\Lip_b(X,\tau,\sfd))\),
which implies that \(b(\bar B_{\Lip_b(X,\tau,\sfd)})\) is not weakly\(^*\) closed in \(L^\infty(\mm)\). We argue
by contradiction: assume that \(g=b(f)\) for some \(f\in\Lip_b(X,\tau,\sfd)\). By Fubini's theorem, we deduce that
for a.e.\ \(x\in\big(0,\frac{1}{2}\big)\) we have \(f'(x,\cdot)(t)=0\) for a.e.\ \(t\in(0,1)\), and
for a.e.\ \(x\in\big(\frac{1}{2},1\big)\) we have \(f'(x,\cdot)(t)=\frac{1}{2}\) for a.e.\ \(t\in(0,1)\). In particular,
we can find sequences \((x_k)_k\subseteq\big(0,\frac{1}{2}\big)\) and \((y_k)_k\subseteq\big(\frac{1}{2},1\big)\) such that
\(\big|x_k-\frac{1}{2}\big|,\big|y_k-\frac{1}{2}\big|\to 0\) as \(k\to\infty\), as well as \(f'(x_k,\cdot)=0\) and
\(f'(y_k,\cdot)=\frac{1}{2}\) a.e.\ on \((0,1)\) for every \(k\in\N\). Therefore, the fundamental theorem of calculus gives that
\[
f(x_k,1)-f(x_k,0)=\int_0^1 f'(x_k,\cdot)(t)\,\d t=0,\qquad f(y_k,1)-f(y_k,0)=\int_0^1 f'(y_k,\cdot)(t)\,\d t=\frac{1}{2}.
\]
On the other hand, the \(\tau\)-continuity of \(f\) ensures that \(f(x_k,1)-f(x_k,0)\) and \(f(y_k,1)-f(y_k,0)\) converge
to the same number \(f\big(\frac{1}{2},1\big)-f\big(\frac{1}{2},0\big)\) as \(k\to\infty\), thus leading to a contradiction.
\fr}\end{example}
\begin{lemma}\label{lem:suff_cond_wstar_seq_cont}
Let \(\mathbb X=(X,\tau,\sfd,\mm)\) be an e.m.t.m.\ space such that \(\mm\) is separable. Let \(b\in D(\div;\mathbb X)\)
be given. Assume that there exists \(C>0\) such that \(|b(f)|\leq C\,\Lip(f,\sfd)\) holds \(\mm\)-a.e.\ on \(X\)
for every \(f\in\Lip_b(X,\tau,\sfd)\). Then \(b\) is a Weaver derivation.
\end{lemma}
\begin{proof}
Let \((f_n)_n\subseteq\Lip_b(X,\tau,\sfd)\) and \(f\in\Lip_b(X,\tau,\sfd)\) be such that \(f_n(x)\to f(x)\) for every \(x\in X\) and
\(M\coloneqq\sup_{n\in\N}\|f_n\|_{\Lip_b(X,\tau,\sfd)}<+\infty\). Since \(|b(f_n)|\leq CM\) holds \(\mm\)-a.e.\ for every \(n\in\N\),
the sequence \((b(f_n))_n\) is bounded in \(L^\infty(\mm)\). An application of the Banach--Alaoglu theorem, together with the
separability of \(L^1(\mm)\), ensures (up to a non-relabelled subsequence) that \(b(f_n)\overset{*}{\rightharpoonup}h\)
weakly\(^*\) in \(L^\infty(\mm)\) for some \(h\in L^\infty(\mm)\). Now fix any \(g\in\Lip_b(X,\tau,\sfd)\). We have that
\[\begin{split}
\int gh\,\d\mm&=\lim_{n\to\infty}\int g\,b(f_n)\,\d\mm=\lim_{n\to\infty}\int b(f_n g)-f_n b(g)\,\d\mm=
-\lim_{n\to\infty}\int f_n(g\,\div(b)+b(g))\,\d\mm\\
&=-\int f(g\,\div(b)+b(g))\,\d\mm=\int b(fg)-f\,b(g)\,\d\mm=\int g\,b(f)\,\d\mm
\end{split}\]
by the dominated convergence theorem. As \(\Lip_b(X,\tau,\sfd)\) is dense in \(L^1(\mm)\) (see \eqref{eq:Lip_dense_Lp}), we deduce that
\(h=b(f)\), so that the original sequence \((f_n)_n\) satisfies \(b(f_n)\overset{*}{\rightharpoonup}b(f)\)
weakly\(^*\) in \(L^\infty(\mm)\). This shows that \(b\) is weakly\(^*\)-type sequentially continuous,
so that \(b\in\mathscr X(\mathbb X)\).
\end{proof}
\subsection{Di Marino derivations}\label{s:Di_Marino}
We now introduce another subclass of Lipschitz derivations, which generalises to e.m.t.m.\ spaces the notions that have been introduced by Di Marino in \cite{DiMar:14,DiMarPhD}.
After having given the relevant definitions and discussed their main properties, we will investigate (in Theorem \ref{thm:W_vs_DM}) the relation between our notions of Weaver
derivation and of Di Marino derivation.
\begin{definition}[Di Marino derivation]\label{def:DiMar_der}
Let \(\mathbb X=(X,\tau,\sfd,\mm)\) be an e.m.t.m.\ space. Then we say that \(b\in{\rm Der}(\mathbb X)\) is a
\textbf{Di Marino derivation} on \(\mathbb X\) if there exists \(g\in L^0(\mm)^+\) such that
\begin{equation}\label{eq:DiMar_der}
|b(f)|\leq g\,\lip_\sfd(f)\quad\text{ holds }\mm\text{-a.e.\ on }X,\text{ for every }f\in\Lip_b(X,\tau,\sfd).
\end{equation}
We denote by \({\rm Der}^0(\mathbb X)\) the set of all Di Marino derivations on \(\mathbb X\). For any \(q,r\in[1,\infty]\),
we define
\[\begin{split}
{\rm Der}^q(\mathbb X)&\coloneqq\big\{b\in{\rm Der}^0(\mathbb X)\;\big|\;\eqref{eq:DiMar_der}\text{ holds for some }g\in L^q(\mm)^+\big\},\\
\Der_r^q(\mathbb X)&\coloneqq\big\{b\in\Der^q(\mathbb X)\cap D(\div;\mathbb X)\;\big|\;\div(b)\in L^r(\mm)\big\}.
\end{split}\]
\end{definition}

The space \({\rm Der}^0(\mathbb X)\) is an \(L^0(\mm)\)-submodule (and, thus, a vector subspace) of \({\rm Der}(\mathbb X)\).
Moreover, \({\rm Der}^q(\mathbb X)\) is an \(L^\infty(\mm)\)-submodule of \({\rm Der}^0(\mathbb X)\), and \(\Der^q_q(\mathbb X)\)
is a \(\Lip_b(X,\tau,\sfd)\)-submodule of \(\Der^q(\mathbb X)\), for every \(q\in[1,\infty]\). To any Di Marino derivation
\(b\in\Der^0(\mathbb X)\), we associate the function
\[
|b|\coloneqq\bigwedge\big\{g\in L^0(\mm)^+\;\big|\;|b(f)|\leq g\,\lip_\sfd(f)
\;\mm\text{-a.e.\ for every }f\in\Lip_b(X,\tau,\sfd)\big\}\in L^0(\mm)^+.
\]
Since in this paper we are primarily interested in Di Marino derivations (for defining a metric Sobolev space, in Section \ref{s:def_W1p}),
we use the notation \(|b|\) (instead e.g.\ of the more descriptive \(|b|_{DM}\)). In this regard, it is worth pointing out that if a derivation
\(b\) is both a Weaver derivation and a Di Marino derivation, it might happen that \(|b|_W\) and \(|b|\) are different.
\medskip

Note that \(|b(f)|\leq|b|\,\lip_\sfd(f)\) holds \(\mm\)-a.e.\ on \(X\) for every \(f\in\Lip_b(X,\tau,\sfd)\), and that
\[
\Der^q(\mathbb X)=\big\{b\in\Der^0(\mathbb X)\;\big|\;|b|\in L^q(\mm)\big\}.
\]
One can readily check that \((\Der^q(\mathbb X),|\cdot|)\) is an \(L^q(\mm)\)-Banach \(L^\infty(\mm)\)-module for any \(q\in(1,\infty)\).
In particular, \((\Der^q(\mathbb X),\|\cdot\|_{\Der^q(\mathbb X)})\) is a Banach space for every \(q\in(1,\infty)\), where we define
\[
\|b\|_{\Der^q(\mathbb X)}\coloneqq\||b|\|_{L^q(\mm)}\quad\text{ for every }b\in\Der^q(\mathbb X).
\]
Furthermore, in analogy with \cite[Eq.\ (4.9)]{AILP24}, for any \(q\in(1,\infty)\) we define the space \(L^q_\Lip(T\mathbb X)\) as
\begin{equation}\label{eq:def_Lip_tg_mod}
L^q_\Lip(T\mathbb X)\coloneqq{\rm cl}_{\Der^q(\mathbb X)}(\Der^q_q(\mathbb X)).
\end{equation}
The notation \(L^q_\Lip(T\mathbb X)\), which reminds of the fact that its elements are defined in duality with the space \(\Lip_b(X,\tau,\sfd)\),
is needed to distinguish it from the `Sobolev' tangent modules \(L^q(T\mathbb X)\) and \(L^q_{\rm Sob}(T\mathbb X)\) that we introduced in Section \ref{s:H1p}.
The relation between \(L^q_\Lip(T\mathbb X)\) and \(L^q_{\rm Sob}(T\mathbb X)\) (in the setting of metric measure spaces) will
be an object of study in the forthcoming paper \cite{AILP25}. We claim that
\[
hb\in L^q_\Lip(T\mathbb X)\quad\text{ for every }h\in L^\infty(\mm)\text{ and }b\in L^q_\Lip(T\mathbb X).
\]
To prove it, take a sequence \((b_n)_n\subseteq\Der^q_q(\mathbb X)\) such that \(b_n\to b\) strongly in \(\Der^q(\mathbb X)\), and (using \eqref{eq:Lip_dense_Lp})
one can find a sequence \((h_n)_n\subseteq\Lip_b(X,\tau,\sfd)\) such that \(\sup_{n\in\N}\|h_n\|_{C_b(X,\tau)}\leq\|h\|_{L^\infty(\mm)}\) and \(h(x)=\lim_n h_n(x)\)
for \(\mm\)-a.e.\ \(x\in X\), so that \(\Der^q_q(\mathbb X)\ni h_n b_n\to h b\) strongly in \(\Der^q(\mathbb X)\) by the dominated convergence theorem,
and thus accordingly \(hb\in L^q_\Lip(T\mathbb X)\). Since \((\Der^q(\mathbb X),|\cdot|)\) is an \(L^q(\mm)\)-Banach \(L^\infty(\mm)\)-module,
we deduce that \(L^q_\Lip(T\mathbb X)\) is an \(L^q(\mm)\)-Banach \(L^\infty(\mm)\)-module.
\medskip

The next result, whose proof is very similar to that of Lemma \ref{lem:suff_cond_wstar_seq_cont}, studies the continuity properties
of Di Marino derivations with divergence.
\begin{lemma}\label{lem:DM_der_w-cont}
Let \(\mathbb X=(X,\tau,\sfd,\mm)\) be an e.m.t.m.\ space. Let \(q\in[1,\infty)\) and \(b\in\Der^q_q(\mathbb X)\). Then
\begin{equation}\label{eq:DM_der_w-cont}
b|_{\bar B_{\Lip_b(X,\tau,\sfd)}}\colon(\bar B_{\Lip_b(X,\tau,\sfd)},\tau_{pt})\longrightarrow(L^q(\mm),\tau_w)\;\text{ is sequentially continuous},
\end{equation}
where \(\tau_w\) denotes the weak topology of \(L^q(\mm)\).
\end{lemma}
\begin{proof}
First, note that \(|b(f)|\leq|b|\,\lip_\sfd(f)\leq\Lip(f,\sfd)|b|\in L^q(\mm)\) \(\mm\)-a.e.\ for every \(f\in\Lip_b(X,\tau,\sfd)\),
thus \(b(f)\in L^q(\mm)\) for every \(f\in\Lip_b(X,\tau,\sfd)\). Now fix any sequence \((f_n)_n\subseteq\bar B_{\Lip_b(X,\tau,\sfd)}\).
The above estimate shows that the sequence \((b(f_n))_n\) is dominated in \(L^q(\mm)\), thus the Dunford--Pettis theorem ensures
(up to taking a non-relabelled subsequence) that \(b(f_n)\rightharpoonup G\) weakly in \(L^q(\mm)\), for some \(G\in L^q(\mm)\).
By using also the dominated convergence theorem, we then obtain that
\[\begin{split}
\int h\,G\,\d\mm&=\lim_{n\to\infty}\int h\,b(f_n)\,\d\mm=\lim_{n\to\infty}\int b(hf_n)-f_n\,b(h)\,\d\mm\\
&=-\lim_{n\to\infty}\int f_n(h\,\div(b)+b(h))\,\d\mm=-\int f(h\,\div(b)+b(h))\,\d\mm=\int h\,b(h)\,\d\mm
\end{split}\]
for every \(h\in\Lip_b(X,\tau,\sfd)\). Letting \(p\in(1,\infty)\) be the conjugate exponent of \(q\), we know
from \eqref{eq:Lip_dense_Lp} that \(\Lip_b(X,\tau,\sfd)\)
is strongly dense (resp.\ weakly\(^*\) dense) in \(L^p(\mm)\) if \(p<\infty\) (resp.\ if \(p=\infty\)), thus we get that \(G=b(f)\).
Consequently, we have that the original sequence \((f_n)_n\) satisfies \(b(f_n)\rightharpoonup b(f)\) weakly in \(L^q(\mm)\). This shows
the validity of \eqref{eq:DM_der_w-cont}.
\end{proof}

As a consequence, Di Marino derivations with divergence are local:
\begin{corollary}\label{cor:DM_der_local}
Let \(\mathbb X=(X,\tau,\sfd,\mm)\) be an e.m.t.m.\ space. Let \(q\in[1,\infty)\) and \(b\in\Der^q_q(\mathbb X)\) be given. Then \(b\) is a local derivation.
\end{corollary}
\begin{proof}
Fix any \(f\in\Lip_b(X,\tau,\sfd)\). For any \(n\in\N\), we define the auxiliary function \(\phi_n\colon\R\to\R\) as
\[
\phi_n(t)\coloneqq\left\{\begin{array}{lll}
t+\frac{1}{n}\\
0\\
t-\frac{1}{n}
\end{array}\quad\begin{array}{lll}
\text{ if }t\leq-\frac{1}{n},\\
\text{ if }-\frac{1}{n}<t<\frac{1}{n},\\
\text{ if }t\geq\frac{1}{n}.
\end{array}\right.
\]
Note that \(\phi_n\circ f\in\Lip_b(X,\tau,\sfd)\) with \(\|\phi_n\circ f\|_{C_b(X,\tau)}\leq\|f\|_{C_b(X,\tau)}+1\) and \(\Lip(\phi_n\circ f,\sfd)\leq\Lip(f,\sfd)\).
It also holds that \((\phi_n\circ f)(x)\to f(x)\) for every \(x\in X\), thus Lemma \ref{lem:DM_der_w-cont} gives that \(b(\phi_n\circ f)\rightharpoonup b(f)\) weakly
in \(L^q(\mm)\). Moreover, one can readily check that \(\lip_\sfd(\phi_n\circ f)\leq(\lip_{\sfd_{\rm Eucl}}(\phi_n)\circ f)\,\lip_\sfd(f)\), so that the \(\mm\)-a.e.\ inequality
\(|b(\phi_n\circ f)|\leq|b|\,\lip_\sfd(\phi_n\circ f)\) implies that \(b(\phi_n\circ f)=0\) holds \(\mm\)-a.e.\ on the set \(\{f=0\}\) (as \(\lip_{\sfd_{\rm Eucl}}(\phi_n)(0)=0\)), 
thus accordingly \(b(f)=0\) holds \(\mm\)-a.e.\ on \(\{f=0\}\).
\end{proof}
\begin{proposition}\label{prop:glob_to_loc}
Let \(\mathbb X=(X,\tau,\sfd,\mm)\) be an e.m.t.m.\ space. Let \(b\in{\rm Der}(\mathbb X)\) be a local derivation.
Assume that there exists a function \(g\in L^0(\mm)^+\) such that
\[
|b(f)|\leq g\|f\|_{\Lip_b(X,\tau,\sfd)}\quad\text{ holds }\mm\text{-a.e.\ on }X,\text{ for every }f\in\Lip_b(X,\tau,\sfd).
\]
Let \(C\subseteq X\) be a \(\tau\)-closed set. Then for any entourage \(\mathcal U\in\mathfrak B_{\tau,\sfd}\) we have that
\begin{equation}\label{eq:glob_to_loc_cl1}
|b(f)|\leq g\,\Lip(f,C\cap\mathcal U[\cdot],\sfd)\quad\text{ holds }\mm\text{-a.e.\ on }C,\text{ for every }f\in\Lip_b(X,\tau,\sfd).
\end{equation}
In particular, if the topology \(\tau\) is metrisable on \(C\), then (letting \(\sfd_C\coloneqq\sfd|_{C\times C}\)) we have that
\begin{equation}\label{eq:glob_to_loc_cl2}
|b(f)|\leq g\,\lip_{\sfd_C}(f|_C)\quad\text{ holds }\mm\text{-a.e.\ on }C,\text{ for every }f\in\Lip_b(X,\tau,\sfd).
\end{equation}
\end{proposition}
\begin{proof}
By definition of uniform structure, we can find \(\mathcal V\in\mathfrak U_{\tau,\sfd}\) such that \(\mathcal V\circ\mathcal V\subseteq\mathcal U\), where we set
\[
\mathcal V\circ\mathcal V\coloneqq\big\{(x,z)\in X\times X\;\big|\;(x,y),(y,z)\in\mathcal V\text{ for some }y\in X\big\}.
\]
Fix any \(\varepsilon>0\) and \(f\in\Lip_b(X,\tau,\sfd)\). Since \(\mm\) is a Radon measure,
we can find a sequence \((K_n)_n\) of pairwise disjoint \(\tau\)-compact subsets of \(X\) such that \({\rm Osc}_{K_n}(f)\leq\varepsilon\)
for every \(n\in\N\) and \(\mm\big(X\setminus\bigcup_{n\in\N}K_n\big)=0\). Now fix \(n\in\N\). Given any \(x\in K_n\cap C\), we can find
a \(\tau\)-closed \(\tau\)-neighbourhood \(F^x_n\) of \(x\) such that \(F^x_n\subseteq\mathcal V[x]\). Since \(K_n\cap C\) is \(\tau\)-compact, there exist \(k(n)\in\N\) and
\(x_{n,1},\ldots,x_{n,k(n)}\in K_n\cap C\) such that \(K_n\cap C\subseteq\bigcup_{i=1}^{k(n)}F_{n,i}\), where we set \(F_{n,i}\coloneqq F^{x_{n,i}}_n\). Denote also
\(K_{n,i}\coloneqq K_n\cap C\cap F_{n,i}\) for every \(i=1,\ldots,k(n)\). Since \(K_{n,i}\) is \(\tau\)-compact, by applying Corollary \ref{cor:Lip_ext_cpt} we obtain
a function \(\tilde f_{n,i}\in\Lip_b(X,\tau,\sfd)\) such that
\[
\tilde f_{n,i}|_{K_{n,i}}=f|_{K_{n,i}},\qquad\Lip(\tilde f_{n,i},\sfd)=\Lip(f,K_{n,i},\sfd),\qquad{\rm Osc}_X(\tilde f_{n,i})={\rm Osc}_{K_{n,i}}(f)\leq\varepsilon.
\]
Next, let us define the function \(f_{n,i}\in\Lip_b(X,\tau,\sfd)\) as \(f_{n,i}\coloneqq\tilde f_{n,i}-\inf_X\tilde f_{n,i}\). Note that
\[
\Lip(f_{n,i},\sfd)=\Lip(f,K_{n,i},\sfd),\qquad\|f_{n,i}\|_{C_b(X,\tau)}\leq\varepsilon.
\]
Therefore, the locality of \(b\) ensures that the following inequalities hold for \(\mm\)-a.e.\ point \(x\in K_{n,i}\):
\[\begin{split}
|b(f)|(x)&=|b(\tilde f_{n,i})|(x)=|b(f_{n,i})|(x)\leq g(x)\|f_{n,i}\|_{\Lip_b(X,\tau,\sfd)}\leq g(x)(\Lip(f,K_{n,i},\sfd)+\varepsilon)\\
&\leq g(x)(\Lip(f,C\cap\mathcal V[x_{n,i}],\sfd)+\varepsilon)\leq g(x)(\Lip(f,C\cap\mathcal U[x],\sfd)+\varepsilon).
\end{split}\]
By the arbitrariness of \(n\in\N\) and \(i=1,\ldots,k(n)\), it follows that \(|b(f)|\leq g(\Lip(f,C\cap\mathcal U[\cdot],\sfd)+\varepsilon)\) holds \(\mm\)-a.e.\ on \(C\).
Thanks to the arbitrariness of \(\varepsilon>0\), we thus conclude that \eqref{eq:glob_to_loc_cl1} is verified.

Finally, assume that the restriction \(\tau_C\) of the topology \(\tau\) to \(C\) is metrisable. Recalling \eqref{eq:unif_struct_restr}, we can find a sequence
\((\mathcal U_n)_{n\in\N}\subseteq\mathfrak B_{\tau,\sfd}\) such that \(\{\mathcal U_n|_{C\times C}:n\in\N\}\subseteq\mathfrak B_{\tau_C,\sfd_C}\) is a basis of
entourages for \(\mathfrak U_{\tau_C,\sfd_C}\). Given that \((\mathcal U_n|_{C\times C})[x]=C\cap\mathcal U_n[x]\) and \(\lip_{\sfd_C}(f|_C)(x)=\inf_{n\in\N}\Lip(f,C\cap\mathcal U_n[x],\sfd)\)
hold for every \(x\in C\), we have that the inequality in \eqref{eq:glob_to_loc_cl2} follows from \eqref{eq:glob_to_loc_cl1}.
\end{proof}
\begin{theorem}[Relation between Weaver and Di Marino derivations]\label{thm:W_vs_DM}
Let \(\mathbb X=(X,\tau,\sfd,\mm)\) be an e.m.t.m.\ space such that \(\mm\) is separable. Then it holds that
\[
\Der^\infty_1(\mathbb X)\subseteq\mathscr X(\mathbb X)
\]
and \(|b|_W\leq|b|\) holds \(\mm\)-a.e.\ on \(X\) for every \(f\in\Der^\infty_1(\mathbb X)\). Assuming in addition that \(\tau\)
is metrisable on all \(\tau\)-compact subsets of \(X\), we also have that
\[
\mathscr X(\mathbb X)\subseteq{\rm Der}^\infty(\mathbb X)
\]
and \(|b|_W=|b|\) holds \(\mm\)-a.e.\ on \(X\) for every \(b\in\mathscr X(\mathbb X)\).
\end{theorem}
\begin{proof}
Assume \(\mm\) is separable and fix \(b\in\Der^\infty_1(\mathbb X)\). As \(|b(f)|\leq|b|\,\lip_\sfd(f)\leq\||b|\|_{L^\infty(\mm)}\Lip(f,\sfd)\)
holds \(\mm\)-a.e.\ on \(X\) for every \(f\in\Lip_b(X,\tau,\sfd)\), we know from Lemma \ref{lem:suff_cond_wstar_seq_cont} that \(b\in\mathscr X(\mathbb X)\).
Moreover, the \(\mm\)-a.e.\ inequalities \(|b(f)|\leq|b|\,\lip_\sfd(f)\leq\|f\|_{\Lip_b(X,\tau,\sfd)}|b|\) imply that \(|b|_W\leq|b|\) \(\mm\)-a.e.\ on \(X\).

Now, assume in addition that \(\tau\) is metrisable on all \(\tau\)-compact sets and fix any \(b\in\mathscr X(\mathbb X)\). As \(\mm\) is a Radon measure,
we find a sequence \((K_n)_n\) of \(\tau\)-compact sets such that \(\mm\big(X\setminus\bigcup_{n\in\N}K_n\big)=0\). Since \(b\) is local by Theorem
\ref{thm:Weaver_locality_and_bdd}, and \(\tau\) is metrisable on \(K_n\), we deduce from Proposition \ref{prop:glob_to_loc} that
\[
|b(f)|\leq|b|_W\,\lip_{\sfd_{K_n}}(f|_{K_n})\leq|b|_W\,\lip_\sfd(f)\quad\text{ holds }\mm\text{-a.e.\ on }K_n\text{, for every }f\in\Lip_b(X,\tau,\sfd).
\]
By the arbitrariness of \(n\in\N\), it follows that \(|b(f)|\leq|b|_W\,\lip_\sfd(f)\) holds \(\mm\)-a.e.\ on \(X\) for every \(f\in\Lip_b(X,\tau,\sfd)\).
This proves that \(b\in{\rm Der}^\infty(\mathbb X)\) and \(|b|\leq|b|_W\), thus yielding the statement.
\end{proof}

We close this section with a result that illustrates the relation between derivations on an e.m.t.m.\ space and derivations on
its compactification. We denote by \(\iota^*\colon\Lip_b(\hat X,\hat\tau,\hat\sfd)\to\Lip_b(X,\tau,\sfd)\)
the inverse of the Gelfand transform \(\Gamma\colon\Lip_b(X,\tau,\sfd)\to\Lip_b(\hat X,\hat\tau,\hat\sfd)\), cf.\ with
Lemma \ref{lem:Gelfand_isom}. With the same symbol \(\iota^*\) we denote the linear bijection
\(\iota^*\colon L^0(\hat\mm)\to L^0(\mm)\) that maps the \(\hat\mm\)-a.e.\ equivalence class of a Borel function
\(\hat f\colon\hat X\to\R\) to the \(\mm\)-a.e.\ equivalence class of \(\hat f\circ\iota\colon X\to\R\), whereas
\(\iota_*\colon L^0(\mm)\to L^0(\hat\mm)\) denotes its inverse.
\begin{proposition}[Derivations on the compactification]
Let \(\mathbb X=(X,\tau,\sfd,\mm)\) be an e.m.t.m.\ space. Denote by \(\hat{\mathbb X}=(\hat X,\hat\tau,\hat\sfd,\hat\mm)\)
its compactification, with embedding \(\iota\colon X\hookrightarrow\hat X\). Let us define the operator
\(\iota_*\colon{\rm Der}(\mathbb X)\to{\rm Der}(\hat{\mathbb X})\) as
\[
(\iota_*b)(\hat f)\coloneqq\iota_*(b(\iota^*\hat f))\in L^0(\hat\mm)\quad\text{ for every }b\in{\rm Der}(\mathbb X)\text{ and }\hat f\in\Lip_b(\hat X,\hat\tau,\hat\sfd).
\]
Then \(\iota_*\) is a linear bijection such that \(\iota_*(hb)=(\iota_*h)(\iota_*b)\) for every
\(b\in{\rm Der}(\mathbb X)\) and \(h\in L^0(\mm)\). Moreover, the following properties are satisfied:
\begin{itemize}
\item[\(\rm i)\)] \(\iota_*(D(\div;\mathbb X))=D(\div;\hat{\mathbb X})\) and \(\div(\iota_*b)=\iota_*(\div(b))\)
for every \(b\in D(\div;\mathbb X)\).
\item[\(\rm ii)\)] \(\iota_*(\mathscr X(\mathbb X))\subseteq\mathscr X(\hat{\mathbb X})\)
and \(|\iota_*b|_W=\iota_*|b|_W\) for every \(b\in\mathscr X(\mathbb X)\).
\item[\(\rm iii)\)] Given any derivation \(b\in\Der(\mathbb X)\), we have that \(b\) is local if and only if \(\iota_*b\) is local.
\item[\(\rm iv)\)] \(\iota_*(\Der^0(\mathbb X))\subseteq\Der^0(\hat{\mathbb X})\) and \(|\iota_*b|\leq\iota_*|b|\) for every \(b\in\Der^0(\mathbb X)\).
In particular, we have that \(\iota_*(\Der^q(\mathbb X))\subseteq\Der^q(\hat{\mathbb X})\) and \(\iota_*(\Der_r^q(\mathbb X))\subseteq\Der_r^q(\hat{\mathbb X})\)
for every \(q,r\in[1,\infty]\).
\item[\(\rm v)\)] Assume in addition that \(\tau\) is metrisable on all \(\tau\)-compact subsets of \(X\). Then it holds that
\(\iota_*({\rm Der}^q_q(\mathbb X))={\rm Der}^q_q(\hat{\mathbb X})\) for every \(q\in[1,\infty]\), and that \(|\iota_*b|=\iota_*|b|\)
for every \(b\in{\rm Der}^q_q(\mathbb X)\).
\end{itemize}
\end{proposition}
\begin{proof}
Let \(b\in\Der(\mathbb X)\) be given. The map \(\iota_*b\colon\Lip_b(\hat X,\hat\tau,\hat\sfd)\to L^0(\hat\mm)\) is linear
(as a composition of linear maps). Moreover, for every \(\hat f,\hat g\in\Lip_b(\hat X,\hat\tau,\hat\sfd)\) we have that
\[
(\iota_*b)(\hat f\hat g)=\iota_*\big(b((\iota^*\hat f)(\iota^*\hat g))\big)
=\iota_*\big((\iota^*\hat f)\,b(\iota^*\hat g)+(\iota^*\hat g)\,b(\iota^*\hat f)\big)
=\hat f\,(\iota_*b)(\hat g)+\hat g\,(\iota_*b)(\hat f),
\]
so that \(\iota_*b\) satisfies the Leibniz rule, thus \(\iota_*b\in\Der(\hat{\mathbb X})\). The resulting map
\(\iota_*\colon\Der(\mathbb X)\to\Der(\hat{\mathbb X})\) is clearly linear. Similar arguments show that
\[
(\iota^*\hat b)(f)\coloneqq\iota^*\big(\hat b(\Gamma(f))\big)\in L^0(\mm)\quad\text{ for every }\hat b\in\Der(\hat{\mathbb X})
\text{ and }f\in\Lip_b(X,\tau,\sfd)
\]
defines a linear operator \(\iota^*\colon\Der(\hat{\mathbb X})\to\Der(\mathbb X)\) whose inverse is the map
\(\iota_*\colon\Der(\mathbb X)\to\Der(\hat{\mathbb X})\), thus in particular the latter is a bijection. For any
\(b\in\Der(\mathbb X)\) and \(h\in L^0(\mm)\), we also have that
\[
(\iota_*(hb))(\hat f)=\iota_*(h\,b(\iota^*\hat f))=(\iota_*h)\,\iota_*(b(\iota^*\hat f))=(\iota_*h)\,(\iota_*b)(\hat f)
\quad\text{ for every }\hat f\in\Lip_b(\hat X,\hat\tau,\hat\sfd),
\]
which gives that \(\iota_*(hb)=(\iota_*h)(\iota_*b)\). Let us now pass to the verification of i), ii), iii), iv) and v).\\
\(\bf i)\) Let \(b\in\Der(\mathbb X)\) be a given derivation. Note that \(b(f)\in L^1(\mm)\) for every \(f\in\Lip_b(X,\tau,\sfd)\)
if and only if \((\iota_*b)(\hat f)\in L^1(\hat\mm)\) for every \(\hat f\in\Lip_b(\hat X,\hat\tau,\hat\sfd)\). Moreover,
if \(b\in D(\div;\mathbb X)\), then
\[
\int(\iota_*b)(\hat f)\,\d\hat\mm=\int b(\iota^*\hat f)\,\d\mm=-\int(\iota^*\hat f)\,\div(b)\,\d\mm
=-\int\hat f\,\iota_*(\div(b))\,\d\hat\mm
\]
holds for every \(\hat f\in\Lip_b(\hat X,\hat\tau,\hat\sfd)\), so that \(\iota_*b\in D(\div;\hat{\mathbb X})\)
and \(\div(\iota_*b)=\iota_*(\div(b))\). Conversely, if we assume \(\iota_*b\in D(\div;\hat{\mathbb X})\),
then similar computations show that \(b\in D(\div;\mathbb X)\). This proves i).\\
\(\bf ii)\) If \(b\in\mathscr X(\mathbb X)\), then \((\iota_*b)(\hat f)=\iota_*(b(\iota^*\hat f))\in L^\infty(\hat\mm)\)
for every \(\hat f\in\Lip_b(\hat X,\hat\tau,\hat\sfd)\). Moreover, assuming that \((\hat f_n)_n\subseteq\Lip_b(\hat X,\hat\tau,\hat\sfd)\)
and \(\hat f\in\Lip_b(\hat X,\hat\tau,\hat\sfd)\) satisfy \(\sup_{n\in\N}\|\hat f_n\|_{\Lip_b(\hat X,\hat\tau,\hat\sfd)}<+\infty\)
and \(\hat f(\varphi)=\lim_n\hat f_n(\varphi)\) for every \(\varphi\in\hat X\), we have
\(\sup_{n\in\N}\|\iota^*\hat f_n\|_{\Lip_b(X,\tau,\sfd)}<+\infty\) by Lemma \ref{lem:Gelfand_isom} and
\((\iota^*\hat f)(x)=\hat f(\iota(x))=\lim_n\hat f_n(\iota(x))=\lim_n(\iota^*\hat f_n)(x)\) for every \(x\in X\). Hence,
the weak\(^*\)-type sequential continuity of \(b\) ensures that \(b(\iota^*\hat f_n)\overset{*}{\rightharpoonup}b(\iota^*\hat f)\)
weakly\(^*\) in \(L^\infty(\mm)\), so that accordingly
\[
(\iota_*b)(\hat f_n)=\iota_*(b(\iota^*\hat f_n))\overset{*}{\rightharpoonup}\iota_*(b(\iota^*\hat f))=(\iota_*b)(\hat f)
\quad\text{ weakly\(^*\) in }L^\infty(\hat\mm).
\]
This shows that \(\iota_*b\in\mathscr X(\hat{\mathbb X})\). Finally, it follows from the \(\hat\mm\)-a.e.\ inequalities
\[\begin{split}
|(\iota_*b)(\hat f)|&=\iota_*|b(\iota^*\hat f)|\leq\|\iota^*\hat f\|_{\Lip_b(X,\tau,\sfd)}\,\iota_*|b|_W
=\|\hat f\|_{\Lip_b(\hat X,\hat\tau,\hat\sfd)}\,\iota_*|b|_W,\\
\iota_*|b(f)|&=|(\iota_*b)(\Gamma(f))|\leq\|\Gamma(f)\|_{\Lip_b(\hat X,\hat\tau,\hat\sfd)}|\iota_*b|_W
=\|f\|_{\Lip_b(X,\tau,\sfd)}|\iota_*b|_W,
\end{split}\]
which hold for all \(\hat f\in\Lip_b(\hat X,\hat\tau,\hat\sfd)\) and \(f\in\Lip_b(X,\tau,\sfd)\), that
\(|\iota_*b|_W=\iota_*|b|_W\). This proves ii).\\
{\bf iii)} Note that \(\1_{\{\Gamma(f)=0\}}=\iota_*\1_{\{f=0\}}\) holds \(\hat\mm\)-a.e.\ on \(\hat X\) for every \(f\in\Lip_b(X,\tau,\sfd)\). In particular,
\[
\1_{\{\Gamma(f)=0\}}(\iota_*b)(\Gamma(f))=\iota_*(\1_{\{f=0\}}b(f))\quad\text{ holds }\hat\mm\text{-a.e.\ on }\hat X,
\]
whence it follows that \((\iota_*b)(\Gamma(f))=0\) \(\hat\mm\)-a.e.\ on \(\{\Gamma(f)=0\}\) if and only if \(b(f)=0\) \(\mm\)-a.e.\ on \(\{f=0\}\).
As \(\Gamma\colon\Lip_b(X,\tau,\sfd)\to\Lip_b(\hat X,\hat\tau,\hat\sfd)\) is bijective, we deduce that \(b\) is local if and only if \(\iota_*b\) is local.\\
{\bf iv)} If \(b\in\Der^0(\mathbb X)\), then by applying \eqref{eq:ineq_lip_a_cpt_a.e.} we obtain the \(\hat\mm\)-a.e.\ inequalities
\[
|(\iota_*b)(\hat f)|=\iota_*|b(\iota^*\hat f)|\leq(\iota_*|b|)(\iota_*\lip_\sfd(\iota^*\hat f))\leq(\iota_*|b|)\,\lip_{\hat\sfd}(\hat f)
\quad\text{ for every }\hat f\in\Lip_b(\hat X,\hat\tau,\hat\sfd),
\]
whence it follows that \(\iota_*b\in\Der^0(\hat{\mathbb X})\) and \(|\iota_*b|\leq\iota_*|b|\).\\
{\bf v)} Fix any \(\hat b\in\Der^q_q(\hat{\mathbb X})\). We know from Corollary \ref{cor:DM_der_local} if \(q<\infty\), or from Theorems \ref{thm:W_vs_DM}
and \ref{thm:Weaver_locality_and_bdd} if \(q=\infty\), that \(\hat b\) is a local derivation. For any \(f\in\Lip_b(X,\tau,\sfd)\), we have the \(\mm\)-a.e.\ inequalities
\[
|(\iota^*\hat b)(f)|=\iota^*|\hat b(\Gamma(f))|\leq(\iota^*|\hat b|)\big(\iota^*\lip_{\hat\sfd}(\Gamma(f))\big)\leq\Lip(\Gamma(f),\hat\sfd)(\iota^*|\hat b|)
\leq\|f\|_{\Lip_b(X,\tau,\sfd)}\iota^*|\hat b|.
\]
Therefore, Proposition \ref{prop:glob_to_loc} guarantees that for every \(\tau\)-compact set \(K\subseteq X\) we have that
\[
|(\iota^*\hat b)(f)|\leq(\iota^*|\hat b|)\,\lip_{\sfd_K}(f|_K)\quad\text{ holds }\mm\text{-a.e.\ on }K,\text{ for every }f\in\Lip_b(X,\tau,\sfd).
\]
Since the Radon measure \(\mm\) is concentrated on the union \(\bigcup_n K_n\) of countably many \(\tau\)-compact subsets \((K_n)_{n\in\N}\) of \(X\),
we deduce that \(|(\iota^*\hat b)(f)|\leq(\iota^*|\hat b|)\,\lip_\sfd(f)\) \(\mm\)-a.e.\ on \(X\), so that \(\iota^*\hat b\in\Der^q(\mathbb X)\)
and \(|\iota^*\hat b|\leq\iota^*|\hat b|\). Taking also i) and iv) into account, we can finally conclude that v) holds.
\end{proof}
\section{Sobolev spaces via Lipschitz derivations}\label{s:Sob_space}
\subsection{The space \texorpdfstring{\(W^{1,p}\)}{W1p}}\label{s:def_W1p}
We introduce a new notion of metric Sobolev space \(W^{1,p}(\mathbb X)\) over an e.m.t.m.\ space \(\mathbb X\), defined via an integration-by-parts
formula in duality with the space \(\Der^q_q(\mathbb X)\) of Di Marino derivations with divergence. Our definition generalises Di Marino's notion of
\(W^{1,p}\) space for metric measure spaces (\cite[Definition 1.5]{DiMar:14}, \cite[Definition 7.1.4]{DiMarPhD}) to the extended setting.
\begin{definition}[The Sobolev space \(W^{1,p}(\mathbb X)\)]\label{def:Sobolev_space_via_der}
Let \(\mathbb X=(X,\tau,\sfd,\mm)\) be an e.m.t.m.\ space. Let \(p,q\in(1,\infty)\) be conjugate exponents.
Then we define the \textbf{Sobolev space} \(W^{1,p}(\mathbb X)\) as the set of all functions \(f\in L^p(\mm)\)
for which there exists a linear operator \(L_f\colon{\rm Der}^q_q(\mathbb X)\to L^1(\mm)\) such that:
\begin{itemize}
\item[\(\rm i)\)] There exists a function \(g\in L^q(\mm)^+\) such that \(|L_f(b)|\leq g|b|\) for every \(b\in{\rm Der}^q_q(\mathbb X)\).
\item[\(\rm ii)\)] \(L_f(hb)=h\,L_f(b)\) for every \(h\in\Lip_b(X,\tau,\sfd)\) and \(b\in\Der^q_q(\mathbb X)\).
\item[\(\rm iii)\)] The following integration-by-parts formula holds:
\[
\int L_f(b)\,\d\mm=-\int f\,{\rm div}(b)\,\d\mm\quad\text{ for every }b\in{\rm Der}^q_q(\mathbb X).
\]
\end{itemize}
Given any function \(f\in W^{1,p}(\mathbb X)\), we define its \textbf{minimal \(p\)-weak gradient} \(|Df|\in L^p(\mm)^+\) as
\[
|Df|\coloneqq\bigwedge\big\{g\in L^p(\mm)^+\;\big|\;|L_f(b)|\leq g|b|\;\;\forall b\in{\rm Der}^q_q(\mathbb X)\big\}
=\bigvee_{b\in{\rm Der}^q_q(\mathbb X)}\1_{\{|b|>0\}}\frac{|L_f(b)|}{|b|}.
\]
\end{definition}

We use the notation \(|Df|\) (instead e.g.\ of \(|Df|_W\)) because the space \(W^{1,p}(\mathbb X)\) will be our main object of study in
the rest of the paper. Note that \(|L_f(b)|\leq|Df||b|\) holds \(\mm\)-a.e.\ for every \(f\in W^{1,p}(\mathbb X)\)
and \(b\in\Der^q_q(\mathbb X)\). It can also be readily checked that
\[
\|f\|_{W^{1,p}(\mathbb X)}\coloneqq\big(\|f\|_{L^p(\mm)}^p+\||Df|\|_{L^p(\mm)}^p\big)^{1/p}\quad\text{ for every }f\in W^{1,p}(\mathbb X)
\]
defines a complete norm on \(W^{1,p}(\mathbb X)\), so that \((W^{1,p}(\mathbb X),\|\cdot\|_{W^{1,p}(\mathbb X)})\) is a Banach space.
\medskip

Some more comments on the Sobolev space \(W^{1,p}(\mathbb X)\):
\begin{itemize}
\item Since \(\int h\,L_f(b)\,\d\mm=-\int f\,\div(hb)\,\d\mm\) for every \(h\in\Lip_b(X,\tau,\sfd)\), and \(\Lip_b(X,\tau,\sfd)\)
is weakly\(^*\) dense in \(L^\infty(\mm)\) by \eqref{eq:Lip_dense_Lp}, the map \(L_f\colon\Der^q_q(\mathbb X)\to L^1(\mm)\)
is uniquely determined.
\item It easily follows from the uniqueness of \(L_f\) that \(W^{1,p}(\mathbb X)\ni f\mapsto L_f\) is a linear operator, whose target is
the vector space of all linear operators from \(\Der^q_q(\mathbb X)\) to \(L^1(\mm)\).
\item \(\Lip_b(X,\tau,\sfd)\subseteq W^{1,p}(\mathbb X)\) and \(L_f(b)=b(f)\) for every \(f\in\Lip_b(X,\tau,\sfd)\) and
\(b\in\Der^q_q(\mathbb X)\), thus in particular \(|Df|\leq\lip_\sfd(f)\) holds \(\mm\)-a.e.\ on \(X\) for every \(f\in\Lip_b(X,\tau,\sfd)\).
\item For any \(f\in W^{1,p}(\mathbb X)\), the operator \(L_f\colon\Der^q_q(\mathbb X)\to L^1(\mm)\) can be uniquely extended to an element \(L_f\in L^q_\Lip(T\mathbb X)^*\), whose pointwise norm \(|L_f|\) coincides with \(|Df|\).
\end{itemize}
\begin{example}{\rm
Let \((X,\tau,\sfd_{\rm discr})\) be a `purely-topological' e.m.t.\ space (as in Example \ref{ex:Schultz's_space}) together with a finite Radon measure \(\mm\),
so that \(\mathbb X\coloneqq(X,\tau,\sfd_{\rm discr},\mm)\) is an e.m.t.m.\ space. For any given function \(f\in\Lip_b(X,\tau,\sfd_{\rm discr})\),
we have that \(\Lip(f,U,\sfd_{\rm discr})={\rm Osc}_U(f)\) for every \(U\in\tau\), thus the \(\tau\)-continuity of \(f\) implies that \(\lip_{\sfd_{\rm discr}}(f)(x)=0\)
for all \(x\in X\). In particular, \(\Der^q_q(\mathbb X)=\Der^q(\mathbb X)=\{0\}\) for every \(q\in[1,\infty]\), whence it follows that
\(W^{1,p}(\mathbb X)=L^p(\mm)\) for every \(p\in(1,\infty)\), with \(L_f=0\) and thus \(|Df|=0\) for every \(f\in W^{1,p}(\mathbb X)\).
\fr}\end{example}
\subsection{The equivalence \texorpdfstring{\(H^{1,p}=W^{1,p}\)}{H=W}}\label{s:H=W}
The goal of this section is to prove that the metric Sobolev spaces \(W^{1,p}(\mathbb X)\) and \(H^{1,p}(\mathbb X)\) coincide on \emph{any} e.m.t.m.\ space.
In the setting of (complete) metric measure spaces, such equivalence was previously known (see \cite[Section 2]{DiMar:14} or \cite[Section 7.2]{DiMarPhD}),
but the result seems to be new for non-complete metric measure spaces; see Theorem \ref{thm:H=W} below. Our proof of the inclusion
\(H^{1,p}(\mathbb X)\subseteq W^{1,p}(\mathbb X)\) follows along the lines of \cite[Section 2.1]{DiMar:14}, whereas our proof of the converse inclusion
(inspired by \cite[Theorem 3.3]{Luc:Pas:23}) relies on a new argument using tools in Convex Analysis. The latter is robust enough to be potentially useful in other contexts.
\medskip

Fix an e.m.t.m.\ space \(\mathbb X=(X,\tau,\sfd,\mm)\) and \(p\in(1,\infty)\). The differential
\(\d\colon H^{1,p}(\mathbb X)\to L^p(T^*\mathbb X)\) given by Theorem \ref{thm:cotg_mod} induces an unbounded operator
\(\d\colon L^p(\mm)\to L^p(T^*\mathbb X)\) whose domain is \(D(\d)=H^{1,p}(\mathbb X)\); see Appendix \ref{app:conv_an}.
As \(\Lip_b(X,\tau,\sfd)\) is contained in \(H^{1,p}(\mathbb X)\), and it is dense in \(L^p(\mm)\) by \eqref{eq:Lip_dense_Lp},
we deduce that \(\d\) is densely defined, thus its adjoint operator \(\d^*\colon L^p(T^*\mathbb X)'\to L^q(\mm)\)
is well posed. Letting \(\textsc{I}_{p,\mathbb X}\colon L^q(T\mathbb X)\to L^p(T^*\mathbb X)'\) be as in \eqref{eq:def_I_pX},
the operator \(\d^*\) is characterised by
\begin{equation}\label{eq:characterisation_d_star}
\int f\,\d^*V\,\d\mm=\langle V,\d f\rangle=\int\d f(\textsc{I}_{p,\mathbb X}^{-1}(V))\,\d\mm
\quad\text{ for every }f\in H^{1,p}(\mathbb X)\text{ and }V\in D(\d^*).
\end{equation}
The next result shows that each element of \(D(\d^*)\) induces a Di Marino derivation with divergence:
\begin{lemma}[Derivation induced by a vector field]\label{lem:vf_induces_der}
Let \(\mathbb X=(X,\tau,\sfd,\mm)\) be an e.m.t.m.\ space and \(q\in(1,\infty)\).
Fix any \(v\in L^q(T\mathbb X)\). Define the operator \(b_v\colon\Lip_b(X,\tau,\sfd)\to L^1(\mm)\) as
\[
b_v(f)\coloneqq\d f(v)\quad\text{ for every }f\in\Lip_b(X,\tau,\sfd).
\]
Then it holds that \(b_v\in{\rm Der}^q(\mathbb X)\) and \(|b_v|\leq|v|\). If in addition
\(V\coloneqq{\rm I}_{p,\mathbb X}(v)\in D(\d^*)\), then
\[
b_v\in{\rm Der}^q_q(\mathbb X),\qquad\div(b_v)=-\d^*V.
\]
\end{lemma}
\begin{proof}
The map \(b_v\) is linear by construction and satisfies the Leibniz rule \eqref{eq:Leibniz_Lip_der}
by \eqref{eq:Leibniz_rule_diff}, thus it is a Lipschitz derivation. Since \(|b_v(f)|\leq|v||Df|_H\leq|v|\,\lip_\sfd(f)\)
holds \(\mm\)-a.e.\ on \(X\), we deduce that \(b_v\in{\rm Der}^q(\mathbb X)\) and \(|b_v|\leq|v|\).
Now, let us assume that \(V\coloneqq{\rm I}_{p,\mathbb X}(v)\in D(\d^*)\). Then \eqref{eq:characterisation_d_star} yields
\[
\int b_v(f)\,\d\mm=\int\d f(v)\,\d\mm=\int f\,\d^*V\,\d\mm\quad\text{ for every }f\in\Lip_b(X,\tau,\sfd),
\]
whence it follows that \(b_v\in{\rm Der}^q_q(\mathbb X)\) and \(\div(b_v)=-\d^*V\). Hence, the statement is achieved.
\end{proof}

We now pass to the equivalence result between \(W^{1,p}\) and \(H^{1,p}\). We will use ultralimit techniques (see Appendix
\ref{app:ultralim}) to obtain one of the two inclusions, and tools in Convex Analysis (see Appendix \ref{app:conv_an})
to prove the other one.
\begin{theorem}[\(H^{1,p}=W^{1,p}\)]\label{thm:H=W}
Let \(\mathbb X=(X,\tau,\sfd,\mm)\) be an e.m.t.m.\ space and \(p\in(1,\infty)\). Then
\[
H^{1,p}(\mathbb X)=W^{1,p}(\mathbb X)
\]
and it holds that \(|Df|=|Df|_H\) for every \(f\in W^{1,p}(\mathbb X)\).
\end{theorem}
\begin{proof}
Fix a non-principal ultrafilter \(\omega\) on \(\N\). Let \(f\in H^{1,p}(\mathbb X)\) be a given function.
Take a sequence \((f_n)_n\subseteq\Lip_b(X,\tau,\sfd)\) such that \(f_n\to f\) and \(\lip_\sfd(f_n)\to|Df|_H\)
strongly in \(L^p(\mm)\). Up to passing to a non-relabelled subsequence, we can also assume that there exists a function
\(h\in L^p(\mm)^+\) such that \(\lip_\sfd(f_n)\leq h\) holds \(\mm\)-a.e.\ for every \(n\in\N\). In particular,
\(|b(f_n)|\leq|b|h\in L^1(\mm)\) holds for every \(b\in{\rm Der}^q_q(\mathbb X)\) and \(n\in\N\). Therefore,
by virtue of Lemma \ref{lem:lsc_ultralim} the following map is well defined:
\[
L_f(b)\coloneqq\omega\text{-}\lim_n b(f_n)\in L^1(\mm)\quad\text{ for every }b\in{\rm Der}^q_q(\mathbb X),
\]
where the ultralimit is intended with respect to the weak topology of \(L^1(\mm)\). Moreover:
\begin{itemize}
\item Fix \(\lambda_1,\lambda_2\in\R\) and \(b_1,b_2\in{\rm Der}^q_q(\mathbb X)\). Since
\(L^1(\mm)\times L^1(\mm)\ni(g_1,g_2)\mapsto\lambda_1 g_1+\lambda_2 g_2\in L^1(\mm)\) is continuous if the
domain is endowed with the product of the weak topologies and the codomain with the weak topology, by applying
Lemma \ref{lem:cont_ultralim} we obtain that
\[\begin{split}
L_f(\lambda_1 b_1+\lambda_2 b_2)&=\omega\text{-}\lim_n\big(\lambda_1\,b_1(f_n)+\lambda_2\,b_2(f_n)\big)\\
&=\lambda_1\big(\omega\text{-}\lim_n b_1(f_n)\big)+\lambda_2\big(\omega\text{-}\lim_n b_2(f_n)\big)
=\lambda_1 L_f(b_1)+\lambda_2 L_f(b_2).
\end{split}\]
This proves that \(L_f\colon{\rm Der}^q_q(\mathbb X)\to L^1(\mm)\) is a linear operator.
\item Fix \(b\in{\rm Der}^q_q(\mathbb X)\). Lemma \ref{lem:lsc_ultralim} and the weak continuity of \(L^p(\mm)\ni g\mapsto|b|g\in L^1(\mm)\) yield
\[
|L_f(b)|=\big|\omega\text{-}\lim_n b(f_n)\big|\leq\omega\text{-}\lim_n|b(f_n)|\leq\omega\text{-}\lim_n\big(|b|\,\lip_\sfd(f_n)\big)=|b||Df|_H.
\]
\item Fix \(b\in{\rm Der}^q_q(\mathbb X)\) and \(h\in\Lip_b(X,\tau,\sfd)\). Then Lemma \ref{lem:cont_ultralim} implies that
\[
L_f(hb)=\omega\text{-}\lim_n\big(h\,b(f_n)\big)=h\big(\omega\text{-}\lim_n b(f_n)\big)=h\,L_f(b).
\]
\item Since \(L^1(\mm)\ni g\mapsto\int g\,\d\mm\in\R\) is weakly continuous and \(L^p(\mm)\ni\tilde f\mapsto\int\tilde f\,{\rm div}(b)\,\d\mm\in\R\)
is strongly continuous for every \(b\in{\rm Der}^q_q(\mathbb X)\), by applying Lemma \ref{lem:cont_ultralim} we obtain that
\[
\int L_f(b)\,\d\mm=\omega\text{-}\lim_n\int b(f_n)\,\d\mm=-\omega\text{-}\lim_n\int f_n{\rm div}(b)\,\d\mm=-\int f\,{\rm div}(b)\,\d\mm.
\]
\end{itemize}
All in all, we showed that \(L_f\) verifies the conditions of Definition \ref{def:Sobolev_space_via_der}
and that \(|L_f(b)|\leq|Df|_H|b|\) holds for every \(b\in{\rm Der}^q_q(\mathbb X)\). Consequently, we can conclude
that \(f\in W^{1,p}(\mathbb X)\) and \(|Df|\leq|Df|_H\).

Conversely, let \(f\in W^{1,p}(\mathbb X)\) be given. Since \(\mathcal E_p\) is convex and \(L^p(\mm)\)-lower semicontinuous,
we have that \(\mathcal E_p=\mathcal E_p^{**}\) by the Fenchel--Moreau theorem. Note also that
\(\mathcal E_p=\frac{1}{p}\|\cdot\|_{L^p(T^*\mathbb X)}^p\circ\d\). Therefore, by applying Theorem \ref{thm:Fenchel_comp},
\eqref{eq:conj_norm}, Lemma \ref{lem:vf_induces_der} and Young's inequality, we obtain that
\[\begin{split}
\mathcal E_p(f)&=\mathcal E_p^{**}(f)=\sup_{g\in L^q(\mm)}\bigg(\int gf\,\d\mm-\mathcal E_p^*(g)\bigg)
=\sup_{g\in L^q(\mm)}\bigg(\int gf\,\d\mm-\bigg(\frac{1}{p}\|\cdot\|_{L^p(T^*\mathbb X)}^p\circ\d\bigg)^*(g)\bigg)\\
&=\sup_{g\in L^q(\mm)}\bigg(\int gf\,\d\mm-\inf\bigg\{\frac{1}{q}\|V\|_{L^p(T^*\mathbb X)'}^q\;\bigg|
\;V\in D(\d^*),\,\d^*V=g\bigg\}\bigg)\\
&\leq\sup_{g\in L^q(\mm)}\bigg(\int gf\,\d\mm-\inf\bigg\{\frac{1}{q}\|b\|_{{\rm Der}^q(\mathbb X)}^q\;\bigg|
\;b\in{\rm Der}^q_q(\mathbb X),\,-\div(b)=g\bigg\}\bigg)\\
&=\sup_{b\in{\rm Der}^q_q(\mathbb X)}\bigg(-\int f\,\div(b)\,\d\mm-\frac{1}{q}\|b\|_{{\rm Der}^q(\mathbb X)}^q\bigg)
=\sup_{b\in{\rm Der}^q_q(\mathbb X)}\int L_f(b)-\frac{1}{q}|b|^q\,\d\mm\\
&\leq\sup_{b\in{\rm Der}^q_q(\mathbb X)}\int|Df||b|-\frac{1}{q}|b|^q\,\d\mm
\leq\frac{1}{p}\int|Df|^p\,\d\mm<+\infty.
\end{split}\]
It follows that \(f\in H^{1,p}(\mathbb X)\) and \(\int|Df|_H^p\,\d\mm=p\,\mathcal E_p(f)\leq\int|Df|^p\,\d\mm\).
Since we also know from the first part of the proof that \(|Df|\leq|Df|_H\), we can finally conclude that
\(W^{1,p}(\mathbb X)=H^{1,p}(\mathbb X)\) and \(|Df|_H=|Df|\) for every \(f\in W^{1,p}(\mathbb X)\), thus proving the statement.
\end{proof}
\begin{example}[Derivations on abstract Wiener spaces]\label{ex:der_Wiener}{\rm
Let \(\mathbb X_\gamma\coloneqq(X,\tau,\sfd,\gamma)\) be the e.m.t.m.\ space obtained by equipping an abstract Wiener space \((X,\gamma)\) with
the norm topology \(\tau\) of \(X\) and with the extended distance \(\sfd\) induced by its Cameron--Martin space; see Section \ref{s:examples_emtms}.
We claim that the space \(\mathbb X_\gamma\) is `purely non-\(\sfd\)-separable', meaning that
\[
\gamma({\rm S}_{\mathbb X_\gamma})=0.
\]
To prove it, we denote by \((H(\gamma),|\cdot|_{H(\gamma)})\) the Cameron--Martin space of \((X,\gamma)\). We recall that
\[
\sfd(x,y)=\left\{\begin{array}{ll}
|x-y|_{H(\gamma)}\\
+\infty
\end{array}\quad\begin{array}{ll}
\text{ if }x,y\in X\text{ and }x-y\in H(\gamma),\\
\text{ if }x,y\in X\text{ and }x-y\notin H(\gamma),
\end{array}\right.
\]
and that \(\gamma(x+H(\gamma))=0\) for every \(x\in X\); see \cite{Bogachev15}. Hence, if \(E\in\mathscr B(X,\tau)\) is a given \(\sfd\)-separable subset of \(X\)
and \((x_n)_n\) is a \(\sfd\)-dense sequence in \(E\), then \(E\subseteq\bigcup_{n\in\N}B^\sfd_1(x_n)\subseteq\bigcup_{n\in\N}(x_n+H(\gamma))\) and thus accordingly
\(\gamma(E)\leq\sum_{n\in\N}\gamma(x_n+H(\gamma))=0\), whence it follows that \(\gamma({\rm S}_{\mathbb X_\gamma})=0\).

By taking Corollary \ref{cor:w*_cont_trivial} i) into account, we deduce that the unique weakly\(^*\)-type continuous derivation on \(\mathbb X_\gamma\)
is the null derivation. On the other hand, we know from \cite[Example 5.3.14]{Savare22} that \(H^{1,p}(\mathbb X_\gamma)\) coincides with the usual Sobolev
space on \(\mathbb X_\gamma\) defined as the completion of \emph{cylindrical functions} \cite{Bogachev15}. In particular, the identity
\(W^{1,p}(\mathbb X_\gamma)=H^{1,p}(\mathbb X_\gamma)\) we proved in Theorem \ref{thm:H=W} guarantees the existence of (many) non-null Di Marino derivations
with divergence, and thus (by Lemma \ref{lem:suff_cond_wstar_seq_cont}) of non-null weakly\(^*\)-type sequentially continuous derivations.
\fr}\end{example}
\subsection{The equivalence \texorpdfstring{\(W^{1,p}=B^{1,p}\)}{W=B}}\label{s:W=B}
In this section, we investigate the relation between the spaces \(W^{1,p}(\mathbb X)\) and \(B^{1,p}(\mathbb X)\).
By combining Theorem \ref{thm:H=W} with Theorem \ref{thm:H=B}, we see that a sufficient condition for the identity \(W^{1,p}(\mathbb X)=B^{1,p}(\mathbb X)\)
to hold is the completeness of the extended metric space \((X,\sfd)\):
\begin{corollary}[\(W^{1,p}=B^{1,p}\) on complete e.m.t.m.\ spaces]\label{cor:W=B}
Let \(\mathbb X=(X,\tau,\sfd,\mm)\) be an e.m.t.m.\ space such that \((X,\sfd)\) is a complete extended metric space.
Let \(p\in(1,\infty)\) be given. Then
\[
W^{1,p}(\mathbb X)=B^{1,p}(\mathbb X).
\]
Moreover, it holds that \(|Df|_B=|Df|\) for every \(f\in W^{1,p}(\mathbb X)\).
\end{corollary}
\begin{remark}[Relation with the Newtonian space \(N^{1,p}\)]\label{rmk:Newt}{\rm
The \textbf{Newtonian space} \(N^{1,p}(\mathbb X)\) over an e.m.t.m.\ space \(\mathbb X\) has been introduced by
Savar\'{e} in \cite[Definition 5.1.19]{Savare22}, thus generalising the notion of Newtonian space for metric measure
spaces introduced by Shanmugalingam in \cite{Shanmugalingam00}. It follows from Corollary \ref{cor:W=B} and
\cite[Corollary 5.1.26]{Savare22} that if \(\mathbb X=(X,\tau,\sfd,\mm)\) is an e.m.t.m.\ space such that \((X,\sfd)\)
is complete and \((X,\tau)\) is a \emph{Souslin space} (i.e.\ the continuous image of a complete separable metric
space), then the Sobolev space \(W^{1,p}(\mathbb X)\) is fully consistent with \(N^{1,p}(\mathbb X)\).
\fr}\end{remark}

On an arbitrary e.m.t.m.\ space \(\mathbb X\), it can happen that the spaces \(W^{1,p}(\mathbb X)\) and \(B^{1,p}(\mathbb X)\) are different,
as the example we discussed in the last paragraph of Section \ref{s:def_B} shows. Nevertheless, we are going to show that every \(\mathcal T_q\)-test plan
\(\ppi\) on \(\mathbb X\) induces a Di Marino derivation with divergence \(b_\sppi\in\Der^q_q(\mathbb X)\) (Proposition \ref{prop:der_induced_by_tp}),
and as a corollary we will prove that \(W^{1,p}(\mathbb X)\) is always contained in \(B^{1,p}(\mathbb X)\) and that \(|Df|_B\leq|Df|\) for every
\(f\in W^{1,p}(\mathbb X)\) (Theorem \ref{thm:W_in_B}).
\medskip

For brevity, we denote by \(\mathscr L_1\) the restriction of the \(1\)-dimensional Lebesgue measure \(\mathscr L^1\)
to the unit interval \([0,1]\subseteq\R\). To any given \(\mathcal T_q\)-test plan \(\ppi\in\mathcal T_q(\mathbb X)\), we associate the product measure
\[
\hat\ppi\coloneqq\ppi\otimes\mathscr L_1\in\mathcal M_+({\rm RA}(X,\sfd)\times[0,1]),
\]
where the space \({\rm RA}(X,\sfd)\times[0,1]\) is endowed with the product topology.
\medskip

The next result is inspired by (and generalises) \cite[Proposition 2.4]{DiMar:14} and \cite[Proposition 4.10]{AILP24}.
\begin{proposition}[Derivation induced by a \(\mathcal T_q\)-test plan]\label{prop:der_induced_by_tp}
Let \(\mathbb X=(X,\tau,\sfd,\mm)\) be an e.m.t.m.\ space and \(q\in(1,\infty)\).
Let \(\ppi\in\mathcal T_q(\mathbb X)\) be given. Then for any \(f\in\Lip_b(X,\tau,\sfd)\) we have that
\[
\hat{\sf e}_\#({\rm D}_f^+\hat\ppi),\hat{\sf e}_\#({\rm D}_f^-\hat\ppi)\ll\mm,\qquad b_\sppi(f)\coloneqq\frac{\d\hat{\sf e}_\#({\rm D}_f^+\hat\ppi)}{\d\mm}-\frac{\d\hat{\sf e}_\#({\rm D}_f^-\hat\ppi)}{\d\mm}\in L^q(\mm),
\]
where \(\hat{\sf e}\) denotes the arc-length evaluation map \eqref{eq:def_e}, while \({\rm D}_f^+\) and
\({\rm D}_f^-\) denote the positive and the negative parts, respectively, of the function \({\rm D}_f\) defined
in Lemma \ref{cor:D_f_univ_Lusin_meas}. Moreover, the resulting map \(b_\sppi\colon\Lip_b(X,\tau,\sfd)\to L^q(\mm)\)
belongs to \({\rm Der}^q_q(\mathbb X)\) and it holds that
\begin{equation}\label{eq:formulas_b_sppi}
|b_\sppi|\leq h_\sppi,\qquad\div(b_\sppi)=\frac{\d(\hat{\sf e}_0)_\#\ppi}{\d\mm}-\frac{\d(\hat{\sf e}_1)_\#\ppi}{\d\mm}.
\end{equation}
\end{proposition}
\begin{proof}
First of all, observe that \({\rm D}_f^\pm\hat\ppi\) are Radon measures because \({\rm D}_f^\pm\) is Borel \(\hat\ppi\)-measurable
(by Corollary \ref{cor:D_f_univ_Lusin_meas}) and \(\hat\ppi\) is a Radon measure. Since \(\hat{\sf e}\) is universally
Lusin measurable by Lemma \ref{lem:hat_e_univ_Lusin_meas}, we have that \(\hat{\sf e}_\#({\rm D}_f^\pm\hat\ppi)\in\mathcal M_+(X)\).
Given any \(f,g\in\Lip_b(X,\tau,\sfd)\) with \(g\geq 0\), we can estimate
\[\begin{split}
\int g\,\d\hat{\sf e}_\#({\rm D}_f^\pm\hat\ppi)&=\int\!\!\!\int_0^1 g(R_\gamma(t)){\rm D}_f^\pm(\gamma,t)\,\d t\,\d\ppi(\gamma)
\overset{\eqref{eq:ineq_D_f}}\leq\int\!\!\!\int_0^1\ell(\gamma)(g\,\lip_\sfd(f))(R_\gamma(t))\,\d t\,\d\ppi(\gamma)\\
&=\int\bigg(\int_\gamma g\,\lip_\sfd(f)\bigg)\,\d\ppi(\gamma)=\int g\,\lip_\sfd(f)\,\d\mu_\sppi=\int g\,\lip_\sfd(f) h_\sppi\,\d\mm.
\end{split}\]
By the arbitrariness of \(g\), we deduce that \(\hat{\sf e}_\#({\rm D}_f^\pm\hat\ppi)\ll\mm\) and that
\(b_\sppi(f)\coloneqq\frac{\d\hat{\sf e}_\#({\rm D}_f^+\hat\ppi)}{\d\mm}-\frac{\d\hat{\sf e}_\#({\rm D}_f^-\hat\ppi)}{\d\mm}\)
satisfies \(|b_\sppi(f)|\leq 2\,\lip_\sfd(f)h_\sppi\), so that \(b_\sppi(f)\in L^q(\mm)\). By \eqref{eq:prop_D_f},
for every \(f,g\in\Lip_b(X,\tau,\sfd)\), \(\alpha,\beta\in\R\) and \(\gamma\in{\rm RA}(X,\sfd)\) we have that
\[\begin{split}
{\rm D}_{\alpha f+\beta g}(\gamma,t)&=\alpha\,{\rm D}_f(\gamma,t)+\beta\,{\rm D}_g(\gamma,t),\\
{\rm D}_{fg}(\gamma,t)&={\rm D}_f(\gamma,t)g(R_\gamma(t))+{\rm D}_g(\gamma,t)f(R_\gamma(t))
\end{split}\]
hold for \(\mathscr L_1\)-a.e.\ \(t\in[0,1]\). In particular, \({\rm D}_{\alpha f+\beta g}=\alpha\,{\rm D}_f+\beta\,{\rm D}_g\) and
\({\rm D}_{fg}=g\circ\hat{\sf e}\,{\rm D}_f+f\circ\hat{\sf e}\,{\rm D}_g\) are verified in the \(\hat\ppi\)-a.e.\ sense. It follows that
\[\begin{split}
b_\sppi(\alpha f+\beta g)&=\frac{\d\hat{\sf e}_\#((\alpha\,{\rm D}_f+\beta\,{\rm D}_g)\hat\ppi)}{\d\mm}
=\alpha\frac{\d\hat{\sf e}_\#({\rm D}_f\hat\ppi)}{\d\mm}+\beta\frac{\d\hat{\sf e}_\#({\rm D}_g\hat\ppi)}{\d\mm}=\alpha\,b_\sppi(f)+\beta\,b_\sppi(g),\\
b_\sppi(fg)&=\frac{\d\hat{\sf e}_\#((g\circ\hat{\sf e}\,{\rm D}_f)\hat\ppi)}{\d\mm}+\frac{\d\hat{\sf e}_\#((f\circ\hat{\sf e}\,{\rm D}_g)\hat\ppi)}{\d\mm}
=\frac{\d(g\,\hat{\sf e}_\#({\rm D}_f\hat\ppi))}{\d\mm}+\frac{\d(f\,\hat{\sf e}_\#({\rm D}_g\hat\ppi))}{\d\mm}\\
&=b_\sppi(f)g+b_\sppi(g)f.
\end{split}\]
Hence, \(b_\sppi\colon\Lip_b(X,\tau,\sfd)\to L^q(\mm)\) is a linear operator satisfying the Leibniz rule, thus it is a Lipschitz derivation on \(\mathbb X\).
Given any \(f,g\in\Lip_b(X,\tau,\sfd)\) with \(g\geq 0\), we can now estimate
\[\begin{split}
\bigg|\int g\,b_\sppi(f)\,\d\mm\bigg|
&\overset{\phantom{\eqref{eq:ineq_D_f}]}}=\bigg|\int\!\!\!\int_0^1 g(R_\gamma(t)){\rm D}_f(\gamma,t)\,\d t\,\d\ppi(\gamma)\bigg|
\leq\int\!\!\!\int_0^1 g(R_\gamma(t))|{\rm D}_f(\gamma,t)|\,\d t\,\d\ppi(\gamma)\\
&\overset{\eqref{eq:ineq_D_f}}\leq\int\!\!\!\int_0^1\ell(\gamma)(g\,\lip_\sfd(f))(R_\gamma(t))\,\d t\,\d\ppi(\gamma)
=\int g\,\lip_\sfd(f) h_\sppi\,\d\mm,
\end{split}\]
so that \(|b_\sppi(f)|\leq\lip_\sfd(f)h_\sppi\) for every \(f\in\Lip_b(X,\tau,\sfd)\). Therefore, \(b_\sppi\in{\rm Der}^q(\mathbb X)\) and \(|b_\sppi|\leq h_\sppi\).
Moreover, for any \(f\in\Lip_b(X,\tau,\sfd)\) we can compute
\[\begin{split}
\int b_\sppi(f)\,\d\mm&=\int{\rm D}_f\,\d\hat\ppi\overset{\eqref{eq:prop_D_f}}=\int\!\!\!\int_0^1(f\circ R_\gamma)'(t)\,\d t\,\d\ppi(\gamma)
=\int f(\gamma_1)-f(\gamma_0)\,\d\ppi(\gamma)\\
&=-\int f\bigg(\frac{\d(\hat{\sf e}_0)_\#\ppi}{\d\mm}-\frac{\d(\hat{\sf e}_1)_\#\ppi}{\d\mm}\bigg)\,\d\mm,
\end{split}\]
which shows that \(b_\sppi\in{\rm Der}^q_q(\mathbb X)\) and \(\div(b_\sppi)=\frac{\d(\hat{\sf e}_0)_\#\ppi}{\d\mm}-\frac{\d(\hat{\sf e}_1)_\#\ppi}{\d\mm}\). The proof is complete.
\end{proof}

As a consequence of Proposition \ref{prop:der_induced_by_tp}, the space \(W^{1,p}(\mathbb X)\) is always contained in \(B^{1,p}(\mathbb X)\):
\begin{theorem}[\(W^{1,p}\subseteq B^{1,p}\)]\label{thm:W_in_B}
Let \(\mathbb X=(X,\tau,\sfd,\mm)\) be an e.m.t.m.\ space and \(p\in(1,\infty)\). Then
\[
W^{1,p}(\mathbb X)\subseteq B^{1,p}(\mathbb X).
\]
Moreover, it holds that \(|Df|_B\leq|Df|\) for every \(f\in W^{1,p}(\mathbb X)\).
\end{theorem}
\begin{proof}
Let \(f\in W^{1,p}(\mathbb X)\) be given. Fix some \(\tau\)-Borel representative \(G_f\colon X\to[0,+\infty)\) of \(|Df|\). For any \(\ppi\in\mathcal T_q(\mathbb X)\) (where \(q\in(1,\infty)\)
denotes the conjugate exponent of \(p\)), the derivation \(b_\sppi\in{\rm Der}^q_q(\mathbb X)\) given by Proposition \ref{prop:der_induced_by_tp} satisfies
\[
\int f(\gamma_1)-f(\gamma_0)\,\d\ppi(\gamma)\overset{\eqref{eq:formulas_b_sppi}}=-\int f\,\div(b_\sppi)\,\d\mm=\int L_f(b_\sppi)\,\d\mm\leq\int|Df||b_\sppi|\,\d\mm
\overset{\eqref{eq:formulas_b_sppi}}\leq\int G_f\,h_\sppi\,\d\mm.
\]
By virtue of Lemma \ref{lem:integral_def_B1p}, we deduce that \(G_f\) is a \(\mathcal T_q\)-weak upper gradient of \(f\). Therefore, we proved that
\(f\in B^{1,p}(\mathbb X)\) and \(|Df|_B\leq|Df|\), whence the statement follows.
\end{proof}
\subsection{\texorpdfstring{\(W^{1,p}\)}{W1p} as a dual space}\label{s:predual_W1p}
In this section, our aim is to provide a new description of some \emph{isometric predual} of the metric Sobolev space, and
the formulation of Sobolev space in terms of derivations serves this purpose very well. More precisely, in Theorem \ref{thm:predual_W1p}
we give an explicit construction of a Banach space whose dual is isometrically isomorphic to \(W^{1,p}(\mathbb X)\). The existence
and the construction of an isometric predual of the space \(H^{1,p}(\mathbb X)\) were previously obtained by Ambrosio and Savar\'{e}
in \cite[Corollary 3.10]{Amb:Sav:21}.
\medskip

In the proof of Theorem \ref{thm:predual_W1p}, we use some facts in Functional Analysis that we collect below:
\begin{itemize}
\item If \(\mathbb B\), \(\mathbb V\) are Banach spaces and \(q\in(1,\infty)\), the product vector space \(\mathbb B\times\mathbb V\)
is a Banach space if endowed with the \emph{\(q\)-norm}
\[
\|(v,w)\|_q\coloneqq\big(\|v\|_{\mathbb B}^q+\|w\|_{\mathbb V}^q\big)^{1/q}\quad\text{ for every }(v,w)\in\mathbb B\times\mathbb V. 
\]
We write \(\mathbb B\times_q\mathbb V\) to indicate the Banach space \((\mathbb B\times\mathbb V,\|\cdot\|_q)\).
\item If \(p,q\in(1,\infty)\) are conjugate exponents, then \((\mathbb B\times_q\mathbb V)'\) and \(\mathbb B'\times_p\mathbb V'\)
are isometrically isomorphic. The canonical duality pairing between \(\mathbb B'\times_p\mathbb V'\) and \(\mathbb B\times_q\mathbb V\) is given by
\[
\langle(\omega,\eta),(v,w)\rangle=\langle\omega,v\rangle+\langle\eta,w\rangle\quad\text{ for every }(\omega,\eta)\in\mathbb B'\times\mathbb V'
\text{ and }(v,w)\in\mathbb B\times\mathbb V.
\]
\item The \textbf{annihilator} \(\mathbb W^\perp\) of a closed vector subspace \(\mathbb W\) of \(\mathbb B\) is defined as
\[
\mathbb W^\perp\coloneqq\big\{\omega\in\mathbb B'\;\big|\;\langle\omega,v\rangle=0\text{ for every }v\in\mathbb W\big\}.
\]
Then \(\mathbb W^\perp\) is a closed vector subspace of \(\mathbb B'\). Moreover, \(\mathbb W^\perp\) is isometrically isomorphic to the dual
\((\mathbb B/\mathbb W)'\) of the quotient Banach space \(\mathbb B/\mathbb W\).
\end{itemize}
\begin{theorem}[A predual of \(W^{1,p}\)]\label{thm:predual_W1p}
Let \(\mathbb X=(X,\tau,\sfd,\mm)\) be an e.m.t.m.\ space. Let \(p,q\in(1,\infty)\) be conjugate exponents.
We define the closed vector subspace \(\mathbb B_{\mathbb X,q}\) of \(L^q(\mm)\times_q L^q_\Lip(T\mathbb X)\) as the closure of its vector subspace
\[
\big\{(g,b)\in L^q(\mm)\times{\rm Der}^q_q(\mathbb X)\;\big|\;g=\div(b)\big\}.
\]
Then \(W^{1,p}(\mathbb X)\) is isometrically isomorphic to the dual of the quotient \((L^q(\mm)\times_q L^q_\Lip(T\mathbb X))/\mathbb B_{\mathbb X,q}\).
\end{theorem}
\begin{proof}
For any \(f\in W^{1,p}(\mathbb X)\), we define \(\mathfrak L_f\coloneqq\textsc{Int}_{L^q_\Lip(T\mathbb X)}(L_f)\in L^q_\Lip(T\mathbb X)'\),
so that accordingly
\begin{equation}\label{eq:predual_W1p_aux1}
\|\mathfrak L_f\|_{L^q_\Lip(T\mathbb X)'}=\|L_f\|_{L^q_\Lip(T\mathbb X)^*}=\||L_f|\|_{L^p(\mm)}=\||Df|\|_{L^p(\mm)}.
\end{equation}
Clearly, \(W^{1,p}(\mathbb X)\ni f\mapsto\mathfrak L_f\in L^q_\Lip(T\mathbb X)'\) is linear.
Define \(\phi\colon W^{1,p}(\mathbb X)\to L^p(\mm)\times_p L^q_\Lip(T\mathbb X)'\) as
\[
\phi(f)\coloneqq(f,\mathfrak L_f)\in L^p(\mm)\times L^q_\Lip(T\mathbb X)'\quad\text{ for every }f\in W^{1,p}(\mathbb X).
\]
It follows from \eqref{eq:predual_W1p_aux1} and the definition of \(\|\cdot\|_{W^{1,p}(\mathbb X)}\) that
\(\phi\) is a linear isometry. We claim that
\begin{equation}\label{eq:predual_W1p_aux2}
\phi(W^{1,p}(\mathbb X))=\mathbb B_{\mathbb X,q}^\perp,
\end{equation}
where we are identifying \(\mathbb B_{\mathbb X,q}^\perp\subseteq(L^q(\mm)\times_q L^q_\Lip(T\mathbb X))'\) with a subspace of \(L^p(\mm)\times_p L^q_\Lip(T\mathbb X)'\).
To prove \(\phi(W^{1,p}(\mathbb X))\subseteq\mathbb B_{\mathbb X,q}^\perp\), it suffices to observe that for any \(f\in W^{1,p}(\mathbb X)\) and \(b\in{\rm Der}^q_q(\mathbb X)\) it holds
\[
\langle\phi(f),(\div(b),b)\rangle=\langle f,\div(b)\rangle+\mathfrak L_f(b)=\int f\,\div(b)\,\d\mm+\int L_f(b)\,\d\mm=0.
\]
We now prove the converse inclusion \(\mathbb B_{\mathbb X,q}^\perp\subseteq\phi(W^{1,p}(\mathbb X))\). Fix
\((f,\mathfrak L)\in\mathbb B_{\mathbb X,q}^\perp\subseteq L^p(\mm)\times_p L^q_\Lip(T\mathbb X)'\). Letting
\(L\coloneqq\textsc{Int}_{L^q_\Lip(T\mathbb X)}^{-1}(\mathfrak L)\in L^q_\Lip(T\mathbb X)^*\), we have in particular
that \(L|_{\Der^q_q(\mathbb X)}\colon\Der^q_q(\mathbb X)\to L^1(\mm)\) is a linear operator satisfying \(|L(b)|\leq|L||b|\)
for every \(b\in\Der^q_q(\mathbb X)\), for some function \(|L|\in L^p(\mm)^+\) such that
\(\||L|\|_{L^p(\mm)}=\|\mathfrak L\|_{L^q_\Lip(T\mathbb X)'}\). Moreover, the \(L^\infty(\mm)\)-linearity of \(L\)
implies \(L(hb)=h\,L(b)\) for every \(h\in\Lip_b(X,\tau,\sfd)\) and \(b\in\Der^q_q(\mathbb X)\),
and using that \((\div(b),b)\in\mathbb B_{\mathbb X,q}\) we deduce that
\[
\int f\,\div(b)\,\d\mm+\int L(b)\,\d\mm=\langle f,\div(b)\rangle+\mathfrak L(b)=\langle(f,\mathfrak L),(\div(b),b)\rangle=0,
\]
so that \(\int L(b)\,\d\mm=-\int f\,\div(b)\,\d\mm\). All in all, we proved that \(f\in W^{1,p}(\mathbb X)\)
and \(L_f=L\), which gives \((f,\mathfrak L)=(f,\mathfrak L_f)=\phi(f)\in\phi(W^{1,p}(\mathbb X))\).
Consequently, the claimed identity \eqref{eq:predual_W1p_aux2} is proved.
Writing \(\cong\) to indicate that two Banach spaces are isometrically isomorphic, we then conclude that
\[
W^{1,p}(\mathbb X)\cong\phi(W^{1,p}(\mathbb X))\cong\mathbb B_{\mathbb X,q}^\perp\cong
\big((L^q(\mm)\times_q L^q_\Lip(T\mathbb X))/\mathbb B_{\mathbb X,q}\big)',
\]
proving the statement.
\end{proof}
\appendix
\section{Ultrafilters and ultralimits}\label{app:ultralim}
We collect here some definitions and results concerning ultrafilters and ultralimits, which we use in the proof of Theorem \ref{thm:H=W}.
See e.g.\ \cite{Jech78} or \cite[Chapter 10]{DrutuKapovich18} for more on these topics.
\medskip

Let \(\omega\) be a \textbf{filter} on \(\N\), i.e.\ a collection of subsets of \(\N\) that is closed under supersets and finite intersections.
Then we say that \(\omega\) is an \textbf{ultrafilter} provided it is a maximal filter with respect to inclusion, or equivalently if for any
subset \(A\subseteq\N\) we have that either \(A\in\omega\) or \(\N\setminus A\in\omega\). Moreover, we say that \(\omega\) is \textbf{non-principal}
provided it does not contain any finite subset of \(\N\). The existence of non-principal ultrafilters on \(\N\) follows e.g.\ from the so-called
\emph{Ultrafilter Lemma} \cite[Lemma 10.18]{DrutuKapovich18}, which is (in ZF) strictly weaker than the Axiom of Choice \cite{Tarski30,Halpern64}.
It holds that an ultrafilter \(\omega\) on \(\N\) is non-principal if and only if it contains the \emph{Fr\'{e}chet filter} (i.e.\ the
collection of all cofinite subsets of \(\N\)).
\medskip

Let \(\omega\) be a non-principal ultrafilter on \(\N\), \((X,\tau)\) a Hausdorff topological space and \((x_n)_{n\in\N}\subseteq X\)
a given sequence. Then we say that an element \(\omega\text{-}\lim_n x_n\in X\) is the \textbf{ultralimit} of \((x_n)_n\) provided
\[
\{n\in\N\;|\;x_n\in U\}\in\omega\quad\text{ for every }U\in\tau\text{ with }\omega\text{-}\lim_n x_n\in U.
\]
The Hausdorff assumption on \(\tau\) ensures that if the ultralimit exists, then it is unique. The existence of the ultralimits of all sequences
in \((X,\tau)\) is guaranteed when the topology \(\tau\) is compact.
\medskip

We now discuss technical results about ultralimits, which we prove for the reader's convenience.
\begin{lemma}\label{lem:cont_ultralim}
Let \(\omega\) be a non-principal ultrafilter on \(\N\). Let \(X_1,\ldots,X_k,Y\) be Hausdorff topological spaces,
for some \(k\in\N\) with \(k\geq 1\). Let \(\varphi\colon X_1\times\ldots\times X_k\to Y\) be a continuous map,
where the domain \(X_1\times\ldots\times X_k\) is endowed with the product topology. For any \(i=1,\ldots,k\),
let \((x_i^n)_{n\in\N}\subseteq X_i\) be a sequence whose ultralimit \(x_i\coloneqq\omega\text{-}\lim_n x_i^n\in X_i\)
exists. Then it holds that
\begin{equation}\label{eq:cont_ultralim}
\exists\,\omega\text{-}\lim_n\varphi(x_1^n,\ldots,x_k^n)=\varphi(x_1,\ldots,x_k)\in Y.
\end{equation}
\end{lemma}
\begin{proof}
Fix a neighbourhood \(U\) of \(\varphi(x_1,\ldots,x_k)\) in \(Y\). Since \(\varphi\) is continuous, \(\varphi^{-1}(U)\) is a neighbourhood of \((x_1,\ldots,x_k)\).
Thus, for any \(i=1,\ldots,k\) there exists a neighbourhood \(U_i\) of \(x_i\) in \(X_i\) such that \(U_1\times\ldots\times U_k\subseteq\varphi^{-1}(U)\). Recalling
that \(x_i=\omega\text{-}\lim_n x_i^n\) for all \(i=1,\ldots,k\), we get that
\[
\omega\ni\bigcap_{i=1}^k\{n\in\N\;|\;x_i^n\in U_i\}\subseteq\big\{n\in\N\;\big|\;\varphi(x_1^n,\ldots,x_k^n)\in U\big\}
\]
and thus \(\big\{n\in\N\;\big|\;\varphi(x_1^n,\ldots,x_k^n)\in U\big\}\in\omega\). Thanks to the arbitrariness of \(U\), \eqref{eq:cont_ultralim} is proved.
\end{proof}
\begin{remark}\label{rmk:conseq_Dunford-Pettis}{\rm
Let \((X,\Sigma,\mm)\) be a finite measure space. Let \(h\in L^1(\mm)^+\) be given. Then
\begin{equation}\label{eq:conseq_Dunford-Pettis_cpt}
\mathcal F_h\coloneqq\big\{f\in L^1(\mm)\;\big|\;|f|\leq h\big\}\;\text{ is a weakly compact subset of }L^1(\mm).
\end{equation}
The validity of this property follows from the Dunford--Pettis theorem and the fact that
\(\mathcal F_h\) is a weakly closed subset of \(L^1(\mm)\).
\fr}\end{remark}
\begin{lemma}\label{lem:lsc_ultralim}
Let \(\omega\) be a non-principal ultrafilter on \(\N\). Let \((X,\Sigma,\mm)\) be a finite measure space.
Assume that \((f_n)_n\subseteq L^1(\mm)\) and \(h\in L^1(\mm)^+\) satisfy \(|f_n|\leq h\) for every \(n\in\N\).
Then the weak ultralimits \(f\coloneqq\omega\text{-}\lim_n f_n\in L^1(\mm)\) and \(\omega\text{-}\lim_n|f_n|\in L^1(\mm)\)
exist. Moreover, it holds that
\begin{equation}\label{eq:lsc_ultralim}
|f|\leq\omega\text{-}\lim_n|f_n|\leq h.
\end{equation}
\end{lemma}
\begin{proof}
The existence of the ultralimits \(\omega\text{-}\lim_n f_n\) and \(\omega\text{-}\lim_n|f_n|\) in the weak topology of \(L^1(\mm)\) follows
from Remark \ref{rmk:conseq_Dunford-Pettis}. For any \(g\in L^\infty(\mm)^+\), we consider the functional \(\varphi_g\colon L^1(\mm)\to\R\)
given by \(\varphi_g(\tilde f)\coloneqq\int\tilde f g\,\d\mm\) for every \(\tilde f\in L^1(\mm)\), which is weakly continuous. Hence, Lemma
\ref{lem:cont_ultralim} yields
\[
\bigg|\int fg\,\d\mm\bigg|=|\varphi_g(f)|=\big|\omega\text{-}\lim_n\varphi_g(f_n)\big|=\omega\text{-}\lim_n|\varphi_g(f_n)|
\leq\omega\text{-}\lim_n\varphi_g(|f_n|)=\varphi_g\big(\omega\text{-}\lim_n|f_n|\big)
\]
and \(\int(\omega\text{-}\lim_n|f_n|)g\,\d\mm=\omega\text{-}\lim_n\varphi_g(|f_n|)\leq\varphi_g(h)=\int hg\,\d\mm\),
whence the claimed inequalities in \eqref{eq:lsc_ultralim} follow thanks to the arbitrariness of \(g\in L^\infty(\mm)^+\).
\end{proof}
\section{Tools in Convex Analysis}\label{app:conv_an}
Let \(\mathbb B\), \(\mathbb V\) be Banach spaces. Then by an \textbf{unbounded operator} \(A\colon\mathbb B\to\mathbb V\) we mean a vector subspace
\(D(A)\) of \(\mathbb B\) (called the \textbf{domain} of \(A\)) together with a linear operator \(A\colon D(A)\to\mathbb V\). When \(A\) if \textbf{densely defined}
(i.e.\ the set \(D(A)\) is dense in \(\mathbb B\)), it is possible to define its \textbf{adjoint} operator
\(A^*\colon\mathbb V'\to\mathbb B'\), which is characterised by
\[\begin{split}
D(A^*)&\coloneqq\big\{\eta\in\mathbb V'\;\big|\;\mathbb B\ni v\mapsto\langle\eta,A(v)\rangle\in\R\text{ is continuous}\big\},\\
\langle\eta,A(v)\rangle&=\langle A^*(\eta),v\rangle\quad\text{ for every }\eta\in D(A^*)\text{ and }v\in D(A).
\end{split}\]
See e.g.\ \cite[Chapter 5]{pedersen1989analysis} for more on unbounded operators.
\medskip

Given any function \(f\colon\mathbb B\to[-\infty,+\infty]\), we denote by \(f^*\colon\mathbb B'\to[-\infty,+\infty]\) its \textbf{Fenchel conjugate}, which is defined as
\[
f^*(\omega)\coloneqq\sup\big\{\langle\omega,v\rangle-f(v)\;\big|\;v\in\mathbb B\big\}\quad\text{ for every }\omega\in\mathbb B'.
\]
Assuming \(\mathbb B\) is reflexive, we have (unless the function \(f\) is identically equal to \(+\infty\) or identically equal to \(-\infty\))
that the \textbf{Fenchel biconjugate} \(f^{**}\coloneqq(f^*)^*\colon\mathbb B\to[-\infty,+\infty]\) coincides with \(f\) if and only if \(f\) is convex and lower semicontinuous.
This follows from the \emph{Fenchel--Moreau theorem}. Furthermore, if \(p,q\in(1,\infty)\) are conjugate exponents, then it is straightforward to check that
\begin{equation}\label{eq:conj_norm}
\bigg(\frac{1}{p}\|\cdot\|_\B^p\bigg)^*=\frac{1}{q}\|\cdot\|_{\B'}^q.
\end{equation}
See e.g.\ \cite{Rockafellar74} for a thorough discussion on Fenchel conjugates.
\medskip

In Theorem \ref{thm:H=W} we use the following result, for whose proof we refer to \cite[Theorem 5.1]{BBS97}.
\begin{theorem}\label{thm:Fenchel_comp}
Let \(\B\) and \(\V\) be Banach spaces. Let \(A\colon\B\to\V\) be a densely-defined unbounded operator.
Let \(\phi\colon\V\to\R\) be a convex function that is continuous at some point of \(A(D(A))\). Then
\[
(\phi\circ A)^*(\omega)=\inf\big\{\phi^*(\eta)\;\big|\;\eta\in D(A^*),\,A^*(\eta)=\omega\big\}\quad\text{ for every }\omega\in\B',
\]
where we adopt the convention that \((\phi\circ A)(v)\coloneqq+\infty\) for every \(v\in\B\setminus D(A)\).
\end{theorem}
%
%
%
%
%
%
\def\cprime{$'$} \def\cprime{$'$}

\end{document}